\documentclass[11pt,reqno]{amsproc}
\usepackage{amsmath,amsfonts,amssymb,amsthm}
\usepackage[abbrev,lite,nobysame]{amsrefs}
\usepackage{mathrsfs,esint}
 \usepackage[usenames,dvipsnames]{color}
\usepackage[usenames,dvipsnames]{color}
\usepackage[dvipsnames]{xcolor}
\usepackage{tikz}

\usepackage[margin=1in]{geometry}
\usepackage{dsfont}
\usepackage{enumerate,enumitem}
\usepackage{comment}
\usepackage{bbm}
\usepackage{graphicx}

\usepackage{tikz}
\usetikzlibrary{calc}

\usepackage[colorlinks=true, pdfstartview=FitV, linkcolor=BrickRed,citecolor=black, urlcolor=black]{hyperref}

\usepackage{mathtools}
\mathtoolsset{showonlyrefs}

\newtheorem{MAINtheorem}{Theorem}

\newtheorem{proposition}{Proposition}[section]
\newtheorem{lemma}[proposition]{Lemma}
\newtheorem{corollary}[proposition]{Corollary}
\newtheorem{theorem}[proposition]{Theorem}

\theoremstyle{definition}
\newtheorem{definition}[proposition]{Definition}

\newtheorem{remark}[proposition]{Remark}

\numberwithin{equation}{section}

 \makeatletter

\renewcommand\subsubsection{\@startsection{subsubsection}{3}%
\normalparindent{.5\linespacing\@plus.7\linespacing}{-.5em}
{\normalfont\bfseries}}

\def\@tocline#1#2#3#4#5#6#7{\relax
  \ifnum #1>\c@tocdepth % then omit
  \else
    \par \addpenalty\@secpenalty\addvspace{#2}%
    \begingroup \hyphenpenalty\@M
    \@ifempty{#4}{%
      \@tempdima\csname r@tocindent\number#1\endcsname\relax
    }{%
      \@tempdima#4\relax
    }%
    \parindent\z@ \leftskip#3\relax \advance\leftskip\@tempdima\relax
    \rightskip\@pnumwidth plus4em \parfillskip-\@pnumwidth
    #5\leavevmode\hskip-\@tempdima
      \ifcase #1
       \or\or \hskip 1em \or \hskip 2em \else \hskip 3em \fi%
      #6\nobreak\relax
    \dotfill\hbox to\@pnumwidth{\@tocpagenum{#7}}\par
    \nobreak
    \endgroup
  \fi}
\makeatother

\newcommand\eps{\varepsilon}
\newcommand\e{{\rm e}}
\newcommand\dd{{\rm d}}

\newcommand\Id{{\rm Id}}

\newcommand{\dist}{{\rm dist}}
\newcommand{\diver}{\text{div}}
\newcommand{\RR}{\mathbb{R}}
\newcommand{\NN}{\mathbb{N}}
\newcommand{\CC}{\mathbb{C}}
\newcommand{\TT}{\mathbb{T}}
\newcommand{\ZZ}{\mathbb{Z}}

\newcommand{\cL}{\mathcal{L}}
\newcommand{\cP}{\mathcal{P}}

\newcommand{\cM}{\mathcal{M}}
\newcommand{\cQ}{\mathcal{Q}}

\newcommand{\cG}{\mathcal{G}}

\newcommand{\ini}{{\rm in}}

\newcommand{\initial}{\text{in}}
\newcommand{\xx}{\mathbf{x}}

\newcommand{\www}{\mathbf{w}}
\newcommand{\HH}{\mathcal{H}}

\newcommand{\XX}{T}

\newcommand{\vvv}{\mathbf{v}}

\begin{document}
\title[A fast dynamo on the three-torus]{\vspace*{-3cm}A fast dynamo on the three-torus} 

\author[M. Coti Zelati]{Michele Coti Zelati}
\address{Department of Mathematics, Imperial College London}
\email{m.coti-zelati@imperial.ac.uk}

\author[M. Sorella]{Massimo Sorella}
\email{m.sorella@imperial.ac.uk}

\author[D. Villringer]{David Villringer}
\email{d.villringer22@imperial.ac.uk}

\subjclass[2020]{35Q35, 34L05,  37D20, 47A55, 76E25}

\keywords{Fast dynamo, anisotropic Banach spaces, exponential growth, spectral perturbation, uniformly hyperbolic maps}

\begin{abstract}
We study the kinematic dynamo equation on the three-torus $\mathbb{T}^3$ and provide a rigorous proof of fast dynamo action for a time-periodic, divergence-free, Lipschitz velocity field. 
Our construction is based on a stretch--fold--shear mechanism generating a uniformly hyperbolic flow. To analyze the associated dynamics, we develop anisotropic Banach spaces adapted to the underlying hyperbolic structure, allowing us to recover a discrete spectral picture for the ideal dynamo operator. In the strong-chaos regime, we show that this operator admits an eigenvalue with modulus strictly larger than one. 
We then prove that this instability persists under the singular perturbation induced by diffusion, yielding exponential growth of the magnetic field uniformly in the vanishing resistivity limit $\varepsilon \to 0$.
\end{abstract}

\maketitle

\tableofcontents

\section{Introduction}\label{sec:intro}

The fast dynamo conjecture asserts the existence of exponentially growing solutions to the kinematic dynamo equation
\begin{equation} \label{passive-vector}\tag{KDE} 
\begin{cases} \partial_t B=\nabla \times (u\times B)+\varepsilon \Delta B, \\ 
\nabla \cdot B =0, 
\end{cases} \qquad (t,x)\in (0,\infty)\times \mathbb{T}^3, 
\end{equation}
uniformly in the vanishing resistivity limit $\varepsilon\to 0$. Here, $\mathbb{T}^3=\mathbb{R}^3/\mathbb{Z}^3\simeq [0,1)^3$ denotes the periodic box, and the velocity field $u:(0,\infty)\times \mathbb{T}^3\to \RR^3$ is divergence-free. 
The conjecture, originally posed by Ya.~B.~Zeldovich and A.~D.~Sakharov in the 1970s (see \cite{AK98}*{Chapter~V} and \cite{Arnold04}*{Pb.~1994--28}), can be stated formally as follows.

\smallskip

\noindent\textbf{Fast Dynamo Conjecture.} \textit{There exist a smooth, time-independent, divergence-free field $u:\TT^3\to\RR^3$, and constants $\eps_0, \gamma_0>0$ such that, for every $\varepsilon\in(0,\eps_0]$, one can find an initial datum $B_{in}^\varepsilon\in L^2$ whose corresponding solution $B^\varepsilon(t)$ to \eqref{passive-vector} satisfies}
\begin{equation}\label{eq:expoGro}
\|B^\varepsilon(t)\|_{L^2}\ge \e^{\gamma_0 t}\,\|B_{in}^\varepsilon\|_{L^2},
\qquad \forall\, t\ge 0.
\end{equation}
\textit{Such a velocity field $u$ is called a \emph{fast dynamo}.}

\smallskip

Despite its simple formulation, this problem remains open in the class of smooth, time-independent velocity fields.

Originating in astrophysics as a mechanism to explain the persistence of magnetic fields in stars and planets \cite{larmor1919possible}, the problem has since motivated extensive research across dynamical systems, spectral theory, and stochastic analysis \cites{Arnold04,AK98,CG95,BaxendaleRozovskii93,CZNF24}.
Whether purely autonomous flows can sustain fast dynamo action remains a major open problem. Even numerically, the situation is unclear: for instance, simulations of the classical ABC flows suggest that the principal growth rate may not reach its asymptotic regime even for very small resistivities $\varepsilon\sim 10^{-4}$, see \cite{BD13}.
Instead, a variant of the problem, where one allows for time-periodic velocity fields $u$, has widely been regarded as a natural and physically relevant model of fast dynamo action \cites{AK98,CG95}. While allowing for a richer class of velocity fields than the autonomous setting, it still retains many of the central analytical, dynamical and spectral difficulties associated with sustained magnetic field amplification. Our main result establishes a fast dynamo theorem in this time-periodic framework for a class of Lipschitz velocity fields.

\begin{MAINtheorem}\label{thm:fastdynamo}
There exists a time-periodic, Lipschitz, divergence-free velocity field on $\TT^3$ which is a fast dynamo.
\end{MAINtheorem}
The significance of this result is threefold. First, the growth parameters in \eqref{eq:expoGro} are independent of $\eps \in (0,1]$. This demonstrates the sharpness of the general upper bound
\begin{equation}\label{eq:expbound}
  \|B(t)\|_{L^2}\le \e^{\|\nabla u\|_{L^\infty_{t,x}} t}\,\|B_{in}\|_{L^2},  \qquad \forall\, t\ge 0,
\end{equation}
which is valid for any  solution to \eqref{passive-vector}. 
Second, to the best of our knowledge, this is the first result that constructs a genuine fast dynamo on $\mathbb{T}^3$ driven by a Lipschitz velocity field, rather than a discrete-time map. 
Third, our proof rigorously establishes the existence of an isolated eigenvalue for the ideal dynamo operator (for $\eps=0$, see Theorem~\ref{thm:main eigenvalue}) in a suitable space of distributions and in the limit of asymptotically large stretching---a result long conjectured for flows \cites{CG95,soward_1993}, but never previously proven. 

A detailed survey of the literature and the technical motivations underlying our approach are provided in the subsequent sections. 

\subsection{The Lagrangian picture and the transfer operator approach}\label{sub:lagrangianpic}
Our approach is intrinsically Lagrangian, hinging on the representation formula for the ideal dynamo equation ($\eps=0$), given by\footnote{This representation formula yields the solution to $\partial_t B + u \cdot \nabla B - B \cdot \nabla u = 0$. If the initial datum $B_{\ini}$ satisfies $\diver B_{\ini}=0$, then this coincides with the solution to \eqref{passive-vector}.}
\begin{equation}\label{eq:solution-formula}
    B(t,x)
    = \bigl(D \Phi_t\, B_{\ini}\bigr)\circ \Phi_t^{-1}(x),
    \qquad
    \frac{\dd}{\dd t}\Phi_t(x)=u\bigl(t,\Phi_t(x)\bigr),
    \qquad
    \Phi_0(x)=x \,,
\end{equation}
where $\nabla\cdot B_{\ini}=0$. Consequently, the ideal dynamo evolution is determined by the flow map $\Phi_t$ and its derivative. By restricting attention to  $T$-periodic velocity fields, this evolution defines a linear operator $\mathcal{L}h := (D\Phi_T\, h)\circ \Phi_T^{-1}$, which is the pushforward operator induced by $\Phi_T$ acting on vector fields. In the language of dynamical systems, $\cL$ may be viewed as a \emph{vector-valued transfer operator}---the natural extension to vector fields of the scalar transfer operator $h\circ \Phi_T^{-1}$ acting on functions.

To identify a regime where $\cL$ admits unstable eigenvalues, we construct a velocity field $u$ that implements the \emph{stretch-fold-shear} (SFS) mechanism \cite{CG95} (see Section \ref{sub:velocity}). This model utilizes a pair of alternating shear flows to generate chaotic advection in the $(x,y)$-plane, a strategy that has been used extensively to establish optimal exponential mixing and enhanced dissipation rates \cites{ELM25,pierrehumbert1994tracer,BCZG23,CIS24,CRWZ23,LM26}. However, such two-dimensional dynamics are insufficient for dynamo action; as mandated by the Zeldovich anti-dynamo theorem \cite{zeldovich1980magnetic}, any purely planar flow eventually leads to the decay of the magnetic field. To circumvent this topological restriction, we introduce a third, out-of-plane shear component. This SFS paradigm has long been a cornerstone of numerical and theoretical investigations into the fast dynamo problem \cites{baylychildress,Gilbert_1993,DO93,finn1990fast,GilbertSFS,PMG23,Soward94}, as it provides a minimal setting that reconciles the simplicity of alternating shears with the three-dimensional requirements for magnetic growth.

\subsection{Uniform hyperbolicity and anisotropic Banach spaces}
\label{sec:hyperbolicity intro}
The primary obstacle in establishing the fast dynamo property \eqref{eq:expoGro} is the singular nature of the vanishing resistivity limit $\varepsilon \to 0$.  This difficulty is already apparent in simplified settings, such as pulsed diffusions, where one might attempt to view the operator $\e^{\varepsilon \Delta} \mathcal{L}$ as a perturbation of the vector-valued transfer operator $\mathcal{L}$. However, classical results \cite{Chicone_Latushkin_Montgomery-Smith_1995} show that, in the $L^2$ framework, the ideal operator does not possess a discrete spectrum; instead, its spectrum typically fills a vertical strip in the complex plane. 

Such a spectral structure is highly unstable under singular perturbations \cite{K76}. In particular, there exist velocity fields for which the ideal operator is spectrally unstable at $\varepsilon=0$, yet becomes spectrally stable for all $\varepsilon>0$. This instability explains why recent rigorous works \cites{CZSV25, sorellaVillringer2025, RowanDynamo25, NF_DV2025spectral, Oseledets_1993} do not proceed via perturbative arguments. Instead, they typically establish growth for fixed $\varepsilon>0$ and, in some cases \cites{CZSV25, sorellaVillringer2025}, use scaling arguments to pass to the limit $\varepsilon \to 0$.

To overcome this difficulty, we move beyond the $L^2$ framework and exploit the \emph{uniform hyperbolicity} of the flow. The guiding heuristic is that fast dynamo action should be associated with a limiting eigenmode as $\varepsilon\to 0$. However, it has long been understood that such a limiting eigenfunction cannot be smooth \cite{Moffatt_Proctor_1985}, see also \cite{CG95}*{Section 4.3}. Rather, as suggested in \cites{CG95,Oseledets_1993,C92}, one expects it to belong to a suitable space of distributions adapted to the chaotic features of the flow \cite{soward_1993}. 

This leads to the following strategy: identify an appropriate distribution space on which the limiting eigenfunction lives; establish favorable spectral properties of the ideal operator on this space; and finally treat the diffusive term $\varepsilon \Delta$ via singular perturbation theory.

The idea of extending transfer operators to distributional spaces in order to recover improved spectral properties has proved highly successful in hyperbolic dynamics. In the case of uniformly hyperbolic systems, the associated scalar transfer operator acts anisotropically: it enhances regularity along expanding directions while improving distributional behaviour along contracting ones. This observation underlies the construction of \emph{anisotropic Banach spaces}, developed in a series of works \cites{Rugh1992CorrelationSpectrum,Kitaev1999FredholmDeterminants,BlankKellerLiverani2002RPFAnosov,GouezelLiverani2006BanachSpacesAnosov,demers_liverani,BaladiTsujii2007AnisotropicHolderSobolev}.

In our vector-valued setting, this anisotropic structure allows one to control the essential spectrum by confining it to a disk of radius strictly smaller than the spectral radius. As a consequence, the leading spectral behaviour is governed by isolated eigenvalues—so-called \emph{Ruelle resonances}---which are stable under small diffusive perturbations \cite{Keller_Liverani}. This picture is further supported by \cite{zworski_2015}, where it is shown that, for smooth contact Anosov flows, resonances persist under sufficiently small elliptic perturbations. Related ideas have also been successfully applied to the passive scalar problem in the recent work \cite{LM26}, where anisotropic spaces were used to establish an averaged version of Batchelor's law.

There are, however, two major challenges in our setting. First, since smooth Anosov flows do not exist on $\TT^3$, we must construct anisotropic Banach spaces adapted to \emph{piecewise smooth} dynamics. Second, while the existing theory provides sharp bounds on the essential spectral radius, much less is known about the identification of the \emph{leading} spectral radius. For our purposes, it is crucial not only to produce an isolated resonance, but also to prove that its modulus is strictly larger than $1$.

To address this issue, we analyze the vector-valued transfer operator in a strong-chaos regime, where both stretching and mixing become large (see Section \ref{sub:velocity}). In this limit, the leading resonance can be identified explicitly (Theorem~\ref{thm:main eigenvalue}).

\subsection{From maps to continuous flows}
A central difficulty in the fast dynamo problem lies in the presence of strong geometric and topological constraints. In particular, a large class of \emph{antidynamo theorems} \cites{zeldovich1980magnetic,Cowling33,AK98,CG95,Klapper_Young_1995} rule out dynamo action in settings with specific symmetries; most notably, Zeldovich's theorem \cite{zeldovich1980magnetic} precludes such behaviour in two-dimensional flows. These results highlight the rigidity of the continuous-time problem and help explain the scarcity of rigorous constructions.

As a result, much of the mathematical progress has taken place in the setting of discrete-time dynamical systems \cites{baylychildress,CG95,Oseledets_1993,C92,AG93,GilbertSFS,Gilbert_1993,Vytnova_2014}. In this framework, the evolution of the magnetic field is modeled via a vector-valued transfer operator associated with a diffeomorphism $T$ of the manifold. This approach allows one to consider strongly chaotic maps, such as toral automorphisms, which need not arise as time-$1$ maps of any incompressible flow. 
While this flexibility has led to the only rigorous examples of fast dynamo action on compact domains to date \cites{Oseledets_1993,Gilbert_1993}, it comes at a cost: such maps may bypass key structural constraints of the physical problem.\footnote{A common unphysical feature of algebraic maps, such as the CAT map, is that the spatial average of their Jacobian matrix need not be the identity. As a result, even the spatial average of a vector field may grow exponentially. By contrast, such growth of the mean field cannot occur for maps generated by the flow of a divergence-free velocity field.}

In contrast, when the dynamics are generated by a continuous-time, divergence-free velocity field on a compact domain, the full force of antidynamo constraints and spectral obstructions must be confronted. This makes the construction of fast dynamos in this setting significantly more challenging.

The present work addresses this difficulty by constructing a fast dynamo generated by a time-periodic, Lipschitz velocity field on the compact manifold $\TT^3$. In doing so, we provide the first example of fast dynamo action in continuous time on a compact domain, within the natural class of divergence-free flows.

This result should be contrasted with several recent developments. The construction in \cite{CZSV25} establishes fast dynamo action for autonomous flows on $\RR^3$, where the lack of compactness alters the spectral properties of the associated operator and allows for the use of localized approximate eigenfunctions. On the other hand, recent ``$\limsup$''-type results \cites{RowanDynamo25,sorellaVillringer2025} consider fully time-dependent velocity fields and establish growth along sequences of times $t_n \to \infty$, rather than uniform exponential growth.

By contrast, our result yields uniform-in-time exponential growth for a time-periodic flow on the compact torus, fully within the constraints of the classical setting. Notably, the velocity field constructed in Theorem~\ref{thm:fastdynamo} is also a \emph{perfect dynamo}; see Corollary~\ref{corollary:perfect-dynamo}. In particular, it provides an explicit example for which the flux conjecture holds (see \cite{CG95}*{Chapter 4, Conjecture 1}).

\subsection{Continuous diffusion and singular perturbations}
As discussed in Section \ref{sec:hyperbolicity intro}, the main motivation for working on anisotropic spaces of distributions is that they provide a framework in which the diffusive term $\eps \Delta$ may be treated perturbatively. However, while the piecewise smooth structure of our dynamics is crucial for constructing suitable anisotropic Banach spaces for the inviscid problem, it also makes the perturbative analysis of diffusion substantially more delicate.

Indeed, for $\eps>0$, the stochastic flow associated with \eqref{passive-vector} typically deviates from the deterministic trajectory by distances of order $\mathcal O(\sqrt{\eps})$. Near the singularity set \eqref{d:forward-singular}, which separates the different smoothness regions of the dynamics, such fluctuations may cause trajectories to cross between distinct branches of the map, where the stable and unstable directions vary discontinuously. Consequently, the random transfer operators associated with individual noise realisations no longer preserve the hyperbolic structure underlying the anisotropic framework, preventing a direct pathwise perturbative analysis.

To overcome this difficulty, we work instead at the level of the averaged \emph{Green's matrix} (see Section \ref{sec:heat} and \eqref{eq:greens matrix expression}) associated with \eqref{passive-vector}. A central part of the paper is devoted to establishing sharp $\eps$-dependent kernel estimates showing that, away from the singularity set, the Green's matrix is well approximated by an explicit \emph{model kernel},  see  Section \ref{sec:kernel}. Inside each smoothness strip, this model kernel behaves as a Gaussian transport kernel, while the discrepancy between the true and model dynamics is exponentially small outside an $\mathcal O(\sqrt{\eps})$ neighbourhood of the singularity set.

The remaining difficulty is therefore concentrated precisely at spatial scales of order $\mathcal O(\sqrt{\eps})$. To handle this regime, we introduce an additional ``idle heat block'', namely an auxiliary diffusive step acting exactly at these critical scales. This additional smoothing restores sufficient regularity to close the singular perturbation argument and ultimately yields spectral stability of the leading resonance under continuous-time diffusion.

This should be contrasted with the somewhat simpler setting of \emph{pulsed diffusion}, which is more common in the dynamical systems literature and in previous works on fast dynamos for maps \cites{AK98,Oseledets_1993,CG95}. In that setting, one studies the vector-valued transfer operator
\begin{equation}\label{eq:transfer}
 \cL B=(DT\,B)\circ T^{-1},
\end{equation}
associated with a diffeomorphism $T:\TT^3\to\TT^3$, together with the pulsed operator
\[
B \mapsto \e^{\eps \Delta}\cL B.
\]
A fast dynamo in the setting of maps corresponds to a map $T$ for which the spectral radius of $\e^{\eps\Delta}\cL$ remains strictly larger than $1$ uniformly as $\eps\to0$. As a consequence of the methods developed in the present paper, we also obtain the following result for pulsed dynamos.

\begin{theorem}
\label{thm:fastpulsed}
There exists a Lipschitz continuous, measure-preserving map $T:\TT^3\to\TT^3$, isotopic to the identity, such that the associated pulsed transfer operator $\e^{\eps\Delta}\cL$, with $\cL$ defined in \eqref{eq:transfer}, satisfies
$$
\liminf_{\eps \to 0} \lim_{n \to \infty} \|(\e^{\eps \Delta} \cL)^n\|^{1/n} >1.
$$
In other words, it is a fast dynamo in the sense of maps.
\end{theorem}

\begin{remark}
The kernel estimates developed in this paper are also closely connected to problems in mixing and enhanced dissipation. In particular, similar perturbative and spectral-stability arguments can be used to establish exponential mixing estimates for passive scalars, uniform for $0<\eps\ll1$, as well as enhanced dissipation estimates with decay rates independent of $\eps$ for the first two components of the velocity field defined in \eqref{eq:u_def}. Since this lies beyond the scope of the present work, we leave these questions for future investigation.
\end{remark}

% {\color{red}
%     \begin{remark}
% Similar perturbative and spectral-stability arguments can be used to prove exponential mixing estimates for passive scalars, uniform for $0<\varepsilon\ll1$,  as well as  enhanced dissipation estimates with exponential rate independent on $\eps$
% with  the first two components of the velocity field defined in \eqref{eq:u_def}. 
% Since this lies beyond the scope of the present paper, we leave it for future work.
% \end{remark}}

\section{Main ideas and outline of the proof}

The proof of Theorem~\ref{thm:fastdynamo} is based on a quantitative analysis of a family of transfer operators in a strong-chaos regime, where stretching, folding, and phase mixing interact to produce spectral instability. 

This section outlines the main ingredients of the argument. We first describe the geometric construction of the velocity field and the associated dynamical mechanism (Section~\ref{sub:velocity}). We then present the perturbative framework of Keller and Liverani \cite{Keller_Liverani}, which allows us to control the effect of diffusion (Section~\ref{sub:KellerLiverani}). Finally, we summarize the key steps in the proof of Theorem~\ref{thm:fastdynamo} (Section~\ref{sub:proof}).

\subsection{The stretch--fold--shear velocity field}\label{sub:velocity}

To generate the hyperbolic dynamics underlying our construction, we build a time-periodic velocity field based on the classical \emph{stretch--fold--shear} (SFS) mechanism from dynamo theory \cite{CG95}, combined with recent ideas from optimal mixing \cites{ELM25,LM26}. The goal is to produce strong planar stretching and folding, while introducing a transverse shear that prevents cancellation of the magnetic field.

Fix $\alpha \in 2\NN$. For $(x,y,z)\in \TT^3$, we define the constituent shear flows
\begin{equation} \label{eq:shears}
V_\alpha(x,y) = \begin{pmatrix} 0 \\ 2\alpha |x - 1/2| \\ 0 \end{pmatrix}, \qquad 
H_\alpha(x,y) = \begin{pmatrix} 2\alpha |y - 1/2| \\ 0 \\ 0 \end{pmatrix}, \qquad 
Z(x,y) = \begin{pmatrix} 0 \\ 0 \\ -g(x,y) \end{pmatrix},
\end{equation}
where $g:\TT^2 \to \RR$ is a $C^1$ function. 
For $N\in \mathbb{N}$ large but fixed, we define the time-periodic velocity field $u_{\alpha,N}$ over one period $[0,N+3)$ by
\begin{equation} \label{eq:u_def}
u_{\alpha, N} (t, x, y, z) = \begin{cases}
0, & t \in [0,N),\\
V_\alpha(x, y), \quad & t \in [N, N+1/2), \\
H_\alpha(x, y), & t \in [N+1/2, N+1), \\
0, & t \in [N+1,N+2),\\
Z(x, y),  & t \in [N+2, N+3)
\end{cases}
\end{equation}
and extend it periodically to all $t\geq 0$. By construction, $u_{\alpha,N}$ is divergence-free. Over one period, the flow mainly consists of a planar stretch--fold step in the $(x,y)$-variables, and a shear in the $z$-direction. 

\begin{remark}
The velocity field can be smoothed in time without affecting the analysis, but spatial regularity is essential: the arguments rely on the piecewise linear structure of the flow.
\end{remark}

The planar component generates, through its flow map over $t\in [N,N+1]$, a uniformly hyperbolic map $T_\alpha:\TT^2\to\TT^2$ (see Section~\ref{sec:themap}). This map was used in \cite{ELM25} to prove optimal exponential mixing and in \cite{LM26} to establish a cumulative form of Batchelor's law. 
However, purely two-dimensional dynamics cannot produce dynamo action due to classical antidynamo theorems \cites{Zeldovich_1992,Oseledets_1993}. Crucially, this obstruction persists even at the level of transfer operators. To overcome it, we introduce the out-of-plane shear $Z$, following the SFS paradigm \cite{CG95}. Physically, this shear modulates the phase of the magnetic field and prevents the destructive interference that would otherwise lead to decay.

To formalize this mechanism, consider the time-$(N+3)$ map
\[
\mathcal{T}_\alpha(x,y,z)
=
\begin{pmatrix}
T_\alpha(x,y) \\
z - g(T_\alpha(x,y))
\end{pmatrix}.
\]
We consider initial data of the form
\begin{equation} \label{d:initial}
B_{\ini}(x,y,z) = \e^{2\pi i z}
\begin{pmatrix}
h(x,y) \\
H^{(3)}(x,y)
\end{pmatrix},
\end{equation}
where $h:\TT^2\to\RR^2$ and $H^{(3)}:\TT^2\to\RR$. The oscillatory factor $\e^{2\pi i z}$ isolates a single Fourier mode in the vertical direction, allowing the effect of the shear to be captured through a phase modulation in the planar variables.
The planar evolution is governed by the vector-valued transfer operator
\begin{equation}\label{eq:Lalpha}
\mathcal{L}_\alpha h
=
\frac{1}{\alpha^2} (D T_\alpha\, h)\circ T_\alpha^{-1}.
\end{equation}
A direct computation shows that
\begin{align} \label{eq:comp-intro}
(D\mathcal{T}_\alpha B_{\ini})\circ \mathcal{T}_\alpha^{-1}(x,y,z)
=
\e^{2\pi i z}\,\e^{2\pi i g(x,y)}
\begin{pmatrix}
\alpha^2 \mathcal{L}_\alpha h \\
\mathcal{K}_{\alpha, g} h + H^{(3)}\circ T_\alpha^{-1}
\end{pmatrix},
\end{align}
for a suitable bounded operator $\mathcal{K}_{\alpha, g}$. While the purely two-dimensional operator $\alpha^2 \mathcal{L}_\alpha$ cannot produce growth due to antidynamo constraints, the phase-modulated operator
\[
\e^{2\pi i g(x,y)}\,\alpha^2 \mathcal{L}_\alpha
\]
has fundamentally different spectral properties. The shear introduces oscillations that prevent cancellation, and we show that this operator admits unstable eigenvalues in the strong-chaos limit $\alpha\to\infty$ (see Theorem~\ref{thm:main eigenvalue}).

\begin{remark}[On the strong-chaos limit]
The strong-chaos limit $\alpha \to \infty$ is non-trivial precisely because of the phase-shearing effect. In our construction, introducing the shear $g$ at each iteration with $\alpha \gg 1$ yields fundamentally different spectral behaviour compared to applying a shear only after $n \gg 1$ iterations of a fixed map.  
Although both scenarios formally correspond to a regime of ``infinite mixing'', we show in Remark~\ref{remark:1 iteration} that even two iterations of the advection step performed before a single shear result in a limiting operator with no non-trivial eigenvalues. This highlights the extreme sensitivity of the dynamo mechanism to the fine-scale temporal structure of the flow.
\end{remark}

Our analysis also extends to the case of positive diffusion $\eps>0$. We focus on the evolution of the first two components of the magnetic field and introduce the corresponding solution operators.

We denote by $\alpha^2 \cL_{\alpha,\eps}$ the solution operator from time $t=N$ to $t=N+2$ associated with the two-dimensional dynamics on $\TT^2$. More precisely, for $b:[N,N+2]\times \TT^2 \to \RR^2$ solving
\begin{align} \label{d:solution-cL-alpha-eps}
\partial_t b + u_{\alpha,N}^{(1,2)} \cdot \nabla b - b \cdot \nabla u_{\alpha,N}^{(1,2)} = \eps \Delta b,
\end{align}
we define $\alpha^2 \cL_{\alpha,\eps} b(N) = b(N+2)$. Here, $u_{\alpha,N}^{(1,2)}$ refers to the first two components of $u_{\alpha,N}$. Note that this evolution includes the interval $[N+1,N+2]$, during which $u_{\alpha,N}\equiv 0$ and only diffusion acts.

We next consider the shear phase, corresponding to the time interval $[N+2,N+3]$. In this regime, the first two components evolve according to
\begin{align} \label{d:transport-diffusion-in-z}
\partial_t b - g \partial_z b = \eps \Delta b,
\end{align}
for $b:[0,1]\times \TT^3 \to \RR^2$. Writing $\Delta_{x,y} = \partial_{xx} + \partial_{yy}$, we observe that for initial data of the form
\[
b_{\initial}(x,y,z) = \e^{2\pi i z} h_{\initial}(x,y),
\]
the solution remains of the form $b(t,x,y,z) = \e^{2\pi i z} h(t,x,y)$, where $h$ solves
\[
\partial_t h
=
\varepsilon (\Delta_{x,y} - 4\pi^2) h
+ 2\pi i\, g\, h.
\]
We denote by $\cL_{g,\eps}$ the corresponding solution operator at time $t=1$, so that $h(1)=\cL_{g,\eps} h_{\initial}$.

Combining the different phases of the evolution, a direct computation shows that if $B$ solves \eqref{passive-vector} with velocity field $u_{\alpha,N}$ and initial datum \eqref{d:initial},
then the first two components $b = B^{(1,2)}$ at time $t=N+3$ satisfy
\begin{align} \label{d:eps-2d-operator}
b(N+3)
=
\e^{2\pi i z}\,
\cL_{g,\eps}
\circ
\e^{-8\pi^2 \eps}\,\alpha^2 \cL_{\alpha,\eps}
\circ
\e^{N\eps(\Delta_{x,y}-4\pi^2)} h.
\end{align}
This factorization reflects the three stages of the dynamics: an initial diffusive phase, a planar advection–diffusion step, and a final shear–diffusion step. It provides the starting point for our perturbative analysis of the full problem.

\subsection{The spectral perturbation framework}\label{sub:KellerLiverani}

As discussed above, our goal is to construct an anisotropic Banach space on which the ideal dynamo operator admits Ruelle resonances, namely isolated eigenvalues separated from the essential spectrum. The key tool for achieving this is a \emph{Lasota--Yorke inequality}, which provides quantitative control of the essential spectral radius via the Nussbaum formula \cite{Nussbaum}. We recall a standard formulation (see, e.g., \cite{Baladi_textbook}).

\begin{proposition}\label{prop:lasota_yorke abstract}
Let $(X, \| \cdot \|)$ and $(X_w,|\cdot |_w)$ be Banach spaces such that $X \subset X_w$, with continuous inclusion $|\cdot|_w \leq \|\cdot \|$. Assume that the unit ball $\{f \in X: \| f \| \leq 1 \}$ is relatively compact in $(X_w,|\cdot |_w)$. 

Let $\cP: X \to X$ be a bounded linear operator that also extends to a bounded operator $\cP: X_w \to X_w$. Suppose that $\cP$ satisfies a \emph{Lasota--Yorke inequality}: there exist constants $0 \leq a < M$ such that 
\begin{equation}\label{eq:lasota-Yorke theorem}
\|\cP f\|\leq a\|f\|+M|f|_w \qquad \forall f \in X \,.
\end{equation}
Then, the essential spectrum $\sigma_{\mathrm{ess}}(\cP)$ of $\cP$ on $X$ is contained in a disk $\overline{B_a(0)}$ of radius $a$. In particular, any spectral point of modulus strictly larger than $a$ is an isolated eigenvalue of finite multiplicity.
\end{proposition}

The relevance of Proposition~\ref{prop:lasota_yorke abstract} is twofold. First, once a Lasota--Yorke inequality is established on a suitable anisotropic space $X$, it yields a spectral gap: any eigenvalue of modulus larger than $a$ is automatically isolated. Second, such isolated eigenvalues are stable under singular perturbations, as ensured by the following theorem of Keller and Liverani \cite{Keller_Liverani}.

\begin{proposition}[Keller--Liverani]\label{prop:liverani-keller}
Let $(\cP_\eps)_{\eps \geq 0}$ be a family of bounded linear operators acting on Banach spaces $(X, \| \cdot \|)$ and $(X_w,|\cdot |_w)$ satisfying the assumptions of Proposition~\ref{prop:lasota_yorke abstract}. Suppose there exist constants $0 \leq a < M$ such that for all $f \in X$:
\begin{enumerate}[label=(\roman*), ref=(\roman*)]
\item \label{check:1} $|\cP_\eps f|_w \leq M |f|_w$;
\item \label{check:2} $|\cP_\eps f- \cP_0 f|_w \leq \tau_\eps \|f\|$, with $\tau_\eps \to 0$ as $\eps \to 0$;
\item \label{check:3} $\|\cP_\eps f\|\leq a \|f\|+ M|f|_w$;
\item \label{check:4} $\cP_0$ admits a unique, (algebraically) simple eigenvalue $\lambda \in \mathbb{C}$ with $|\lambda| > a$.
\end{enumerate}
Then, there exists $\delta > 0$ such that for all $\eps$ small enough, $\sigma(\cP_\eps) \cap \{ \mu \in \mathbb{C} : |\mu - \lambda|\leq \delta \}$ consists of a single eigenvalue that converges to $\lambda$ as $\eps \to 0$. 
\end{proposition}

We apply this framework through a double-limit argument. In the strong-chaos regime $\alpha \to \infty$, we first analyze the ideal operator and identify a leading resonance $\lambda_\alpha$ for $\alpha^2 \e^{2\pi i g}\mathcal{L}_\alpha$ satisfying $|\lambda_\alpha| > \alpha^2/4$ (see \eqref{eq:Lalpha}--\eqref{eq:comp-intro} and Theorem~\ref{thm:main eigenvalue}). We then establish a Lasota--Yorke inequality that is uniform in $\varepsilon$, which ensures that this eigenvalue is isolated and persists under the singular perturbation induced by diffusion.

\begin{remark}
A key technical difficulty is that the anisotropic spaces depend on the stretching parameter $\alpha$. Extracting spectral information in the limit $\alpha \to \infty$ therefore requires quantitative control of the resolvent, as well as a refinement of the Keller--Liverani framework to track this dependence. The limiting resonances provide the spectral anchor for the perturbative analysis of the diffusive problem.
\end{remark}

\subsection{Outline of the proof of Theorem \ref{thm:fastdynamo}}\label{sub:proof}
As hinted throughout the introduction, our proof relies on three key propositions. 
The goal of this section is to demonstrate how, once these building blocks are established, one can deduce the existence of a fast dynamo via the abstract results of Propositions \ref{prop:lasota_yorke abstract} and \ref{prop:liverani-keller}. 
In Section \ref{sec:anisotropic}, we construct suitable anisotropic Banach spaces $(X,\|\cdot\|)$ and $(X_w,|\cdot|_w)$ satisfying the assumptions of Proposition \ref{prop:lasota_yorke abstract}. These spaces depend intrinsically on the large parameter $\alpha \in 2\mathbb{N}$; we suppress this dependence for brevity, though extracting uniform spectral behaviour as $\alpha \to \infty$ requires quantitative resolvent estimates.
The first key ingredient is a Lasota--Yorke inequality.

\begin{proposition}[Lasota--Yorke inequality]\label{proposition:main lasota yorke}
There exist constants $C, \eta > 0$, independent of $\alpha$, such that for all $\alpha$ sufficiently large,
\begin{equation}\label{eq:Lasointro}
\|\e^{2 \pi i g} \mathcal{L}_\alpha h \| \leq C\alpha^{-\eta} \| h \| + C |h|_w, \qquad |\e^{2 \pi i g}  \mathcal{L}_\alpha h |_w \leq C |h|_w, \qquad \forall h \in X .    
\end{equation}   
Consequently, by Proposition \ref{prop:lasota_yorke abstract}, any eigenvalue of $\e^{2 \pi i g} \mathcal{L}_\alpha$ with modulus strictly greater than $C\alpha^{-\eta}$ is isolated and of finite multiplicity.
\end{proposition}

In Section \ref{sec:spectral-ideal}, we characterize the limiting operator $\e^{2 \pi i g}\mathcal{L}_\infty$ of $\e^{2 \pi ig}\mathcal{L}_\alpha$ as $\alpha\to\infty$. In a suitable sense (see Lemma \ref{lemma:convergence}), one has
\[
    \e^{2 \pi i g}\mathcal{L}_\infty h
= \e^{2 \pi i g}
\begin{pmatrix}
\left ( \int_{x \geq 1/2} h^{(1)} - \int_{x<1/2} h^{(1)} \right )
\left(\mathds{1}_{y \geq 1/2} - \mathds{1}_{y<1/2}\right)
\\
0
\end{pmatrix}
 \,, \qquad h=(h^{(1)},h^{(2)})
\]
Using this characterization alongside Proposition \ref{proposition:main lasota yorke}, we prove in Section \ref{sec:spectral-ideal} the existence of a discrete eigenvalue with modulus larger than $1$ (in fact, of order $\approx \alpha^2$ for $\alpha\gg 1$).

\begin{theorem}\label{thm:main eigenvalue}
There exist $g \in C^1(\mathbb{T}^2)$ and $\alpha_0 \geq 1$ such that, for every $\alpha \geq \alpha_0$ with $\alpha \in 2\mathbb{N}$, the operator $\alpha^2 \e^{2\pi i g}\mathcal{L}_\alpha : X \to X$ admits a unique discrete, algebraically simple eigenvalue $\lambda_\alpha \in \mathbb{C}$ satisfying $|\lambda_\alpha| \geq \alpha^2/4$.
\end{theorem}

Finally, we establish a uniform-in-$\varepsilon$ Lasota--Yorke inequality for the diffusive operator and a weak--strong convergence result (Section \ref{sec:heat}).

\begin{proposition}\label{prop: uniform Lasota Yorke}
For any $\alpha \geq 1$ there exists  $N \in \NN$ large enough and $\eps_0>0$ so that for all $\eps\leq \eps_0$ it holds
\begin{align} \label{eq:uniform-lasota-yorke}
\|\cL_{g, \eps} \circ  \e^{- 8\eps \pi^2 } \cL_{\alpha, \eps} \circ \e^{N  \eps (\Delta_{x,y} - 4\pi^2)} h\|\leq C\alpha^{-\eta} \|h\|+C|h|_w\,, \qquad \forall h \in X\,,
\end{align}
and 
\begin{align}  \label{eq:weak-full-operator}
|\cL_{g, \eps} \circ  \e^{- 8\eps \pi^2 } \cL_{\alpha, \eps} \circ \e^{N  \eps (\Delta_{x,y} - 4\pi^2)} h |_w \leq C|h|_w\,, \qquad \forall h \in X\,, 
\end{align}
where the constant $C>0$ depends only on $g$.
Furthermore, there exists $\beta>0$ so that the following weak-strong convergence holds true  for all $\eps \leq \eps_0$
\begin{align} \label{eq:weak-strong-convergence-full-operator}
|\cL_{g, \eps} \circ  \e^{- 8\eps \pi^2 } \cL_{\alpha, \eps} \circ \e^{N  \eps (\Delta_{x,y} - 4\pi^2)} h -\e^{2 \pi i g}\cL_\alpha h|_w \leq C\eps^{ \beta /2}\|h\| \,, \qquad \forall h \in X\,,
\end{align}
where the constant $C>0$ depends on $\alpha, N$ and $g$.
\end{proposition}

With these three propositions in hand, we can now establish our main result.

\begin{proof}[Proof of Theorem \ref{thm:fastdynamo}]
We consider the velocity field $u_{\alpha, N}$ defined in \eqref{eq:u_def}, extended periodically in time, and initial data of the form \eqref{d:initial}. The time-$N+3$ ideal dynamo operator $\cP_{3d,0}$ acts as\footnote{This is the solution formula for the PDE $\partial_t B + u \cdot \nabla B - B \cdot \nabla u =0$ which is equivalent to $\partial_t B = \nabla \times (u \times B)$ if both $u$ and $B_{\initial}$ are divergence-free.}
\[
(D \mathcal{T}_\alpha B_{\mathrm{in}})\circ \mathcal{T}_\alpha^{-1}
= \e^{2 \pi i z} \e^{2 \pi i g(x,y)}
\begin{pmatrix}
\alpha^2\mathcal{L}_\alpha h \\
\mathcal{K}_{\alpha, g} h + H^{(3)} \circ \XX_\alpha^{-1}
\end{pmatrix},
\]
where $\mathcal{K}_{\alpha, g}$ is a  bounded linear operator on $L^2$.

For $\varepsilon>0$, the time-$N+3$ operator $\cP_{3d,\varepsilon}$ is given by
$$
\cP_{3d,\eps} (\e^{2 \pi i z}H) =  \e^{2 \pi iz} \begin{pmatrix}      \cL_{g, \eps} \circ  \e^{- 8\eps \pi^2 } \alpha^2\cL_{\alpha, \eps} \circ \e^{N  \eps (\Delta_{x,y} - 4\pi^2)} h
\\
\mathcal{K}_{\alpha, \eps, g} h +  \bar{\cL}_{g,\eps }\circ \e^{ (N +2)\eps \Delta} H^{(3)} 
    \end{pmatrix} \,,
    $$
    where the operator on the first two components has been introduced in \eqref{d:eps-2d-operator} and $\mathcal{K}_{\alpha, \eps, g}$ is a bounded linear operator on $L^2$ and $\bar{\cL}_{g,\eps }$ is the time-one  solution operator of the one-dimensional equation \eqref{d:transport-diffusion-in-z}.
 
 \medskip

\noindent   
\emph{Step 1: Spectral instability of the $2\times2$ block.}
By Theorem \ref{thm:main eigenvalue}, for $\alpha \geq \alpha_0$, the ideal operator 
$$\cP_{0}= \e^{2 \pi ig(x,y)} \alpha^2\mathcal{L}_\alpha h $$ 
has an eigenvalue $|\lambda| \geq \alpha^2/4$, thus verifying \ref{check:4}. The weak norm bound  \ref{check:1},  the
weak-strong convergence \ref{check:2}  and the uniform Lasota--Yorke inequality \ref{check:3} follow directly from Proposition \ref{proposition:main lasota yorke} and  Proposition \ref{prop: uniform Lasota Yorke}. 
Hence, an application of Proposition \ref{prop:liverani-keller} gives that
for all $\eps \leq \eps_0$, the operator $ \cP_{\eps}$ admits a discrete eigenvalue $\lambda_\eps$ with $|\lambda_\eps| \geq \alpha^2/8 > 1$.

\medskip

\noindent
\emph{Step 2: Lifting to the full system.}
Let $h$ be the corresponding unstable eigenfunction of $\cP_{\eps}$. Define
\[
J_\varepsilon(H^{(3)}) 
:= \bar{\cL}_{g,\varepsilon}\circ \e^{ (N +2)\varepsilon \Delta} H^{(3)}.
\]
Since $\|J_\eps\|_{L^2 \to L^2} < 1$ and $|\lambda_\eps| > 1$, the operator $(\lambda_\eps - J_\eps)$ is invertible. Setting 
\[
H^{(3)} = (\lambda_\varepsilon - J_\varepsilon)^{-1}  \mathcal{K}_{\alpha, \varepsilon, g} h
\]
yields an eigenfunction for $\cP_{3d,\varepsilon}$ with eigenvalue $\lambda_\varepsilon$. By parabolic regularity, the eigenfunction is at least $H^2$-regular, since $g\in C^1(\TT^2)$.

\medskip

\noindent
\emph{Step 3: Divergence-free condition.}
Taking divergence,
\[
\partial_t B + u_{\alpha,N} \cdot \nabla B - B \cdot \nabla u_{\alpha,N} =\eps \Delta B  \qquad   \Rightarrow \qquad  (\partial_t  + u_{\alpha,N} \cdot \nabla )(\nabla\cdot B) =\eps \Delta (\nabla\cdot B).
\]
Hence
\[
\|\nabla \cdot \cP_{3d,\varepsilon} B_{\mathrm{in}}\|_{L^2}
\leq \|\nabla \cdot B_{\mathrm{in}}\|_{L^2}.
\]
For the eigenfunction,
\[
|\lambda_\varepsilon| \|\nabla \cdot B_{\mathrm{in}}\|_{L^2}
\leq \|\nabla \cdot B_{\mathrm{in}}\|_{L^2}.
\]
Since $|\lambda_\varepsilon|>1$, this implies $\nabla \cdot B_{\mathrm{in}} = 0$.

\medskip

\noindent
\emph{Step 4: Extension to continuous time.}
Finally, let $U^\eps(s,t):L^2 \to L^2$ be the propagator of the PDE \eqref{passive-vector} with velocity field $u_{\alpha,N}$, in other words, $B^\eps(t)=U^\eps(s,t)B_{\ini}$ is the solution to \eqref{passive-vector} at time $t$, with initial condition $B^\eps(s)=B_{\ini}$. In particular, $U^\eps(0,N+3)$ is equal to $\mathcal{P}_{3d,\eps}$. Take $B^\eps_{\ini}$ to be the principal eigenfunction of $U^\eps(0,N+3)$, with eigenvalue $\lambda_\eps$, and let $B^\eps(t)=U(0,t)B^\eps_{\ini}$. Then, there holds $\|B^\eps(n(N+3))\|_{L^2(\TT^3)}=|\lambda_\eps|^n\|B^\eps_{\ini}\|_{L^2(\TT^3)}$, and in view of the growth bound \eqref{eq:expbound}, it holds that 
$$
\e^{-\|\nabla u_{\alpha,N}\|_{L_{t,x}^\infty}(N+3)}\|B^\eps(n(N+3))\|_{L^2(\TT^3)}\leq \|B^\eps(n(N+3)+t)\|_{L^2(\TT^3)},
$$
for all $t \in [0,N+3]$ and any $n \in \NN$. Thus, we conclude the proof.

% so we estimate for any $t \in [n(N+3),(n+1)(N+3)]$
% $$
% \|B^\eps(t)\|_{L^2(\TT^3)} \geq \e^{-\|\nabla u_{\alpha,N}\|_{L_{t,x}^\infty}(N+3)}\|B^\eps(n(N+3))\|_{L^2(\TT^3)} \geq |\lambda_\eps|^t |\lambda_\eps|^{-(N+3)} \e^{-\|\nabla u_{\alpha,N}\|_{L_{t,x}^\infty}(N+3)}\|B_{\ini}\|_{L^2(\TT^3)},
% $$
% completing the proof.
\end{proof}

\subsection{Connection with the flux conjecture}
As a consequence of our spectral approach, we show that the velocity field $u_{\alpha, N}$, for $\alpha$ sufficiently large but fixed, is a \emph{perfect dynamo}. This notion concerns the growth of magnetic flux in the ideal limit and is generally more delicate than energy growth, due to the possibility of fine-scale cancellations. We recall the definition.

\begin{definition}[Perfect dynamo, \cite{CG95}] \label{d:perfect-dynamo}
Let $\psi: \TT^3 \to \RR^3$ be a smooth vector field. The \emph{perfect dynamo growth rate relative to $\psi$}, denoted by $\Gamma(\psi)$, is defined as
\[
\Gamma(\psi)
=
\sup_{\substack{B_{\ini} \in L^2(\mathbb{T}^3) \\ \nabla \cdot B_{\ini}=0}}
\limsup_{t \to \infty} \frac{1}{t} 
\log\left | \int_{\mathbb{T}^3} B(t,x) \cdot \psi(x) \, \dd x\right |,
\]
where $B(t,x)$ is the solution to the ideal equation \eqref{passive-vector} (with $\varepsilon=0$) with initial datum $B_{\ini}$. 
A time-periodic, divergence-free velocity field $u \in L^\infty_t W^{1,\infty}_x$ is called a \emph{perfect dynamo} if there exists some $\psi \in C^\infty (\TT^3; \RR^3)$ such that $\Gamma (\psi) > 0$.
\end{definition}
The relationship between ideal flux growth and diffusive energy growth is encapsulated in the following conjecture \cite{CG95}*{Chapter 4, Conjecture 1}.

\smallskip

\noindent\textbf{Flux conjecture.} \textit{Every perfect dynamo is a fast dynamo.}

\smallskip

The flux conjecture suggests that the growth mechanism underlying a perfect dynamo is robust enough to overcome the dissipative effects of arbitrarily small diffusion \cite{CG95}*{Chapter 4}. In Section~\ref{subsec:flux}, we prove that the velocity field $u_{\alpha, N}$ satisfies the perfect dynamo condition. Combined with Theorem~\ref{thm:fastdynamo}, this provides a concrete example for which the flux conjecture holds.

\begin{corollary} \label{corollary:perfect-dynamo}
The velocity field $u_{\alpha, N}$ defined in \eqref{eq:u_def} is a perfect dynamo for any $\alpha \in 2\NN$ sufficiently large and any $N \in \NN$.
\end{corollary}

\subsection{Notation} \label{subsec:notation}  We denote by $(X,\|\cdot\|)$ the anisotropic Banach space endowed with the strong norm, and by $|\cdot|_w$ a weak norm on $X$ defined in Section \ref{sec:anisotropic}.
For a bounded linear operator $\cL:X\to X$, we define the operator norms
\[
\|\cL\|_{X\to X}=\sup_{\|x\|\le 1}\|\cL x\|,
\qquad
\|\cL\|_{w\to w}=\sup_{|x|_w\le 1}|\cL x|_w,
\qquad
\|\cL\|_{X\to w}=\sup_{\|x\|\le 1} |\cL x |_w.
\]
For any vector $\vvv\in\RR^d$ with $d\ge 2$, we write $\vvv^{(i)}$ for its $i$-th component.  
We denote by $\dd_{\TT^d}$ the standard toroidal distance on $\TT^d \cong [0,1]^d/\sim$.

Throughout the manuscript, the constant $C>0$ may change from line to line but is always independent of all relevant parameters, in particular independent of $\alpha$.
We abuse notation slightly and denote by $C^q(W)$ the closure of Lipschitz functions with respect to the $\|\cdot\|_{C^q(W)}$ norm. In particular, we write $C^1(\mathbb{T}^2)$ for the space of Lipschitz functions on $\TT^2$.
Finally, for any curve $W$ and any function $\varphi \in C^1(W)$, we denote by
\[
\int_W \varphi \, \dd \HH^1
\]
the line integral with respect to the one-dimensional Hausdorff measure.

\section{A family of hyperbolic maps on $\TT^2$} \label{sec:themap}
In this section, we analyze a family of hyperbolic maps $(T_\alpha)_{\alpha \geq \alpha_0} : \TT^2 \to \TT^2$, given by the first two components of the time-$ N+1$ flow map of the velocity field $u_{\alpha, N}$ defined in \eqref{eq:u_def}, i.e. without the shearing action induced by $Z(x,y)$. 

We define the indicator functions on the circle $\mathbb{T} = [0,1)$ as
\begin{equation}
\mathbbm{1}_{\geq a} (x) = 
\begin{cases}
    1, \quad& \text{if } x \in [a, 1), \\
    0, & \text{otherwise},
\end{cases}
\end{equation}
and let $\mathbbm{1}_{< a} (x) = 1 - \mathbbm{1}_{\geq a} (x)$. Given the velocity field defined in \eqref{eq:shears}--\eqref{eq:u_def} with shear strength $\alpha \in 2\mathbb{N}$, we define the horizontal and vertical shearing flow maps, $T_{H_\alpha}$ and $T_{V_\alpha}$ respectively, as:
\begin{equation}
T_{H_\alpha}(x,y) = 
\begin{pmatrix} 
(x + \alpha y) \mathbbm{1}_{\geq 1/2} (y) - (x-\alpha y) \mathbbm{1}_{< 1/2} (y) \\ 
y 
\end{pmatrix} \pmod{1}
\end{equation}
and
\begin{equation}
T_{V_\alpha}(x,y) = 
\begin{pmatrix} 
x \\ 
(y + \alpha x) \mathbbm{1}_{\geq 1/2} (x) +(y- \alpha x) \mathbbm{1}_{< 1/2} (x)
\end{pmatrix} \pmod{1}.
\end{equation}
Let $\XX_\alpha: \mathbb{T}^2 \to \mathbb{T}^2$ be the $(x,y)$-components of the time-1 flow map associated with the velocity field $u$. This map is given by the composition
\begin{equation}
    \XX_\alpha = T_{H_\alpha} \circ T_{V_\alpha}, \qquad \XX_\alpha^{-1} = T_{V_\alpha}^{-1} \circ T_{H_\alpha}^{-1}.
\end{equation}
By construction, $\XX_\alpha$ is a piecewise linear, volume-preserving map. There exists a partition of the torus into four regions $\mathcal{M}_\ell \subset \mathbb{T}^2$ for $\ell=1, \dots, 4$, such that the restriction $\XX_\alpha |_{\mathcal{M}_\ell}$ coincides with a linear transformation $A_\ell \in SL(2, \mathbb{Z})$. We refer to these $\mathcal{M}_\ell$ as the \emph{forward smoothness regions} of the map.

Correspondingly, by the inverse function theorem, the inverse map $\XX_\alpha^{-1}: \XX_\alpha(\mathcal{M}_\ell) \to \mathcal{M}_\ell$ is given by the inverse matrices $A_\ell^{-1}$. We refer to the sets $\XX_\alpha(\mathcal{M}_\ell)$ as the \emph{backward smoothness regions}. We now proceed to characterize the geometry of these regions and the explicit form of the matrices $A_\ell$.

\subsection*{The forward map} 
By computing the composition $\XX_\alpha = T_{H_\alpha} \circ T_{V_\alpha}$ explicitly, we find that the second component is determined solely by the initial vertical shear, namely
\begin{equation}\label{eq:Talpha2}
\XX_\alpha^{(2)}(x,y) = (y + \alpha x) \mathbbm{1}_{\geq 1/2}(x) + (y - \alpha x) \mathbbm{1}_{< 1/2}(x) \pmod{1}.
\end{equation}
The first component, which incorporates the horizontal shear acting on the already-sheared vertical coordinate, is given by:
\begin{align}
    \XX_\alpha^{(1)} (x,y) &= \left[ (1+ \alpha^2) x + \alpha y \right] \mathbbm{1}_{\mathcal{M}_1} (x,y) + \left[ (1 - \alpha^2) x + \alpha y \right] \mathbbm{1}_{\mathcal{M}_2} (x,y) \nonumber \\
    &\quad + \left[ (1 - \alpha^2) x - \alpha y \right] \mathbbm{1}_{\mathcal{M}_3} (x,y) + \left[ (1+ \alpha^2) x - \alpha y \right] \mathbbm{1}_{\mathcal{M}_4} (x,y) \pmod{1},\label{eq:Talpha1}
\end{align}
where the four forward smoothness regions $\mathcal{M}_\ell \subset \mathbb{T}^2$ are defined as
\begin{alignat*}{2}
    \mathcal{M}_1 &= \{ x \geq 1/2, \, y + \alpha x \geq 1/2 \pmod{1} \}, &\qquad \mathcal{M}_2 &= \{ x < 1/2, \, y - \alpha x \geq 1/2 \pmod{1} \}, \\
    \mathcal{M}_3 &= \{ x \geq 1/2, \, y + \alpha x < 1/2 \pmod{1} \}, &\qquad \mathcal{M}_4 &= \{ x < 1/2, \, y - \alpha x < 1/2 \pmod{1} \}.
\end{alignat*}
Consequently, the map $\XX_\alpha$ is a piecewise linear automorphism of the torus. On each region $\mathcal{M}_\ell$, the map acts as $\XX_\alpha(x,y) = A_\ell \binom{x}{y} \pmod{1}$, where the matrices $A_\ell \in SL(2, \mathbb{Z})$ are given by:
\begin{equation}\label{eq:matrixMi}
A_1 = \begin{pmatrix} 1+\alpha^2 & \alpha \\ \alpha & 1 \end{pmatrix}, \quad A_2 = \begin{pmatrix} 1-\alpha^2 & \alpha \\ -\alpha & 1 \end{pmatrix}, \quad A_3 = \begin{pmatrix} 1-\alpha^2 & -\alpha \\ \alpha & 1 \end{pmatrix}, \quad A_4 = \begin{pmatrix} 1+\alpha^2 & -\alpha \\ -\alpha & 1 \end{pmatrix}.
\end{equation}
This characterization confirms that the smoothness region of $\XX_\alpha$ is the union of these four components, each corresponding to a specific combination of the directional shears.

\subsection*{The backward map} 
Following a computation analogous to the forward map, we determine the inverse map $\XX_\alpha^{-1} = T_{V_\alpha}^{-1} \circ T_{H_\alpha}^{-1}$. The first component of the inverse is determined by the horizontal inverse shear acting on the $y$-coordinate:
\begin{equation}
(\XX_\alpha^{-1})^{(1)} (x,y) = (x - \alpha y) \mathbbm{1}_{\geq 1/2}(y) + (x + \alpha y) \mathbbm{1}_{< 1/2}(y) \pmod{1},
\end{equation}
and the second component, which accounts for the vertical inverse shear acting on the result of the horizontal shear, is given by:
\begin{align}
    (\XX_\alpha^{-1})^{(2)} (x,y) &= \left[ (1+\alpha^2) y - \alpha x \right] \mathbbm{1}_{\XX_\alpha(\mathcal{M}_1)} (x,y) + \left[ (1 - \alpha^2) y + \alpha x \right] \mathbbm{1}_{\XX_\alpha(\mathcal{M}_2)} (x,y) \nonumber \\
    &\quad + \left[ (1 - \alpha^2) y - \alpha x \right] \mathbbm{1}_{\XX_\alpha(\mathcal{M}_3)} (x,y) + \left[ (1+\alpha^2) y + \alpha x \right] \mathbbm{1}_{\XX_\alpha(\mathcal{M}_4)} (x,y) \pmod{1}.
\end{align}
The \emph{backward smoothness regions} $\XX_\alpha(\mathcal{M}_\ell)$ partition the torus according to the continuity of the inverse map:
\begin{alignat*}{2}
    \XX_\alpha(\mathcal{M}_1) &= \{ y \geq 1/2, \, x - \alpha y \geq 1/2 \pmod{1} \}, &\quad \XX_\alpha(\mathcal{M}_2) &= \{ y \geq 1/2, \, x - \alpha y < 1/2 \pmod{1} \}, \\
    \XX_\alpha(\mathcal{M}_3) &= \{ y < 1/2, \, x + \alpha y \geq 1/2 \pmod{1} \}, &\quad \XX_\alpha(\mathcal{M}_4) &= \{ y < 1/2, \, x + \alpha y < 1/2 \pmod{1} \}.
\end{alignat*}
The structure of these backward smoothness sets is critical, as they define the regions where the vector-valued transfer operator  remains regular. Just as the forward regions $\mathcal{M}_\ell$ are skewed vertically, these backward regions are skewed horizontally, reflecting the time-reversal symmetry of the shear operations. We refer to Figure \ref{figure1} to visualize these sets.

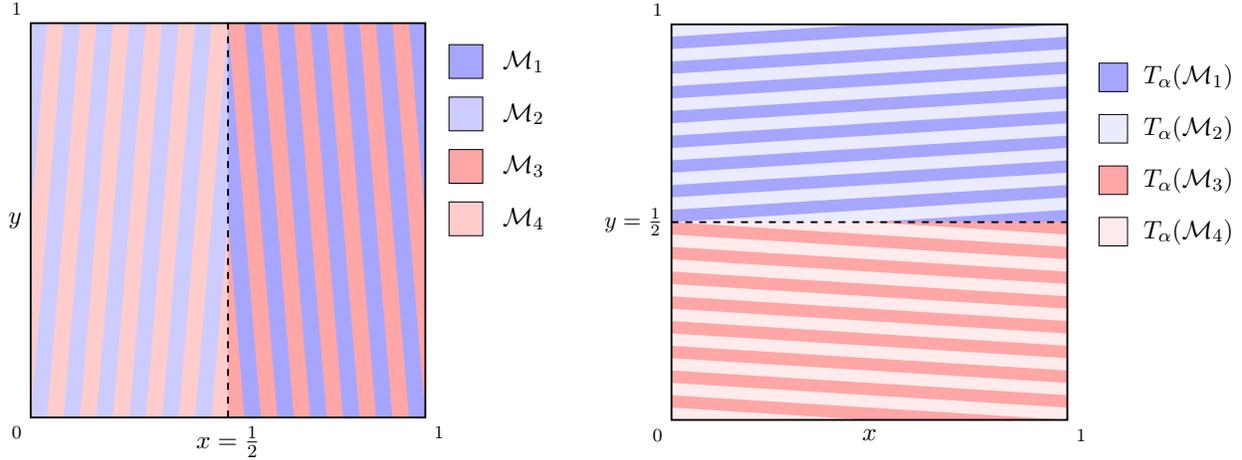
\begin{figure}[ht]
\centering

\begin{minipage}{0.445\textwidth}
\vspace{0.2cm}
\centering
\resizebox{\linewidth}{!}{%
\begin{tikzpicture}[scale=6]

% Parameter alpha
\def\A{12}

% Colors
\colorlet{Mone}{blue!35}   
\colorlet{Mtwo}{blue!20}   
\colorlet{Mthree}{red!35}  
\colorlet{Mfour}{red!20}

% ==========================================================
% Partition of [0,1]^2 into M1,...,M4
% ==========================================================

% Left half: x < 1/2  (M4 base + M2 stripes)
\begin{scope}
  \clip (0,0) rectangle (0.5,1);

  % base region M4
  \fill[Mfour] (0,0) rectangle (0.5,1);

  % stripes for M2: y - A x in [k+1/2, k+1)
  \foreach \k in {-20,...,20} {
    \pgfmathsetmacro{\ca}{\k + 0.5}
    \pgfmathsetmacro{\cb}{\k + 1.0}
    \fill[Mtwo]
      (0,\ca) --
      (1,{\A + \ca}) --
      (1,{\A + \cb}) --
      (0,\cb) -- cycle;
  }
\end{scope}

% Right half: x >= 1/2  (M3 base + M1 stripes)
\begin{scope}
  \clip (0.5,0) rectangle (1,1);

  % base region M3
  \fill[Mthree] (0.5,0) rectangle (1,1);

  % stripes for M1: y + A x in [k+1/2, k+1)
  \foreach \k in {-20,...,40} {
    \pgfmathsetmacro{\ca}{\k + 0.5}
    \pgfmathsetmacro{\cb}{\k + 1.0}
    \fill[Mone]
      (0,\ca) --
      (1,{\ca - \A}) --
      (1,{\cb - \A}) --
      (0,\cb) -- cycle;
  }
\end{scope}

% Frame and guides
\draw[thick] (0,0) rectangle (1,1);
\draw[thick,dashed] (0.5,0) -- (0.5,1);

% Axes labels
\node[below] at (0.5,0) {$x=\frac{1}{2}$};
\node[left]  at (0,0.5) {$y$};

% Tick labels
\node[below left]  at (0,0) {\scriptsize $0$};
\node[below right] at (1,0) {\scriptsize $1$};
\node[above left]  at (0,1) {\scriptsize $1$};
%\node[below]       at (0.5,0) [yshift=-10pt] {\scriptsize $\frac12$};

% Legend
\begin{scope}[shift={(1.06,0.86)}, scale=0.95]
  \fill[Mone] (0,0) rectangle (0.09,0.09);
  \draw (0,0) rectangle (0.09,0.09);
  \node[right] at (0.12,0.045) {$\mathcal M_1$};

  \fill[Mtwo] (0,-0.14) rectangle (0.09,-0.05);
  \draw (0,-0.14) rectangle (0.09,-0.05);
  \node[right] at (0.12,-0.095) {$\mathcal M_2$};

  \fill[Mthree] (0,-0.28) rectangle (0.09,-0.19);
  \draw (0,-0.28) rectangle (0.09,-0.19);
  \node[right] at (0.12,-0.235) {$\mathcal M_3$};

  \fill[Mfour] (0,-0.42) rectangle (0.09,-0.33);
  \draw (0,-0.42) rectangle (0.09,-0.33);
  \node[right] at (0.12,-0.375) {$\mathcal M_4$};
\end{scope}

\end{tikzpicture}%
}
\end{minipage}
\hfill
\begin{minipage}{0.52\textwidth}
\centering
\resizebox{\linewidth}{!}{%
\begin{tikzpicture}[scale=6]

% Large alpha (example)
\def\A{16}

% Colors
\colorlet{Mone}{blue!35}
\colorlet{Mtwo}{blue!8}
\colorlet{Mthree}{red!35}
\colorlet{Mfour}{red!8}

% ==========================================================
% TOP HALF: y >= 1/2
% Centered at y = 1/2  => use x - A(y-1/2) mod 1
% M1: >= 1/2 mod 1   (dark blue)
% M2: <  1/2 mod 1   (light blue)
% ==========================================================
\begin{scope}
  \clip (0,0.5) rectangle (1,1);
  \fill[Mtwo] (0,0.5) rectangle (1,1);

  \foreach \k in {-60,...,60} {
    \pgfmathsetmacro{\ca}{\k + 0.5}
    \pgfmathsetmacro{\cb}{\k + 1.0}
    \fill[Mone]
      ({\A*(-1.5) + \ca},-1) --
      ({\A*( 1.5) + \ca}, 2) --
      ({\A*( 1.5) + \cb}, 2) --
      ({\A*(-1.5) + \cb},-1) -- cycle;
  }
\end{scope}

% ==========================================================
% BOTTOM HALF: y < 1/2
% Centered at y = 1/2  => use x + A(y-1/2) mod 1
% M3: >= 1/2 mod 1   (dark red)
% M4: <  1/2 mod 1   (light red)
% ==========================================================
\begin{scope}
  \clip (0,0) rectangle (1,0.5);
  \fill[Mfour] (0,0) rectangle (1,0.5);

  \foreach \k in {-60,...,60} {
    \pgfmathsetmacro{\ca}{\k + 0.5}
    \pgfmathsetmacro{\cb}{\k + 1.0}
    \fill[Mthree]
      ({-\A*(-1.5) + \ca},-1) --
      ({-\A*( 1.5) + \ca}, 2) --
      ({-\A*( 1.5) + \cb}, 2) --
      ({-\A*(-1.5) + \cb},-1) -- cycle;
  }
\end{scope}

% Main square and split line y = 1/2
\draw[thick] (0,0) rectangle (1,1);
\draw[thick,dashed] (0,0.5) -- (1,0.5);

% Axis labels
\node[below] at (0.5,0) {$x$};

% Mark y=1/2
\node[left] at (0,0.5) {\small $y=\frac12$};

% Corner labels
\node[below left]  at (0,0) {\scriptsize $0$};
\node[below right] at (1,0) {\scriptsize $1$};
\node[above left]  at (0,1) {\scriptsize $1$};

% Legend
\begin{scope}[shift={(1.08,0.83)}, scale=0.9]
  \fill[Mone] (0,0) rectangle (0.08,0.08);
  \draw (0,0) rectangle (0.08,0.08);
  \node[right] at (0.10,0.04) {$\XX_\alpha(\mathcal M_1)$};
\end{scope}

\begin{scope}[shift={(1.08,0.70)}, scale=0.9]
  \fill[Mtwo] (0,0) rectangle (0.08,0.08);
  \draw (0,0) rectangle (0.08,0.08);
  \node[right] at (0.10,0.04) {$\XX_\alpha(\mathcal M_2)$};
\end{scope}

\begin{scope}[shift={(1.08,0.57)}, scale=0.9]
  \fill[Mthree] (0,0) rectangle (0.08,0.08);
  \draw (0,0) rectangle (0.08,0.08);
  \node[right] at (0.10,0.04) {$\XX_\alpha(\mathcal M_3)$};
\end{scope}

\begin{scope}[shift={(1.08,0.44)}, scale=0.9]
  \fill[Mfour] (0,0) rectangle (0.08,0.08);
  \draw (0,0) rectangle (0.08,0.08);
  \node[right] at (0.10,0.04) {$\XX_\alpha(\mathcal M_4)$};
\end{scope}

\end{tikzpicture}%
}
\end{minipage}

\caption{The partition of $[0,1]^2$ into the regions $\mathcal M_1,\mathcal M_2,\mathcal M_3,\mathcal M_4$ and the sets $\XX_\alpha(\mathcal M_1), \XX_\alpha(\mathcal M_2), \XX_\alpha(\mathcal M_3), \XX_\alpha(\mathcal M_4)$, with $\alpha=16$.}
\label{figure1}
\end{figure}

Finally, we define the forward singular set 
\begin{align} \label{d:forward-singular}
S^\star=\bigcup_{\ell=1}^4 \partial T_\alpha(\mathcal{M}_\ell)\,,
\end{align}
namely the image of the partition boundaries under the map \(T_\alpha\), across which the local hyperbolic structure changes discontinuously.
\subsection{The hyperbolic structure.}
The crucial property of the map $\XX_\alpha$ is that it is \emph{uniformly hyperbolic}. Notably, the hyperbolic structure is significantly more rigid than the general case outlined in \cite{demers_liverani}; here, the unstable and stable cone fields are independent of the spatial coordinates $(x,y)$. The following lemma summarizes the spectral and geometric properties required for our analysis.
\begin{lemma}
\label{lemma:uniformly hyperbolic}
There exists $\alpha_0 > 0$ such that for all $\alpha \geq \alpha_0$, there exist two constant cones
\begin{equation}\label{eq:cones}
C_s = \{ \mathbf{v}=(v_1,v_2) : |v_1| \leq 2\alpha^{-1}|v_2| \}, \qquad C_u = \{ \mathbf{v}=(v_1,v_2) : |v_2| \leq 2\alpha^{-1}|v_1| \}
\end{equation}
satisfying the invariance properties
\begin{equation}
D\XX_\alpha C_u \subset C_u, \quad D\XX_\alpha^{-1} C_s \subset C_s.
\end{equation}
Furthermore, we have the expansion estimates
\begin{equation}
|D\XX_\alpha \mathbf{v}| \geq \frac{1}{2}\alpha^{2} |\mathbf{v}|, \quad \forall \mathbf{v} \in C_u, \qquad  |D\XX_\alpha^{-1} \mathbf{v}| \geq \frac{1}{2}\alpha^{2} |\mathbf{v}|, \quad \forall \mathbf{v} \in C_s \,.
\end{equation}
\end{lemma}

\begin{proof}
We provide the details for the forward map, as the argument for the inverse is analogous. 
Recall that on each forward smoothness region $\mathcal{M}_\ell$, the derivative $D\XX_\alpha$ is given by one of the matrices $A_1,\dots, A_4$ derived in \eqref{eq:matrixMi}.

We demonstrate the argument for $A_1 = \begin{pmatrix} 1+\alpha^2 & \alpha \\ \alpha & 1 \end{pmatrix}$; the remaining cases follow by identical estimates. Let $\mathbf{v} = (v_1, v_2) \in C_u$. By definition, $|v_2| \leq 2\alpha^{-1}|v_1|$. The image vector $\mathbf{v}' = A_1 \mathbf{v}$ has components
\begin{align*}
|v_1'| &= |(1+\alpha^2)v_1 + \alpha v_2| \geq (1+\alpha^2)|v_1| - \alpha |v_2| \geq (1+\alpha^2)|v_1| - 2|v_1| = (\alpha^2 - 1)|v_1|, \\
|v_2'| &= |\alpha v_1 + v_2| \leq \alpha |v_1| + |v_2| \leq (\alpha + 2\alpha^{-1})|v_1|.
\end{align*}
To satisfy $D\XX_\alpha C_u \subset C_u$, we require $|v_2'| \leq 2\alpha^{-1}|v_1'|$. Substituting the bounds
\begin{equation}
(\alpha + 2\alpha^{-1})|v_1| \leq 2\alpha^{-1}(\alpha^2 - 1)|v_1| = (2\alpha - 2\alpha^{-1})|v_1|.
\end{equation}
For large $\alpha$, this simplifies to $\alpha \leq 2\alpha$, which is clearly satisfied. For the expansion estimate, since $|v| \approx |v_1|$ for vectors in $C_u$, we have:
\begin{equation}
|D\XX_\alpha \mathbf{v}| \geq |v_1'| \geq (\alpha^2 - 1)|v_1| \geq \frac{1}{2}\alpha^2 |\mathbf{v}|\,, \qquad \forall \alpha \geq \alpha_0 \,,
\end{equation}
concluding the proof.
\end{proof}
The existence of these constant cones facilitates the construction of \emph{anisotropic Banach spaces}. Because the expansion factor $ \approx \alpha^2$ grows quadratically with the shear strength, the essential spectral radius of the vector-valued transfer operator can be made arbitrarily small relative to the dominant eigenvalue as $\alpha \to \infty$.

\section{Functional setting and anisotropic Banach spaces} \label{sec:anisotropic}

Having outlined the hyperbolic structure of the map, we now define suitable anisotropic Banach spaces. These are essentially those of \cite{demers_liverani}, with adaptations made to account for our piecewise affine map structure.\footnote{These adaptations are, in fact, essential in order to treat the Laplacian as a perturbation. The class of admissible perturbations in \cite{demers_liverani} \emph{does not} natively include parabolic smoothing.}

We begin by defining a class of admissible leaves against which we shall integrate. Recall from Lemma \ref{lemma:uniformly hyperbolic} that our dynamics admit constant stable/unstable cones given in \eqref{eq:cones}, which are invariant under the backward and forward dynamics of $\XX_\alpha$, respectively. For any basepoint $\mathbf{x} \in \mathbb{T}^2$, unit direction vector  $\mathbf{v} \in \hat{C}_s = \{\vvv \in C_s : |\vvv|=1 \,, \vvv^{(2)}>0 \}$, and length $S \in (0,1]$, we define a parameterized curve $\gamma_{\mathbf{x},\mathbf{v},S}(t) = \mathbf{x} + t\mathbf{v}$ for $t \in [0,S]$. This yields a line segment $W \subset \mathbb{T}^2$ given by $W = \gamma_{\mathbf{x},\mathbf{v},S}([0,S])$.

We first verify that such lines admit a canonical parameterisation.

\begin{lemma}
\label{lemma:unique parameterisation}
Let $W\subset \mathbb{T}^2 \cong [0,1]^2/\sim$ be a line segment of length at most $1$, with tangent direction in $C_s$. If $W$ is not a vertical segment of length $1$, then there exist unique $\mathbf{x}\in\mathbb{T}^2$, $\mathbf{v}\in \hat{C}_s$, and $S\in[0,1]$ such that $
W=\gamma_{\mathbf{x},\mathbf{v},S}([0,S]).$
If $W$ is a vertical line segment of length $1$, we fix the convention that its base point is $\xx = (x_0,0)$ for some $x_0\in\mathbb{T}$; under this convention, the above parameterisation is unique as well.
\end{lemma}

\begin{proof}
Assume that $W$ is not a vertical line segment of length $1$, then it cannot wrap around the torus and self-intersect; therefore, the set $W$ is a simple line segment with exactly two distinct endpoints on $\TT^2$. Since any tangent vector $\mathbf{v} \in \hat{C}_s$ must satisfy $|\mathbf{v}^{(1)}| \leq 2\alpha^{-1}|\mathbf{v}^{(2)}|$, its vertical component cannot be zero. Consequently, between the two possible unit vectors parallel to $W$, exactly one satisfies $\mathbf{v}^{(2)} > 0$. Choosing this orientation uniquely determines the starting endpoint $\mathbf{x} \in \mathbb{T}^2$, the direction $\mathbf{v}$, and the arc length $S \in (0,1)$, completing the proof.
If $W$ is a vertical line segment of length $1$, it is straightforward to see that  the convention restores a unique parameterisation.
\end{proof}

\subsection*{The class of admissible leaves}
Having established a canonical parameterisation for segments with tangents in $\hat{C}_s$, we may now define a metric space of admissible leaves. Indeed, define first the space $\Lambda$ to be 
$$
\Lambda=\mathbb{T}^2 \times \hat C_s \times [0,1] ,
$$
with metric 
$$
\dd_\Lambda((\mathbf{x_1},\mathbf{v}_1,S_1),(\mathbf{x_2},\mathbf{v}_2,S_2))=\dd_{\mathbb{T}^2}(\mathbf{x}_1,\mathbf{x}_2)+|\mathbf{v}_1-\mathbf{v}_2|+|S_1-S_2|.
$$
Clearly the space $(\Lambda,\dd_\Lambda)$ is a compact metric space. We then define $ \Gamma$ to be a subset of $\Lambda$ given by the union of non-vertical line segments of length at most one, perfectly vertical line segments of length strictly less than one, and finally perfectly vertical line segments of length exactly $1$
$$
\Gamma=\left (\mathbb{T}^2\times \hat{C}_s\setminus \{(0,1)\}  \times (0,1] \right ) \cup \left (\mathbb{T}^2  \times \{(0,1)\} \times (0,1)\right ) \cup \left ( (\mathbb{T} \times \{0\} )  \times \{(0,1)\}\times  \{1\}\right ).
$$
In particular, $(\Gamma, \dd_{\Lambda})$ is a pre-compact metric space.
We now define the class of admissible stable leaves as:
\begin{equation} \label{d:Sigma}
    \Sigma = \Sigma_s = \{ W \subset \mathbb{T}^2 : \exists \mathbf{v}\in \hat{C}_s, \, \exists \mathbf{x} \in \mathbb{T}^2, \, \exists S \in (0,1] \text{ such that } W = \gamma_{\mathbf{x}, \mathbf{v},S} ([0,S]) \}.
\end{equation}

Now, by Lemma \ref{lemma:unique parameterisation}, we note that the map $\Psi: \Gamma \to \Sigma$ given by $\Psi(\mathbf{x},\mathbf{v},S)=\gamma_{\mathbf{x},\mathbf{v},S}([0,S])$ is a bijection. Therefore, we may induce a metric on $\Sigma$ via the map $\Psi$, given by 
$$
\dd_\Sigma(W_1,W_2)=\dd_{\Lambda}(\Psi^{-1}(W_1),\Psi^{-1}(W_2)).
$$
Under this metric, $\Psi$ becomes a homeomorphism, and thus the metric space $(\Sigma,\dd_\Sigma)$ inherits the pre-compactness property of $(\Gamma,\dd_\Lambda)$, which will be of use in Lemma \ref{lemma:compact}.
From now on, we abuse notation and identify elements in $\Sigma$ with their parameterisation. Thus, for $W_1 = \gamma_{\mathbf{x}_1, \mathbf{v}_1, S_1} ([0,S_1])$ and $W_2 = \gamma_{\mathbf{x}_2, \mathbf{v}_2,S_2} ([0,S_2])$, we have 
\begin{equation} \label{d:d-Sigma}
    \dd_\Sigma (W_1, W_2) = \dd_{\mathbb{T}^2}(\mathbf{x}_1,\mathbf{x}_2) + |\mathbf{v}_1 - \mathbf{v}_2| + |S_1 - S_2|,
\end{equation}
where $\dd_{\mathbb{T}^2}(\mathbf{x}_1,\mathbf{x}_2)$ is the standard toroidal distance.

\medskip

We now define a pseudo-metric on test functions.
From now on, we abuse notation and use $C^1(\TT^2; \CC)$ as the space of Lipschitz functions.
We define a distance between test functions defined on these leaves. For any $q \in (0,1]$,  $W \in \Sigma$, and $\varphi \in C^1(\mathbb{T}^2, \mathbb{C})$, we define the Hölder seminorm and norm along $W$ as:
\begin{equation}
    [\varphi]_{C^q(W)} = \sup_{\substack{x, y \in W \\ x \neq y}} \frac{|\varphi(x) - \varphi(y)|}{|x-y|^q}, \qquad \|\varphi\|_{C^q(W)} = \sup_{y \in W} |\varphi(y)| + [\varphi]_{C^q(W)},
\end{equation}
and define, for $\sigma\in (0,1)$, a geometrically weighted norm 
$$\|\varphi\|_{\sigma, W} = \|\varphi\|_{C^q(W)} |W|^\sigma.$$ 
We denote by $C^q(W)$ the closure of the space of Lipschitz continuous functions, $C^1(\mathbb{T}^2, \mathbb{C})$, with respect to the $\|\cdot\|_{C^q(W)}$ norm.\footnote{We remark that $C^q(W)$ defined via this closure is smaller than the full standard Hölder space when $q \in (0,1)$.}

To compare functions on different leaves, let $W_i = \gamma_{\mathbf{x}_i, \mathbf{v}_i,S_i} ([0,S_i]) \in \Sigma$ for $i=1,2$. For test functions $\varphi_i \in C^q(W_i)$, we define the pseudo metric
\begin{equation}
    \dd_q (\varphi_1, \varphi_2) = \left \| \varphi_1(\gamma_{\mathbf{x}_1, \mathbf{v}_1,S_1}(\cdot)) - \varphi_2(\gamma_{\mathbf{x}_2, \mathbf{v}_2,S_2}(\cdot)) \right \|_{C^q([0, S_1 \wedge S_2])}.
\end{equation}

\subsection*{Definitions of the norms}
We fix positive parameters $\sigma, \beta, q \in (0,1)$ satisfying the strict inequalities:
\begin{equation} \label{eq:choice-parameters}
     \sigma > \beta, \quad 1-q > \beta.
\end{equation}
For any vector valued function $f \in C^1(\mathbb{T}^2; \mathbb{C}^2)$, we define the \emph{weak stable norm} by
\begin{equation}\label{eq:weak-norm}
    |f|_w := \sup_{W \in \Sigma} \sup_{\|\varphi \|_{C^1(W)} \leq 1} \left| \int_W f \varphi \, \dd \HH^1 \right|\,,
\end{equation}
where  we denote by $|\cdot|$ the euclidean norm on $\CC^2$.
Next, we define the \emph{strong stable norm} as
\begin{equation}\label{eq:strong-stable}
    \| f \|_s := \sup_{W \in \Sigma} \sup_{\|\varphi\|_{\sigma, W} \leq 1} \left| \int_W f \varphi \, \dd \HH^1\right|,
\end{equation}
and the \emph{strong unstable norm}, which measures the regularity of the distributions across stable leaves 
\begin{equation}\label{eq:strong-unstable}
    \|f\|_u := \sup_{0 < \delta  \leq \frac{2}{\alpha}}\, \, \sup_{\dd_\Sigma(W^1, W^2) < \delta} \, \,  \sup_{\substack{\|\varphi_1 \|_{C^1},\|\varphi_2 \|_{C^1} \leq 1 \\ \dd_q(\varphi_1, \varphi_2) < \delta}} \frac{1}{\delta^\beta} \left| \int_{W^1} f \varphi_1 \, \dd \HH^1 - \int_{W^2} f \varphi_2 \, \dd \HH^1 \right|.
\end{equation}
Finally, for any $h \in C^1(\mathbb{T}^2, \mathbb{C}^2)$, the full \emph{strong norm} is given by 
\begin{equation}\label{eq:Xstrongnorm}
    \|h\| = \|h\|_s + \|h\|_u.
\end{equation}
We define the Banach space   $X$  as the completion of $C^1(\mathbb{T}^2; \mathbb{C}^2)$ with respect to the strong norm $\|\cdot\|$, and the weak space  $X_w$  as the completion with respect to $|\cdot|_w$.

\begin{remark}
    We emphasize that these norms inherently depend on the shear parameter $\alpha \geq 1$.
\end{remark}

\subsubsection*{On the choice of anisotropic Banach spaces}
While all constructions of anisotropic Banach spaces in the literature follow the philosophy of enforcing ``negative regularity in the contracting direction, positive regularity in the expanding direction of the map'', they are by no means a monolith. Over the last few decades, there has been a wide proliferation of different anisotropic spaces, each adapted to particular dynamical systems and coming with distinct advantages and drawbacks. For a comprehensive overview of these spaces, we refer the reader to the excellent survey by Baladi \cite{Baladi_quest}. 

The vast majority of spaces in the literature were constructed to handle relatively smooth transfer operators, such as those arising from smooth Anosov flows. For instance, the microlocal anisotropic spaces utilized in \cites{BaladiTsujii2007AnisotropicHolderSobolev,BaladiTsujii2008DynamicalDeterminantsSpectrum,zworski_2015,DZ16} rely heavily on the high regularity of the underlying map. Even spaces designed to accommodate piecewise smooth maps, such as \cite{BaladiGouezel2009GoodBanachPiecewiseHyperbolic}, often assume some regularity on the stable foliation, which is absent in our setting.

Therefore, the functional spaces adapted to our problem must be flexible enough to handle piecewise smooth transfer operators with merely measurable stable foliations. Following the taxonomy in \cite{Baladi_quest}, candidates with these properties include the ``Triebel'' space approach of \cite{BaladiGouezel2010PiecewiseConeHyperbolic}, the geometric spaces of Demers and Liverani \cite{demers_liverani}, and the spaces introduced in \cite{Baladi_quest}. While it is likely that any of these frameworks could be adapted to our arguments, the present article employs spaces closely modeled on the geometric approach of \cite{demers_liverani}. The simple piecewise-linear geometry of our map makes the analysis of the strong-chaos limit ($\alpha \to \infty$) in these spaces significantly more tractable.

\subsection{Compact embeddings}
Having defined the weak and strong norms on our Banach space, we now establish the fundamental functional-analytic properties required for spectral analysis. Specifically, the spectral gap of the dynamo operator is contingent upon the Lasota-Yorke inequalities. We seek to show that the pair $(X_w, |\cdot|_w)$ and $(X, \|\cdot\|)$ satisfy the following:

\begin{enumerate}
    \item \textit{Norm ordering:} $|f|_w \leq \|f\|$ for all $f \in X$.
    \item \textit{Compact embedding:} The unit ball $\{f \in X : \|f\| \leq 1\}$ is relatively compact in $(X_w, |\cdot|_w)$.
\end{enumerate}
The first property is immediate from the definition of the norms, as the test function space for the weak norm is a subset of that for the strong norm (with a stronger topology). We now prove the second property.

\begin{lemma} \label{lemma:compact}
The unit ball of $X$ embeds compactly into $X_w$.
\end{lemma}

\begin{proof}
Let $\mathcal{B}_1 = \{h \in X : \|h\| \leq 1\}$ and fix $\eps > 0$. Since $(\Sigma,\dd_\Sigma)$ is pre-compact, it is totally bounded. Hence, there exists a finite $\eps$-net of leaves $\{(\mathbf{x}_i, \mathbf{v}_i, S_i)\}_{i=1}^n = \{ W_i \}_{i=1}^n$.
For each leaf $W_i$ in our net, the space of test functions $\{\varphi \in C^1(W_i) : \|\varphi\|_{C^1} \leq 1\}$ is compactly embedded in $C^q(W_i)$ by the Arzel\`a-Ascoli theorem since $q\in (0,1)$. Thus, we can select a finite $\eps$-net of test functions $\{\varphi_{i,j}\}_{j=1}^{n_i}$ for each $W_i$. In other words, for any $\varphi \in C^1 (W_i)$ with $\|\varphi\|_{C^1(W_i)}\leq 1$, there exists some $\varphi_{i,j}$ so that $\|\varphi-\varphi_{i,j}\|_{C^q(W_i)} <\eps$.

Let $h \in \mathcal{B}_1$. For any $W \in \Sigma$ and $\varphi \in C^1(W)$ with $\|\varphi\|_{C^1} \leq 1$, we estimate $\int_W h\varphi$ by comparing it to an integral over some $W_i$ from our net $\{ W_i\}_{i=1}^n$. We parametrize $W, W_i$ respectively by 
$$
x+t \mathbf{v},  \quad  t \leq S; \qquad x_i +t \mathbf{v}_i, \quad t \leq S_i.
$$
hence we can define $\tilde \varphi \in C^1 (W_i)$
$$\tilde \varphi(x_i+t \mathbf{v}_i)= \begin{cases}
    \varphi(x+t \mathbf{v}) & \forall t \in (0,S\wedge S_1)
    \\
    \varphi(x+S \mathbf{v}) & \forall t \in (S\wedge S_1, S_1) \,,
\end{cases} $$
and we have that the Lipschitz norm of $\tilde \varphi $ satisfies $\| \tilde \varphi \|_{C^1 (W_i)} \leq \|  \varphi \|_{C^1 (W)} \leq 1$ and $\dd_q (\varphi, \tilde \varphi) = 0$.  Picking $\varphi_{i,j}$ from our $\eps$-net of test functions such that $\|\varphi_{i,j} - \tilde{\varphi}\|_{C^q(W_i)} < \eps$, we decompose the error as
\begin{equation}
\int_W h \varphi \, \dd \HH^1 - \int_{W_i} h \varphi_{i,j} \, \dd \HH^1 = \left( \int_W h \varphi \, \dd \HH^1 - \int_{W_i} h \tilde{\varphi} \, \dd \HH^1\right) + \int_{W_i} h (\tilde{\varphi} - \varphi_{i,j})\, \dd \HH^1.
\end{equation}
The first term is bounded by $\eps^\beta \|h\|_u$ via the strong unstable norm. The second term is bounded by $\eps |W_i|^\sigma \|h\|_s$. Combining these, we find that for any $\eps > 0$, there exists a finite collection of linear functionals $\ell_{i,j}(h) = \int_{W_i} h \varphi_{i,j}:(X,\|\cdot \|) \to \CC^2$ such that
\begin{equation}
|h|_w \leq \max_{i,j} |\ell_{i,j}(h)| + C(\eps^\beta + \eps) \|h\|.
\end{equation}
Since $C(\eps^\beta + \eps) \to 0$ as $\eps \to 0$, a standard diagonal argument shows that any sequence in $\mathcal{B}_1$ has a subsequence that is Cauchy in $X_w$, completing the proof.
\end{proof}

\subsection{Some technical lemmas}
In this section, we collect a few results on our anisotropic Banach spaces and admissible curves, which we shall need throughout the manuscript.
We begin with a lemma, which shows that indicator functions of half-planes live in our spaces, and have uniformly bounded norms in $\alpha$. Whilst somewhat involved, the proof is in effect a simpler version of the matching of curves we shall do in more generality later.
\begin{lemma}
\label{lemma:indicators-y}
Let $\www \in \CC^2$ be a normalised vector $|\www|=1$ and $\mathds{1}_{y \geq 1/2}$ be the periodised indicator function of $\{y \geq 1/2 \}$. Then the following hold true
$$
\|\www \mathds{1}_{y \geq 1/2}\|_s\leq C \,, \qquad \| \www \mathds{1}_{y \geq 1/2} \|_u \leq C \alpha^{-1 + \beta } \,.
$$
Furthermore, for any $h \in X$ it holds that 
$$ \left |\int_{x\geq 1/2} h \,  \right | \leq \frac{1}{2} |h|_w \,. $$
\end{lemma}
\begin{proof}
We shall prove the result for $\www = (1,0)$ for the sake of convenience.
We estimate the strong stable norm. Let $W \in \Sigma$ be an admissible curve, then  there exist at most two connected components $W_1,W_2$ so that $W \cap \{y \geq 1/2\}=W_1 \cup W_2$. Then, for $\varphi \in C^q(W)$ with $\|\varphi\|_{C^q(W)} \leq |W|^{-\gamma}$, we have 
$$
\int_{W } \mathds{1}_{y \geq 1/2} \varphi \, \dd \HH^1 =\int_{W_1} \varphi \, \dd \HH^1 +\int_{W_2} \varphi \, \dd \HH^1 \leq (|W_1|+|W_2| )\|\varphi\|_{C^q} \leq 2|W|^{1-\gamma} \leq 2 \,.
$$
We estimate the strong unstable norm. We consider two admissible curves $W_1, W_2 \in \Sigma$, with $\dd_\Sigma (W_1, W_2) < \delta < 1/\alpha$ and $\varphi_1 \in C^1(W_1)$, $\varphi_2 \in C^1 (W_2)$ with $\dd_q (\varphi_1, \varphi_2) < \delta$ and $\|\varphi_1\|_{C^1 (W_1)}\leq 1$, $ \| \varphi_2\|_{C^1 (W_2)}\leq 1$. Since we are on $\TT^2$ there exist at most two connected curves $(W_{1,j})_{j=1,2}$ and $(W_{2,j})_{j=1,2}$ so that $\cup_{j=1,2} W_{1,j} = W_1 \cap \{ y \geq 1/2 \}$ and $\cup_{j=1,2} W_{2,j} = W_2 \cap \{ y \geq 1/2 \}$. A computation similar to that in Lemma \ref{lemma:technical} implies that 
$$ \dd_\Sigma (W_{1,j}, W_{2,j}) \leq C \dd_{\Sigma} (W_1, W_2),$$
uniformly in $\alpha$.
Thus, we estimate 
\begin{align*}
\int_{W_1 \cap \{y \geq 1/2\}} \varphi_1\, \dd \HH^1 -\int_{W_2 \cap \{y \geq 1/2\}} \varphi_2\, \dd \HH^1 = \sum_{j=1}^2 \int_{W_{1,j}} \varphi_1 \, \dd \HH^1 -\int_{W_{2,j}} \varphi_2 \, \dd \HH^1 \,.
\end{align*}
Choosing the parameterisations $W_{i,j}$ as $\gamma_{i,j} (t) = \xx_{i,j} + t \vvv_{i,j}$ with $t \in (0, S_{i,j})$ for $i,j=1,2$ we deduce, for e.g. comparing  $W_{1,1}$ and $W_{2,1}$ and assuming that $S_{1,1} \leq S_{2,1}$ we have 
\begin{align*}
&\int_{W_{1,1}} \varphi_1 \, \dd \HH^1 -\int_{W_{2,1}} \varphi_2 \, \dd \HH^1 =\int_{0}^{S_{1,1}} \varphi_1(\mathbf{x}_{1,1}+t \mathbf{v}_{1,1}) \dd t-\int_{0}^{S_{2,1}} \varphi_2(\mathbf{x}_{2,1}+t \mathbf{v}_{2,1}) \dd t\\
& \qquad= \int_{0}^{S_{1,1}} \varphi_1(\mathbf{x}_{1,1}+t \mathbf{v}_{1,1}) \dd t- \int_{0}^{S_{1,1}} \varphi_2(\mathbf{x}_{2,1}+t \mathbf{v}_{2,1}) \dd t
+ \int_{S_{1,1}}^{S_{2,1}} \varphi_2(\mathbf{x}_{2,1}+t \mathbf{v}_{2,1}) \dd t
\\
& \qquad\leq \dd_q (\varphi_1, \varphi_2) S_{1,1} + |S_{2,1} - S_{1,1}| \| \varphi_2 \|_{C^1 (W_{2,1})} \leq  \dd_q (\varphi_1, \varphi_2) + \dd_{\Sigma} (W_{1,1} , W_{2,1}) \leq 2 \delta  \,.
\end{align*}
Since the argument works the same for $W_{1,2}$ and $W_{2,2}$ we deduce that
$
\left | \int_{W_1} \mathds{1}_{y \geq 1/2} \varphi_1 -\int_{W_2}\mathds{1}_{y \geq 1/2} \right | \leq 4 \delta \,.
$
Hence, dividing through by $\delta^\beta$ we obtain an upper bound of order $4\delta^{1-\beta} \leq 4\alpha^{\beta-1}$, showing the claimed estimate for the strong unstable norm.

The proof of the second assertion follows from 
$$ \left | \int_{x\geq 1/2} h \right |\leq  \int_{x\geq 1/2}  \left | \int_{\TT} h(x,y)\dd y \right |\dd x \leq \int_{x\geq 1/2} |h|_w\dd x = \frac{1}{2}|h|_w \,.$$
\end{proof}

Next, we show that multiplication by Lipschitz continuous functions is a bounded operator on our space.
\begin{lemma}
\label{lemma:smooth multiplier}
Let $g \in C^1(\mathbb{T}^2, \CC)$. Then, the following hold true
$$
|g h|_w \leq C\|g\|_{C^1}|h|_w, \quad \|g h\| \leq C \|g\|_{C^1}\|h\| \,, \quad \forall h \in X\,.
$$
\end{lemma}
\begin{proof}
Let $W \in \Sigma$ be an admissible curve, and let $\varphi \in C^1(W)$ with $C^1$ norm at most $1$. Then, we have 
$$
\left |\int_{W} gh \varphi \, \dd \HH^1 \right |=\left |\int_{W} h (g\varphi) \, \dd \HH^1 \right | \leq |h|_w \|g \varphi\|_{C^1}.
$$
Since $\|g \varphi\|_{C^1}\leq 2\|g\|_{C^1}\|\varphi\|_{C^1}$, it follows that $|gh|_w \leq 2\|g\|_{C^1}|h|_w$.

Next, we prove the strong norm bound. As above, we consider $W$ admissible and $\varphi \in C^1 (W)$ with $\|\varphi\|_{\sigma ,W} \leq 1$. Then we estimate
$$
\left |\int_{W} gh\varphi \, \dd \HH^1  \right | \leq \|h\|_{s}\|g\varphi\|_{\sigma,W}.
$$
But now, $\|g \varphi\|_{\sigma, W}=|W|^{\sigma}\|g \varphi\|_{C^q(W)} \leq 2|W|^{\sigma}\|g\|_{C^1}\|\varphi\|_{C^q(W)} \leq 2\|g\|_{C^1}$. Hence, it holds $\|g h\|_{s} \leq 2\|g\|_{C^1}\|h\|_{s}$. 

Finally, we bound the strong unstable norm. We consider $W_1, W_2 \in \Sigma$ with $\dd_\Sigma (W_1,W_2) < \delta$ and $\varphi_1 \in C^1(W_1)$, $ \varphi_2 \in C^1(W_2)$ with $\dd_q(\varphi_1, \varphi_2)< \delta$ and $C^1$ norm bounded by 1. Suppose that the $W_i$ are parameterised respectively by $\mathbf{x}_i +t \mathbf{v}_i$, $t \leq S_i$. Say $S_1 \geq S_2$. Then, define $\widetilde{g\varphi_1} \in C^1(W_2)$ by 
$$
\widetilde{g\varphi_1}(\xx_2+t\mathbf{v}_2)=g\varphi_1(\mathbf{x}_1+t\mathbf{v}_1).
$$
This is a well-defined function, and by construction it holds $\dd_q(g\varphi_1, \widetilde{g \varphi_1})=0$.
Thus, we have 
$$
\int_{W_1} gh \varphi_1 \, \dd \HH^1 -\int_{W_2} gh \varphi_2 \, \dd \HH^1 =\int_{W_1} h (g\varphi_1) \, \dd \HH^1 -\int_{W_2} h (\widetilde{g\varphi_1}) \, \dd \HH^1+\int_{W_2} h (\widetilde{g\varphi_1}-g\varphi_2)\, \dd \HH^1.
$$
Firstly, it follows from the definition of the strong unstable norm that 
$$
\left |\int_{W_1} h (g\varphi_1) \, \dd \HH^1 -\int_{W_2} h (\widetilde{g\varphi_1}) \, \dd \HH^1\right | \leq \delta^{\beta} \|h\|_{u} \max\{\|g \varphi_1\|_{C^1(W_1)},\|\widetilde{g \varphi_1}\|_{C^1(W_2)}\} \leq 2\delta^\beta \|h\|_u \|g\|_{C^1}.
$$
Next, we claim that $\|g \varphi_2-\widetilde{g \varphi_1}\|_{C^q(W_2)} \leq 12 \|g\|_{C^1} \delta^{1-q}$. To do so, note first that $\|g \varphi_2-\widetilde{g \varphi_1}\|_{C^1(W_2)}\leq 4\|g\|_{C^1}$. Next, note that 
\begin{align*}
\sup_{t \leq S_2}\left |g \varphi_2(\mathbf{x}_2+t \mathbf{v}_2)- \widetilde{g \varphi_1}(\mathbf{x_2}+t\mathbf{v}_2)\right | &\leq \sup_{t \leq S_2}\left |g(\mathbf{x_2}+t\mathbf{v}_2)(\varphi_2(\mathbf{x_2}+t \mathbf{v_2})-\varphi_1(\mathbf{x}_1+t\mathbf{v_1})) \right |\\
& \quad +\sup_{t \leq S_2}\left |\varphi_1(\mathbf{x_1}+t\mathbf{v_1})(g(\mathbf{x}_2+t \mathbf{v_2})- g(\mathbf{x}_1+t \mathbf{v_1}))\right |\\
&\leq \|g\|_{C^1}\delta+\|\varphi_1\|_{C^1(W_1)}\|g\|_{C^1}(|\mathbf{x}_1-\mathbf{x}_2|+S_2|\mathbf{v}_1-\mathbf{v_2}|) 
\\
& \leq 3 \delta\|g\|_{C^1},
\end{align*}
where we used that $S_2 \leq 1$ by the definition of our set of admissible leaves $\Sigma$. Therefore, interpolating yields 
$$
\|g \varphi_2-\widetilde{g \varphi_1}\|_{C^q(W_2)} \leq \|g \varphi_2-\widetilde{g \varphi_1}\|^q_{C^1(W_2)} \|g \varphi_2-\widetilde{g \varphi_1}\|^{1-q}_{C^0(W_2)} \leq 12\|g\|_{C^1} \delta^{1-q},
$$
as claimed. Therefore, all in all 
$$
\frac{1}{\delta^\beta}\left | \int_{W_1} gh \varphi_1\, \dd \HH^1 -\int_{W_2} gh \varphi_2\, \dd \HH^1\right | \leq C\|h\|_u\|g\|_{C^1}(1+\delta^{1-q-\beta}) \leq C\|g\|_{C^1}\|h\|,
$$
since $\beta<1-q$ by assumption. Hence, the claim is proven.
\end{proof}

The next result we consider governs the way integration against stable curves interacts with changes of variables through the map $\XX_\alpha$.
\begin{lemma} \label{lemma:jacobian}
Let $W \in \Sigma$ be an admissible curve parametrized by  $\gamma: (0,S) \to \TT^2$ 
$
\gamma(t)=\{(x,y)+t \mathbf{v}, t \in (0,S)\} \,,
$
 Let $\tilde W = T^{-1}_\alpha (W)$ be
         parametrized by $\tilde \gamma = T^{-1}_\alpha (\gamma)$. We define $J_{\tilde W} \XX_\alpha= |D \XX_\alpha(\tilde \gamma(t)) \tilde \tau (t)|$ with $\tilde \tau(t) := \frac{\tilde \gamma'(t)}{| \tilde \gamma'(t)|}\,.$ Then for any $C^1$ vector field $h$, the following holds true:
         \begin{align}
             \int_{{W}} h \, \dd \HH^1= \int_{\tilde W} h\circ \XX_\alpha |J_{\tilde W} \XX_\alpha| \, \dd \HH^1 \,.
         \end{align}
   Furthermore, if $W \in \Sigma$ is such that $W \subset \overline{\XX_\alpha (\cM_\ell)}$ then $J_{\tilde{W}} \XX_\alpha$ is a constant value satisfying
   $|J_{\tilde{W}} \XX_\alpha| \leq 2 \alpha^{-2} \,.$
\end{lemma}

\begin{remark}
    Note that $\tilde{W}$ need not be admissible, since $|\tilde{W}|\approx \alpha^{2}|W|$ may be too long. Crucially however, it may always be written as a disjoint union of admissible curves.
\end{remark}
\begin{proof}
   Since $T_\alpha $ is a bijection we have that $\gamma ' (t) = D T_\alpha (\tilde \gamma (t)) \tilde {\gamma}'(t)$. 
Let $h : \mathbb{T}^2 \to \mathbb{C}^2$ be a $C^1$ function.
Then using the definition of integral along curves with parameterisations $\gamma$ and $ \tilde \gamma$ we deduce
\begin{align}
    \int_{W} h \, \dd \HH^1& = \int_0^S  h({\gamma} (t)) |{\gamma}'(t)| \dd t = \int_0^S h(\XX_\alpha (\tilde \gamma(t))) |D \XX_\alpha(\tilde \gamma(t)) \tilde \gamma'(t)| \dd t
    \\
    & = \int_0^S h(\XX_\alpha (\tilde \gamma(t))) |D \XX_\alpha(\tilde \gamma(t)) \tilde  \tau (t)| |\tilde \gamma'(t)| \dd t = \int_{\tilde W} h\circ \XX_\alpha |J_{\tilde W} \XX_\alpha| \, \dd \HH^1\,.
\end{align}
For the estimate of the Jacobian along the curve  we use ${\gamma} (t) = \XX_\alpha (\tilde \gamma (t))$ and assume $|\gamma '(t)|=1$. Then, it holds that 
$$ \XX_\alpha \circ \XX_\alpha^{-1} ( \gamma(t))=  \gamma(t),$$
and taking the derivative in time we have 
\begin{align} \label{eq:identity-trivial}
    | D\XX_\alpha (\tilde \gamma(t)) \tilde  \gamma'(t) | = | \gamma '(t)| =1 \,.
\end{align}
By definition and by the fact that $ \gamma'(t)$ is in the stable cone, hence expanding through $\XX_\alpha^{-1}$,  we obtain   
$$  |\tilde \gamma '(t)| =  |D\XX_\alpha^{-1} ({\gamma} (t))  \gamma'(t)| \geq \frac{1}{2} \alpha^2  \,.$$
Hence, by this bound and \eqref{eq:identity-trivial} we conclude $|J_{\tilde W} \XX_\alpha| \leq 2 \alpha^{-2}$ for any $t$.
\end{proof}

In the next lemma we show how compositions with $\XX_\alpha$ affect the H\"older norms of test functions.
\begin{lemma} \label{lemma:property-test}
Let $W \in \Sigma $ such that $W\subset \XX_\alpha (\cM_\ell)$ for some $\ell$ and let $\tilde W  \subset \XX_\alpha^{-1} (W)$ be an admissible curve in $\Sigma$. Then,
for any $x,y \in \tilde{W}$ we have 
        $$|\XX_{\alpha} (x) - \XX_\alpha (y)| \leq 2 \alpha^{-2} |x-y| \,.$$
        In particular, $[ \XX_\alpha ]_{C^1 (\tilde{W})} \leq 2 \alpha^{-2}$ and $[\varphi \circ \XX_\alpha]_{C^r (\tilde{W})} \leq 2 \alpha^{-2r} [\varphi]_{C^r (W)}\,,$ for $r \in [0,1]$.
\end{lemma}
\begin{proof}
    Since $\XX_\alpha$ is a bijection for any $x, y \in \tilde{W}$ we can define $x= \XX_\alpha^{-1} (z)$ and $y = \XX_\alpha^{-1} (w)$ with $z,w \in W$. Hence,  
    $$ |\XX_\alpha (x) - \XX_\alpha (y)| = |z-w|$$
    and using that $z-w $ is a vector in the stable cone and $\XX_\alpha^{-1}$ expand such vector by $\alpha^2$ we have 
    $$ |\XX_\alpha^{-1} (z)- \XX_\alpha^{-1} (w)| \geq \frac{\alpha^2}{2} |z-w|$$
    hence 
    $$|\XX_\alpha (x) - \XX_\alpha (y)| = |z-w| \leq 2 \alpha^{-2} |\XX_\alpha^{-1} (z) - \XX_\alpha^{-1} (w)| = 2 \alpha^{-2} |x-y| \,.$$
    The estimate for $\varphi \circ \XX_\alpha$ is an immediate consequence.
\end{proof}

Finally, we state a crucial technical lemma which concerns matching of curves according to smoothness strips in the spirit of the proof of Lemma \ref{lemma:indicators-y}, as well as the behaviour of matched curves under preimages of $\XX_\alpha$. The proof of this result is quite lengthy and technical, and will thus be postponed to Appendix \ref{appendix}.

\begin{lemma} \label{lemma:technical}
    There exist $\alpha_0 \geq 1$ such that for all $\alpha \geq \alpha_0$, and for any two admissible curves $W_1, W_2 \in \Sigma$  with $\dd_{\Sigma} (W_1, W_2)  \leq \frac{2}{\alpha}$, the following hold true:
    \begin{enumerate}
        \item \label{enum:0} For any  $W \in \Sigma$ there exist $(W_{j})_{j=1}^{M_\alpha}$ with $M_\alpha  \leq 3+ \lceil \alpha |W| \rceil $ such that $W_{j} \in \Sigma$, $W= \bigcup_j^{M_\alpha} W_{j}$ and $|W_{j}| \leq \frac{2}{\alpha}$, and for any $j  $  there exists $\ell=1,2,3,4 $ such that  $W_{j} \subset \overline{\XX_\alpha (\cM_\ell)}$.  Furthermore, for each $W_{j}$ we can write $\XX_\alpha^{-1} (W_{j}) = \bigcup_{k=1}^{N_\alpha} \tilde W_{j,k}$ with $N_\alpha \leq 3+ \lceil \alpha^2 |W_j| \rceil $ and $\tilde W_{j,k} \in \Sigma$ for any $j,k$. 
        \item \label{enum:1}  For any $i=1,2$  we have $ W_i = \bigcup_{j=1}^{M_\alpha} W_{i,j} \cup \bigcup_{j=1}^8 U_{i,j}$ with $W_{i,j}, U_{i,j} \in \Sigma$ for any $i,j$ and $W_{i,j} \subset \overline{\XX_\alpha (\cM_\ell)}$ for some $\ell=1,2,3,4$ and $M_\alpha \leq 2\alpha$. Furthermore, $|W_{i,j}| \leq \frac{2}{\alpha}$, $|U_{i,j}| \leq C \dd_\Sigma (W_1, W_2)$ and  $\dd_\Sigma (W_{1,j}, W_{2,j}) \leq C \dd_\Sigma (W_{1}, W_{2})$.
        \item \label{enum:2}   For any $i=1,2$ and  $j$  we have $\XX_\alpha^{-1} (W_{i,j} )= \bigcup_{k=1}^{N_\alpha}
         \tilde W_{i,j,k} \cup \bigcup_{k=1}^8 \tilde U_{i,j,k} $ with $N_\alpha \leq 2 \alpha$ and $$\dd_{\Sigma} (\tilde W_{1,j,k}, \tilde W_{2,j,k} ) \leq C \alpha^{-2} \dd_\Sigma (W_{1,j}, W_{2,j})\,, \quad \forall j,k\,,$$ $\tilde W_{i,j,k}, \tilde U_{i,j,k} \in \Sigma$ for all $i,j,k$ and 
        $$|\tilde U_{i,j,k}|\leq C \alpha^{-1} \dd_{\Sigma} (W_{1,j}, W_{2,j}) \,, \forall i,j,k\,.$$
        Furthermore, if $W_{i,j}$ are parameterised  by $\gamma_{i,j}(t)$, then $\tilde W_{i,j,k}=\XX_\alpha^{-1} \circ \gamma_{i,j}([t_{i,j,k},t_{i,j,k+1}])$, with $|t_{1,j,k}-t_{2,j,k}|\leq C\alpha^{-1} \dd_\Sigma(W_{1,j},W_{2,j})$ and $|t_{1,j,k}|+|t_{2,j,k}|\leq C\alpha$.
    \end{enumerate}
\end{lemma}
\section{Lasota-Yorke inequality} \label{sec:lasota}
In this section, we establish the spectral properties of the vector-valued transfer operator in \eqref{eq:Lalpha}, namely
$$
\mathcal{L}_\alpha : X \to X,\qquad \mathcal{L}_\alpha h = \frac{1}{\alpha^2} (D\XX_\alpha h) \circ \XX_\alpha^{-1}\,,
$$
where $T_\alpha : \TT^2 \to \TT^2$ is the uniformly hyperbolic map defined in Section \ref{sec:themap}. In particular, we shall prove the Lasota-Yorke inequality stated in Proposition \ref{proposition:main lasota yorke} with $(X,\|\cdot \|)$, $(X_w, |\cdot |_w)$ defined in Section \ref{sec:anisotropic}.

Notation-wise, when we write $\mathcal{L}_\alpha h $, we highlight that $h$ is a vector-valued distribution, and its integration against scalar test functions is performed component-wise. To facilitate the estimates in the following subsections, we recall several key properties of the map $\XX_\alpha$ and the geometry of the leaves in $\Sigma$:

\smallskip
\noindent $\diamond$ \emph{Jacobian Scaling:} For any leaf $\tilde{W}$, the Jacobian $J_{\tilde{W}} \XX_\alpha$ is constant on each partition element and satisfies (see Lemma \ref{lemma:jacobian})
    \begin{equation}\label{eq:JacBound}
        |J_{\tilde{W}} \XX_\alpha| \leq 2 \alpha^{-2} \,.
    \end{equation}
$\diamond$ \emph{Matrix norms:} On each partition element $\mathcal{M}_\ell$, the derivative $D \XX_\alpha$ is a constant matrix $A_\ell$ with magnitude (see \eqref{eq:matrixMi})
\begin{equation}\label{eq:AiNorms}
|A_\ell| \leq 2 \alpha^2.
\end{equation}
$\diamond$ \emph{Test function pullback:} For any $\varphi \in C^1(W)$, in light of Lemma \ref{lemma:property-test}, its pullback satisfies the non-expansion property
    \begin{equation}\label{eq:testpullbackbound}
        \| \varphi \circ \XX_\alpha \|_{C^1 (\XX_\alpha^{-1} (W))} \leq \| \varphi \|_{C^1 (W)} \,.
    \end{equation}
$\diamond$ \emph{Leaf averages:} We denote the average of a function $f$ over a leaf $W$ by 
\begin{equation}\label{eq:leafavg}
\overline{f}^{W} = \frac{1}{|W|} \int_{W} f \, \dd \HH^1 \,.
\end{equation}
The main result of this section is summarized in the following proposition, which provides the quantitative bounds necessary to prove the existence of a spectral gap. Note that the following proposition implies Proposition \ref{proposition:main lasota yorke} by recalling \eqref{eq:choice-parameters} and defining $\eta = \min\{ 2q, 2 - 2 \sigma, 2 - \sigma - \beta, 1- \beta \}>0$.

\begin{proposition}\label{prop:LY}
There exists a constant $\alpha_0 > 0$, so that for all  $\alpha \geq \alpha_0$, and any $h \in X$, the following hold. The operator $\cL_\alpha$ is continuous in the weak norm, namely
\begin{equation}\label{eq:weakcont}
|\cL_\alpha h|_w \leq C |h|_w.
\end{equation}   
Moreover, it satisfies the strong stable norm estimate
\begin{equation}\label{eq:strongstabcont}
\| \cL_\alpha h \|_{s} \leq C (\alpha^{-2q} + \alpha^{2\sigma- 2}) \| h \|_s + C |h|_w,
\end{equation}  
and strong unstable norm estimate
\begin{equation}\label{eq:strongunstabcont} 
        \| \cL_\alpha h \|_{u} \leq C \left[ (\alpha^{-2 + \sigma+ \beta }  + \alpha^{- 1 + \beta} + \alpha^{-2q + \beta - \sigma} ) \| h \|_s + \alpha^{-2 \beta} \|h \|_{u} + \alpha^{\beta- \sigma}|h|_w \right].
\end{equation}   
\end{proposition}
The proof of Proposition \ref{prop:LY} is organized as follows: Section \ref{sub:weak} establishes the weak norm continuity \eqref{eq:weakcont}, while Sections \ref{sub:stable} and \ref{sub:unstable} provide the bounds \eqref{eq:strongstabcont}-\eqref{eq:strongunstabcont} for the strong stable and unstable norms, respectively.

In particular, inspecting the proof of Proposition \ref{prop:LY}, which estimates each branch contribution separately before summing over $\ell$, we deduce the following branchwise version, which will be needed in Section \ref{sec:heat}.
\begin{corollary} \label{lemma:improved-lasota}
    For any $\ell=1,2,3,4$ and $A_\ell$ in \eqref{eq:matrixMi} we define the operators
    $ \cL_{\alpha}^\ell (h) = \frac{1}{\alpha^2} A_\ell \mathbbm{1}_{T_\alpha (\cM_\ell)} h(T_\alpha^{-1})$.
    Then, there exists a constant $C>0$ such that 
    $$ \| \cL_{\alpha}^\ell (h) \| \leq C \alpha^{-\eta} \| h \| + C |h|_w \,, \qquad  | \cL_{\alpha}^\ell (h) |_w \leq C |h|_w \,.$$
\end{corollary}
\subsection{Weak norm continuity}\label{sub:weak}
We begin by establishing the bound \eqref{eq:weakcont}. Let $W \in \Sigma$ be an admissible leaf and $\varphi \in C^1(W)$ be a test function such that $\|\varphi\|_{C^1(W)} \leq 1$.

By Lemma \ref{lemma:technical}-\eqref{enum:0}, we partition $W$ into segments $\{W_j\}_{j=1}^{M_\alpha}$, each contained in a single image partition $\overline{\XX_\alpha(\mathcal{M}_{\ell_j})}$. Each segment pulls back to $N_\alpha$ pre-image leaves $\{\tilde{W}_{j,k}\}_{k=1}^{N_\alpha}$. Applying the definition of $\mathcal{L}_\alpha$ and the change of variables formula, we have
\begin{align*}
    \left |\int_W \mathcal{L}_\alpha h \varphi \, \dd \HH^1\right |&= \left |\frac{1}{\alpha^2} \sum_{j=1}^{M_\alpha} \sum_{k=1}^{N_\alpha} \int_{\tilde{W}_{j,k}} D\XX_\alpha h \cdot (\varphi \circ \XX_\alpha) \, |J_{\tilde{W}_{j,k}}\XX_\alpha| \, \dd \HH^1 \right | \\
    &\leq \frac{1}{\alpha^2} \sum_{j=1}^{M_\alpha} \sum_{k=1}^{N_\alpha} |A_{\ell_j}| \cdot |J_{\tilde{W}_{j,k}}\XX_\alpha| \cdot \left| \int_{\tilde{W}_{j,k}} h (\varphi \circ \XX_\alpha) \, \dd \HH^1 \right|,
\end{align*}
where we used that on each pre-image leaf $\tilde{W}_{j,k}$, the derivative $D\XX_\alpha$ is a constant matrix $A_{\ell_j}$ and the Jacobian is a constant. Using the Jacobian bound \eqref{eq:JacBound} and the matrix bound \eqref{eq:AiNorms}, and noting that $\|\varphi \circ \XX_\alpha\|_{C^1} \leq 1$ by \eqref{eq:testpullbackbound}, we obtain
\begin{equation}\label{eq:weakLalphabd1}
    \left| \int_W \mathcal{L}_\alpha h \varphi \, \dd \HH^1 \right| \leq  \frac{C}{\alpha^2} \sum_{j=1}^{M_\alpha} \sum_{k=1}^{N_\alpha} |h|_w.
\end{equation}
From the counting bounds in Lemma \ref{lemma:technical}, the total number of pre-image components satisfies 
\begin{equation}
    \sum_{j=1}^{M_\alpha} N_\alpha \leq \sum_{j=1}^{M_\alpha} (\alpha^2 |W_j| + 1) \leq \alpha^2 |W| + M_\alpha \leq C \alpha^2,
\end{equation}
since $|W| \leq 1$ and $M_\alpha \sim \alpha$. Substituting this into \eqref{eq:weakLalphabd1} completes the proof of  \eqref{eq:weakcont}.

\subsection{Strong stable norm}\label{sub:stable}
We now establish the estimate \eqref{eq:strongstabcont}. Let $W \in \Sigma$ and $\varphi \in C^q(W)$ such that $\| \varphi \|_{\sigma, W} = \|\varphi \|_{C^q (W)} |W|^\sigma \leq 1$. Using the leaf fragmentation and pullback from Lemma \ref{lemma:technical}-\eqref{enum:0}, we write
\begin{align}
    \left | \int_{W} \cL_\alpha h \varphi \, \dd \HH^1 \right | &= \left | \frac{1}{\alpha^2} \sum_{j=1}^{M_\alpha} \sum_{k=1}^{N_\alpha} \int_{\tilde{W}_{j,k}} D\XX_\alpha h \cdot (\varphi \circ \XX_\alpha) \, |J_{\tilde{W}_{j,k}}\XX_\alpha| \, \dd \HH^1 \right | \nonumber \\
    &\leq \frac{4}{\alpha^2} \sum_{j=1}^{M_\alpha} \sum_{k=1}^{N_\alpha} \left| \int_{\tilde{W}_{j,k}} h (\varphi \circ \XX_\alpha) \, \dd \HH^1 \right| \nonumber \\
    &\leq \underbrace{\frac{4}{\alpha^2} \sum_{j,k} \left| \int_{\tilde{W}_{j,k}} h (\varphi \circ \XX_\alpha - \overline{\varphi \circ \XX_\alpha}^{\tilde{W}_{j,k}}) \, \dd \HH^1 \right|}_{I} + \underbrace{\frac{4}{\alpha^2} \sum_{j,k} \left| \int_{\tilde{W}_{j,k}} h \overline{\varphi \circ \XX_\alpha}^{\tilde{W}_{j,k}} \, \dd \HH^1 \right|}_{II}, \label{eq:IandII}
\end{align}
where the leaf-average notation is that of \eqref{eq:leafavg}.

We first estimate term $II$.  Note that $|\overline{\varphi \circ \XX_\alpha}^{\tilde{W}_{j,k}}| \leq \|\varphi\|_{C^0(W)} \leq |W|^{-\sigma}$ given our normalization of $\varphi$. Recall that $j \in \{ 1, \ldots , M_\alpha \}$ and $k \in \{1 , \ldots, N_\alpha \}$. We consider three regimes for the length of the original leaf $|W|$ using the bound for $M_\alpha, N_\alpha$ in Lemma \ref{lemma:technical}:

\smallskip 

\noindent $\diamond$ \emph{Large leaves ($|W| \geq \alpha^{-1}$):} In this case, $M_\alpha \leq 4\alpha|W|$ and $N_\alpha \leq 4\alpha$. Since $\overline{\varphi \circ \XX_\alpha}^{\tilde{W}_{j,k}}$ is constant, the integral is bounded by $|h|_w$. Summing over the fragments yields
    $$ II \leq \frac{4}{\alpha^2} (4\alpha|W|)(4\alpha) |h|_w |W|^{-\sigma} \leq 64 |W|^{1-\sigma} |h|_w \leq C |h|_w. $$
$\diamond$ \emph{Medium leaves ($ \alpha^{-2} \leq |W| < \alpha^{-1}$):} Here, $M_\alpha \leq 4$ and $N_\alpha \leq 3 + \lceil \alpha^2 |W_j| \rceil \leq 4 \alpha^2 |W|$, where the last holds true for any $j$ thanks to $|W| \geq  \alpha^{-2} $. We have
    $$ II \leq \frac{4}{\alpha^2} (3)(4\alpha^2 |W|) |h|_w |W|^{-\sigma} \leq 48 |W|^{1-\sigma} |h|_w \leq C \alpha^{\sigma-1} |h|_w. $$
$\diamond$ \emph{Small leaves ($|W| < \alpha^{-2}$):} At this scale, $M_\alpha, N_\alpha \leq 4$. We utilize the definition of the strong stable norm directly. Since $|\tilde{W}_{j,k}| \leq 2\alpha^2|W|$ and $\|1\|_{\sigma, \tilde{W}_{j,k}} = |\tilde{W}_{j,k}|^\sigma$:
    $$ 
    II \leq \frac{C}{\alpha^2} \|h\|_s \max_{j,k} |\tilde{W}_{j,k}|^\sigma |W|^{-\sigma} \leq C \alpha^{-2} (\alpha^2 |W|)^\sigma |W|^{-\sigma} \|h\|_s \leq C \alpha^{2\sigma-2} \|h\|_s. 
    $$
We now estimate $I$, which measures the oscillation of the test function. By Lemma \ref{lemma:property-test} we have $[\varphi \circ \XX_\alpha]_{C^q (\tilde W_{j,k})} \leq C \alpha^{-2q} \| \varphi \|_{C^q}$. Furthermore, using that $|\XX_\alpha (\tilde W_{j,k})|\leq C \alpha^{-2}$, hence the fluctuation satisfies $\|\varphi \circ \XX_\alpha - \overline{\varphi \circ \XX_\alpha}^{\tilde W_{j,k}} \|_{C^0 (\tilde W_{j,k})} \leq C \alpha^{-2q} \|\varphi\|_{C^q} $. Therefore, we deduce 
$$ \| \varphi \circ \XX_\alpha - \overline{\varphi \circ \XX_\alpha}^{\tilde W_{j,k}} \|_{\sigma, \tilde W_{j,k}} \leq C \alpha^{-2q} \| \varphi \|_{\sigma, W} |W|^{- \sigma} |\tilde W_{j,k}|^\sigma$$
Repeating the same analysis as above, the dominant factor is the contraction of the $C^q$ norm. From the definition of $\|h\|_s$ and the properties of the test function on the pre-image leaves we have
\begin{equation*}
    I \leq \frac{4}{\alpha^2} \sum_{j,k} \|h\|_s \|\varphi \circ \XX_\alpha - \overline{\varphi \circ \XX_\alpha}\|_{\sigma, \tilde{W}_{j,k}} \leq \frac{C}{\alpha^2} \sum_{j,k} \|h\|_s \left( \alpha^{-2q} |W|^{-\sigma} |\tilde W_{j,k}|^\sigma \right) \,.
\end{equation*}
As before, we split the cases of large leaves, medium leaves and small leaves. 

If $|W| \geq \alpha^{-1}$, we use $M_\alpha \leq 4 \alpha |W|$ and $N_\alpha \leq 4 \alpha$ to deduce that $I\leq C\alpha^{-2 q} |W|^{1- \sigma} \| h \|_s$. 

If $  \alpha^{-2} \leq |W|< \alpha^{-1} $, we use $M_\alpha \leq 3$, and $N_\alpha \leq  4 \alpha^2 |W|$ to deduce that $I \leq  \alpha^{- 2q} |W|^{1- \sigma} \| h \|_s$

If $|W| < \alpha^{-2}$, we use $M_\alpha , N_\alpha \leq 3$ and $|\tilde W_{j,k}| \leq 2 \alpha^2 |W|$ to deduce that $I \leq \alpha^{-2 - 2q + 2 \sigma} \| h \|_s$.

Combining the bounds for $I$ and $II$ results in we conclude the proof of \eqref{eq:strongstabcont}.

\subsection{Strong unstable norm}\label{sub:unstable}

To establish the unstable norm estimate \eqref{eq:strongunstabcont}, we consider $W_1, W_2 \in \Sigma$ two admissible curves with $\dd_\Sigma (W_1, W_2) < \delta \leq \alpha^{-1}$. We use property \eqref{enum:1} in Lemma \ref{lemma:technical} to decompose each $W_i$ as
$$ W_i = \bigcup_{j=1}^{M_\alpha} W_{i,j} \cup \bigcup_{j=1}^{8} U_{i,j} \,, $$
where $M_\alpha \leq \alpha$ and $|U_{i,j}| \leq C \delta$. Let $\varphi_1, \varphi_2$ be test functions satisfying $\dd_q (\varphi_1, \varphi_2) \leq \delta$ and $\|\varphi_i\|_{C^1(W_i)} \leq 1$. The integral difference is split into matched and unmatched components 
\begin{align*}
     & \frac{1}{\alpha^2}  \left | \int_{W_1} (D \XX_\alpha h ) \circ \XX_\alpha^{-1} \varphi_1 \, \dd \HH^1  - \int_{W_2} (D \XX_\alpha h ) \circ \XX_\alpha^{-1} \varphi_2 \, \dd \HH^1 \right |
    \\
    & \leq \sum_{j=1}^{M_\alpha} \frac{1}{\alpha^2} \left | A_j \int_{W_{1,j}} h \circ \XX_\alpha^{-1} \varphi_1\, \dd \HH^1 - A_j \int_{W_{2,j}} h \circ \XX_\alpha^{-1} \varphi_2\, \dd \HH^1 \right | \quad (\text{Matched})
    \\
    & \quad + \sum_{i=1,2} \sum_{j=1}^8 \frac{1}{\alpha^2} \left | A_j \int_{U_{i,j}} h \circ \XX_\alpha^{-1} \varphi_i\, \dd \HH^1 \right | \quad (\text{Unmatched}).
\end{align*}
We treat the two components separately.

\subsubsection*{The unmatched curves}
We first estimate the unmatched curves, tackling the integral over $U_{i,j}$ for $i=1$ (the case $i=2$ is symmetric). Using that $|A_j| \leq C\alpha^2$ and  the strong stable bound \eqref{eq:strongstabcont} we deduce
\begin{align*}
\frac{1}{\alpha^2}\left |\int_{U_{i,j}} A_j h \circ \XX_\alpha^{-1} \varphi_i \, \dd \HH^1 \right | \leq \|\varphi_i\|_{\sigma, U_{i,j}} \|\cL_\alpha h\|_{s}  \leq & |U_{i,j}|^\sigma \left ((\alpha^{-2q}+\alpha^{2 \sigma -2})\|h\|_s+C|h|_w \right )\\
\leq &\delta^\beta \alpha^{\beta-\sigma}\left ((\alpha^{-2q}+\alpha^{2 \sigma -2})\|h\|_s+C|h|_w \right ).
\end{align*}

\subsubsection*{The matched curves} We now turn to the matched curves $W_{i,j}$. We will repeatedly use that $M_\alpha, N_\alpha \le 2\alpha$, $|A_j|\le 2\alpha^2$ for any $j$, and that the Jacobian satisfies
$
\bigl|J_{\tilde{W}}\XX_\alpha\bigr|\le \frac{2}{\alpha^2}\,.
$ 
By property \eqref{enum:2} of Lemma \ref{lemma:technical}, for any $j=1, \ldots, M_\alpha$, the pre-image decomposes as
$$ 
\XX_\alpha^{-1} (W_{i,j}) = \bigcup_{k=1}^{N_\alpha} \tilde{W}_{i,j,k} \cup \bigcup_{k=1}^{8} \tilde{U}_{i,j,k} \,,
$$
 where $|\tilde U_{i,j,k}| \leq C \alpha^{-1} \delta$ for any $i,j,k$ and $N_\alpha \leq 2 \alpha$. 
The difference over the matched segments expands to
\begin{align}
    \sum_{j=1}^{M_\alpha} \frac{1}{\alpha^2} & \left | A_j \int_{W_{1,j}} h \circ \XX_\alpha^{-1} \varphi_1\, \dd \HH^1 - A_j \int_{W_{2,j}} h \circ \XX_\alpha^{-1} \varphi_2\, \dd \HH^1 \right | 
    \nonumber\\
    & \leq \sum_{j=1}^{M_\alpha} \sum_{k=1}^{N_\alpha} \frac{2}{\alpha^4} \left | A_j \int_{\tilde W_{1,j,k}} h (\varphi_1 \circ \XX_\alpha)\, \dd \HH^1 - A_j \int_{\tilde W_{2,j,k}} h (\varphi_2 \circ \XX_\alpha)\, \dd \HH^1 \right |\label{eq:estiWijk}
    \\
    & \quad + \sum_{i=1,2} \sum_{j=1}^{M_\alpha} \sum_{k=1}^{8} \frac{2}{\alpha^4} \left | A_j \int_{\tilde U_{i,j,k}} h (\varphi_i \circ \XX_\alpha)\, \dd \HH^1 \right | \,.\label{eq:estiUijk}
\end{align}
We may bound the second term by
\begin{align*}
    \sum_{j=1}^{M_\alpha} \sum_{k=1}^{8} \frac{1}{\alpha^4} \left | A_j \int_{\tilde U_{i,j,k}} h (\varphi_i \circ \XX_\alpha) \, \dd \HH^1 \right |
    & \leq C \alpha^{-1} \| h \|_s \max_{j,k} \| \varphi_i \circ \XX_\alpha \|_{\sigma, \tilde U_{i,j,k}} 
     \\
     &\leq C \alpha^{-1} \|h \|_s \|\varphi_i \|_{C^q(W_i)} |\tilde U_{i,j,k}|^\sigma 
     \leq C \alpha^{-1} (\alpha^{-1} \delta)^\sigma \| h \|_s
    \\
    & \leq C \alpha^{-1-\sigma} \delta^{\sigma-\beta} \delta^\beta \| h \|_s
   \leq C \alpha^{-1 - 2\sigma + \beta} \delta^\beta \| h \|_s \,.
\end{align*}
This takes care of \eqref{eq:estiUijk} in a way that is consistent with \eqref{eq:strongunstabcont}. 

We now turn to \eqref{eq:estiWijk}. We parametrize $W_{i,j}$ as
$$ \gamma_{i,j} (t) = \xx_{i,j} + t \vvv_i \,, \quad t \in [0, t_{fin, i,j}] \,.$$
Hence $\tilde W_{i,j,k}$ is parametrized by 
$$\tilde \gamma_{i,j,k} (t) = \XX_\alpha^{-1}(\xx_{i,j}) + t_{i,j,k} \frac{\XX_\alpha^{-1} (\vvv_i)}{|\XX_\alpha^{-1} (\vvv_i)|} + t \frac{\XX_\alpha^{-1} (\vvv_i)}{|\XX_\alpha^{-1} (\vvv_i)|} \,, \quad t \in [0, \tilde t_{fin, i,j, k}] \,. \footnote{In the case where $\XX_\alpha$ flips orientations, the starting and ending times, as well as the beginning and end points should be reversed.}$$
For any $i,j,k$ we extend the curves $\tilde W_{i,j,k}$ as affine curves on $\RR^2$. From now on we assume without loss of generality that $|\tilde W_{2,j,k}| \leq |\tilde W_{1,j,k}|$. We define an affine map $C_{j,k}: \RR^2 \to \RR^2$ so that $C_{j,k} (\tilde \gamma_{2, j,k} (t)) = \tilde \gamma_{1,j, k} (t)$ for all $t \leq t_{fin,2,j,k}$.
  We define the map ${\varphi}_{1,C_{j,k}} = \varphi_1 ( \XX_\alpha \circ C_{j,k} \circ \XX_\alpha^{-1}) $ which satisfies
  \begin{align*}
      {\varphi}_{1,C_{j,k}} (\XX_\alpha \circ \tilde{\gamma}_{2,j,k} (t)) &= \varphi_{1} (\XX_\alpha \circ C_{j,k} (\tilde{\gamma}_{2,j,k} (t))) = \varphi_1 (\XX_\alpha  (\tilde{\gamma}_{1,j,k} (t)))
      \\
      & = \varphi_1 \left (\XX_\alpha \circ  \left (\XX_\alpha^{-1}(\xx_{1,j}) + t_{1,j,k} \frac{\XX_\alpha^{-1} (\vvv_1)}{|\XX_\alpha^{-1} (\vvv_1)|} + t \frac{\XX_\alpha^{-1} (\vvv_1)}{|\XX_\alpha^{-1} (\vvv_1)|} \right ) \right )
        \\
        & = \varphi_1 \left (\xx_{1,j} + t_{1,j,k} \frac{\vvv_1}{|\XX_\alpha^{-1} (\vvv_1)|} + t \frac{\vvv_1}{|\XX_\alpha^{-1} (\vvv_1)|} \right ) \,.
  \end{align*}
Hence we deduce
    $$ \dd_q (\varphi_1 \circ \XX_\alpha, {\varphi_{1,C_{j,k}}} \circ \XX_\alpha) =0\,.$$
Summing and subtracting the function $\varphi_{1,C_{j,k}}$ we have 
\begin{align*}
    \frac{2}{\alpha^4} & \sum_{j=1}^{M_\alpha} \sum_{k=1}^{N_\alpha}   \left | A_j \int_{\tilde W_{1,j,k}} h  \varphi_1 \circ \XX_\alpha \, \dd \HH^1 -  A_j \int_{\tilde W_{2,j,k}} h  \varphi_2 \circ \XX_\alpha \, \dd \HH^1 \right |
    \\
    & \leq  \frac{2}{\alpha^4} \sum_{j=1}^{M_\alpha} \sum_{k=1}^{N_\alpha}   \left | A_j \int_{\tilde W_{1,j,k}} h  \varphi_1 \circ \XX_\alpha \, \dd \HH^1 -  A_j \int_{\tilde W_{2,j,k}} h  {\varphi}_{1,C_{j,k}} \circ \XX_\alpha \, \dd \HH^1 \right | 
    \\
    & \quad + \frac{2}{\alpha^4} \sum_{j=1}^{M_\alpha} \sum_{k=1}^{N_\alpha}   \left | A_j \int_{\tilde W_{2,j,k}} h  ( {\varphi}_{1,C_{j,k}} \circ \XX_\alpha - {\varphi}_2 \circ \XX_\alpha) \, \dd \HH^1 \right |  \,.
\end{align*}
The first summand can be estimated  thanks to $\dd_q (\varphi_1 \circ \XX_\alpha, \varphi_{1,C_{j,k}} \circ \XX_\alpha)=0$, and properties \eqref{enum:1},  \eqref{enum:2} in Lemma \ref{lemma:technical}
\begin{align*}
    \frac{2}{\alpha^4} \sum_{j=1}^{M_\alpha} \sum_{k=1}^{N_\alpha}   \left | A_j \int_{\tilde W_{1,j,k}} h  \varphi_1 \circ \XX_\alpha\, \dd \HH^1 -  A_j \int_{\tilde W_{2,j,k}} h  \varphi_{1,C_{j,k}} \circ \XX_\alpha \, \dd \HH^1\right | & \leq C \| h \|_u \max_{j,k}\dd_\Sigma^\beta (\tilde W_{1,j,k} , \tilde W_{2,j,k}) 
    \\
    & \leq C \|h \|_u \alpha^{-2\beta} \dd_\Sigma^\beta (W_1, W_2) 
    \\
    & \leq C \alpha^{-2 \beta} \delta^\beta \| h \|_u \,.
\end{align*} 
For the second summand we claim that  
\begin{align}\label{claim:1}
    \| \varphi_{1,C_{j,k}} \circ \XX_\alpha - \varphi_2 \circ \XX_\alpha \|_{C^q (\tilde W_{2,j,k})} \leq   C \delta^{1-q} \alpha^{-2q } + \delta  \,.
\end{align} 
If the claim holds true, recalling that $N_\alpha, M_\alpha \leq 2 \alpha $ and $|A_j|\leq 2\alpha^2$ we bound the second summand
\begin{align*}
    \frac{2}{\alpha^4} \sum_{j=1}^{M_\alpha} \sum_{k=1}^{N_\alpha}   \left | A_j \int_{\tilde W_{2,j,k}} h  ( {\varphi}_{1,C_{j,k}} \circ \XX_\alpha - {\varphi}_2 \circ \XX_\alpha) \, \dd \HH^1 \right | & \leq  C \| h \|_s  \| \varphi_{1,C_{j,k}} \circ \XX_\alpha - \varphi_2 \circ \XX_\alpha \|_{C^q (\tilde W_{2,j,k})}
    \\
    & \leq C \delta^\beta \alpha^{-1+ \beta} \|h \|_s  
\end{align*}
and we conclude the proof.

We now prove the claim. We first estimate the $C^0 (\tilde W_{2,j,k})$ norm, assuming without loss of generality that $|W_{2,j}|\leq |W_{1,j}|$.
\begin{align*}
    & \sup_{t}  |{\varphi}_{1,C_{j,k}} (\XX_\alpha \circ \tilde{\gamma}_{2,j,k} (t)) - \varphi_{2} (\XX_\alpha (\tilde{\gamma}_{2,j,k} (t)))| 
    \\
    & \quad = \sup_{t} \left |\varphi_1 \left (\xx_{1,j} + t_{1,j,k} \frac{\vvv_1}{|\XX_\alpha^{-1} (\vvv_1)|} + t \frac{\vvv_1}{|\XX_\alpha^{-1} (\vvv_1)|} \right ) - \varphi_2 \left (\xx_{2,j} + t_{2,j,k} \frac{ \vvv_2}{|\XX_\alpha^{-1} (\vvv_2)|} + t \frac{\vvv_2}{|\XX_\alpha^{-1} (\vvv_2)|} \right ) \right|
    \\
    & \quad \leq \sup_{t} \left |\varphi_1 \left (\xx_{1,j} + t_{1,j,k} \frac{\vvv_1}{|\XX_\alpha^{-1} (\vvv_1)|} + t \frac{\vvv_1}{|\XX_\alpha^{-1} (\vvv_1)|} \right ) - \varphi_1 \left (\xx_{1,j} + t_{2,j,k} \frac{\vvv_1}{|\XX_\alpha^{-1} (\vvv_2)|} + t \frac{\vvv_1}{|\XX_\alpha^{-1} (\vvv_2)|} \right ) \right |
    \\
    &\quad  \quad + \sup_{t} \left | \varphi_1 \left (\xx_{1,j} + t_{2,j,k} \frac{\vvv_1}{|\XX_\alpha^{-1} (\vvv_2)|} + t \frac{\vvv_1}{|\XX_\alpha^{-1} (\vvv_2)|} \right ) - \varphi_2 \left (\xx_{2,j} + t_{2,j,k} \frac{ \vvv_2}{|\XX_\alpha^{-1} (\vvv_2)|} + t \frac{\vvv_2}{|\XX_\alpha^{-1} (\vvv_2)|} \right ) \right |
\end{align*}
The final term is bounded by $\delta$ by definition of the strong unstable norm. Now, using \eqref{enum:2} in Lemma \ref{lemma:technical}, we deduce 
$$ |t_{1,j,k}-t_{2,j,k} | \leq C \alpha^{-1} \delta \,, \qquad |t_{1,j,k}| + |t_{2,j,k}| \leq C \alpha \,,$$
and 
$$
\left|\frac{1}{|\XX_\alpha^{-1}\mathbf{v}_2|}-\frac{1}{|\XX_\alpha^{-1}\mathbf{v}_1|}\right|
\leq \frac{|\XX_\alpha^{-1}(\mathbf{v}_1-\mathbf{v}_2)|}{|\XX_\alpha^{-1}\mathbf{v}_1||\XX_\alpha^{-1}\mathbf{v}_2|} \leq C\frac{|\mathbf{v_1}-\mathbf{v_2}|}{|\XX_\alpha^{-1}\mathbf{v_1}|} \leq C\alpha^{-2} \delta\,,
$$
and using also $\|\varphi_1\|_{C^1(W_{1})} \leq 1$ we deduce 
$$
 \| \varphi_{1,C_{j,k}} \circ \XX_\alpha - \varphi_2 \circ \XX_\alpha \|_{C^0 (\tilde W_{2,j,k})}  \leq C\delta. 
$$
Now, from Lemma \ref{lemma:property-test} we deduce that the Lipschitz seminorms satisfy the bound
$$[\varphi_{1,C_{j,k}} \circ \XX_\alpha ]_{C^1} + [\varphi_2 \circ \XX_\alpha]_{C^1} \leq  C\alpha^{-2} \,.$$
Hence, using the interpolation inequality
$
 [\cdot  ]_{\dot C^q} \leq [\cdot]^{q}_{ C^1} [\cdot ]^{1-q}_{C^0}
$ we deduce that  the claim \eqref{claim:1} holds true.

\section{Spectral theory of the ideal dynamo operator} \label{sec:spectral-ideal}

We recall that $(X, \| \cdot \|)$ is the anisotropic Banach space defined in Section \ref{sec:anisotropic}. Our focus is the operator:
$$\mathcal{L}_\alpha : X \to X, \qquad \mathcal{L}_\alpha h = \frac{1}{\alpha^2} (D\XX_\alpha  h ) \circ \XX_\alpha^{-1} \,,
$$
where $\XX_\alpha: \mathbb{T}^2 \to \mathbb{T}^2$ is the uniformly hyperbolic map defined in Section \ref{sec:themap} and $h : \mathbb{T}^2 \to \mathbb{C}^2$. To perform the analysis, we partition the torus into four quadrants:$$
\begin{aligned}
Q_1 &:= \{(x,y): x \ge \tfrac12, y \ge \tfrac12\}, \quad Q_2 := \{(x,y): x < \tfrac12, y \ge \tfrac12\},\\
Q_3 &:= \{(x,y): x < \tfrac12, y < \tfrac12\}, \quad Q_4 := \{(x,y): x \ge \tfrac12, y < \tfrac12\} \,.
\end{aligned}
$$
We denote the integral of a function $f$ over a set $Q$ by $$ 
(f)_Q = \int_Q f \, , \qquad \forall f \in L^1\,, \forall Q \subset \TT^2\,.
$$ 
In this section, we prove Theorem \ref{thm:main eigenvalue} by establishing the existence of a leading eigenvalue through a limiting rank-1 operator.
\begin{proposition} \label{prop:eigenvalue}
Let $g \in C^1 (\mathbb{T}^2)$ be such that $(\e^{2 \pi i g})_{Q_1} + (\e^{2 \pi i g})_{Q_3} - (\e^{2 \pi i g})_{Q_4} - (\e^{2 \pi i g})_{Q_2} \neq 0$. There exists $\alpha_0 \geq 1$ such that for any $\alpha \geq \alpha_0$, the operator $\e^{2 \pi i g}\mathcal{L}_\alpha$ admits a simple, discrete eigenvalue $\mu_\alpha \in \mathbb{C}$ satisfying
$$ 
|{\mu}_\alpha| \geq \frac{1}{2} \left| (\e^{2 \pi i g})_{Q_1} + (\e^{2 \pi i g})_{Q_3} - (\e^{2 \pi i g})_{Q_4} - (\e^{2 \pi i g})_{Q_2} \right| \,.
$$
In particular, there exists $g \in C^1 (\mathbb{T}^2)$ such that the operator $\alpha^2 \e^{2 \pi i g} \mathcal{L}_\alpha$ admits a leading eigenvalue $\lambda_\alpha=\alpha^2\mu_\alpha$ with $|\mu_\alpha| \geq 1/4$.
\end{proposition}
A suitable $g \in C^1(\mathbb{T}^2)$ can be constructed via mollification 
$$
g = f \star \varphi_\ell, \qquad f(x,y)=\frac12\,\mathbbm{1}_{Q_2\cup Q_4}(x,y) \,,
$$
where   $\varphi_\ell$ is a standard periodic Friedrichs mollifier.
Then,
\[
(\e^{2 \pi i f})_{Q_1} + (\e^{2 \pi i f})_{Q_3}
- (\e^{2 \pi i f})_{Q_4} - (\e^{2 \pi i f})_{Q_2}
= 1/4+1/4-(-1/4)-(-1/4)=1\,,
\]
which implies that $|(\e^{2 \pi i g})_{Q_1} + (\e^{2 \pi i g})_{Q_3}
- (\e^{2 \pi i g})_{Q_4} - (\e^{2 \pi i g})_{Q_2}| \geq 1/2$ for $\ell $  sufficiently small.

The core of the proof is the existence of a rank-1 ``limiting operator'' $\mathcal{L}_\infty : X \to X$ that approximates $\mathcal{L}_\alpha$ as $\alpha \to \infty$:
\begin{align} \label{d:cL-infty}
\mathcal{L}_\infty h= \begin{pmatrix}
\left ( \int_{x \geq 1/2} h^{(1)}   -\int_{x<1/2}h^{(1)} \right )(\mathbbm{1}_{y \geq 1/2}-\mathbbm{1}_{y<1/2})
\\
0
\end{pmatrix} \,, \qquad h=(h^{(1)},h^{(2)}).
\end{align}
While it is tempting to assume that the spectral properties of $\mathcal{L}_\alpha$ must mirror those of $\mathcal{L}_\infty$ for large $\alpha$, this convergence is non-trivial in our setting. Specifically, because the anisotropic Banach space $(X, \|\cdot\|)$ is itself $\alpha$-dependent, the very notion of operator convergence requires a careful, quantitative treatment. To bridge this gap, we proceed via uniform resolvent estimates, rigorously tracking the dependencies on $\alpha$ to ensure the stability of the leading resonance.

\begin{remark}[The role of $g$] \label{remark:1 iteration}
In the absence of the phase factor $g$ (which represents the out-of-plane shear), the operator $\mathcal{L}_\alpha$ has no non-trivial eigenvalues. This is consistent with anti-dynamo theorems, as the purely $2d$ map cannot sustain growth.

Our proof relies on introducing the shear $\e^{2 \pi ig}$ after only \emph{one} iteration of the $2d$ operator. If the shear were applied after two iterations, we would be considering $\e^{2 \pi ig} \mathcal{L}_\infty^2$. However, $\mathcal{L}_\infty^2 \equiv 0$, as the range of
$\cL_\infty$ is simply the span of $\begin{pmatrix}\mathds{1}_{y \geq 1/2}-\mathds{1}_{y<1/2}\\
0\end{pmatrix}$, and 
$$
\left ( \cL_\infty \begin{pmatrix}
    \mathds{1}_{y \geq 1/2}
    \\
    0
\end{pmatrix} \right )^{(1)}=\left (\int_{x \geq 1/2, y \geq 1/2} 1  -\int_{x < 1/2, y \geq 1/2} 1   \right )(\mathds{1}_{y \geq 1/2}-\mathds{1}_{y<1/2})=0\,,
$$
and the same holds true with $\mathds{1}_{y < 1/2}$.
Hence, the behaviour of $\e^{2\pi i g}\cL_\alpha^n$ for fixed $\alpha$ as $n\to\infty$ is indeed entirely distinct from that of $\e^{2\pi i g}\cL_\alpha$ as $\alpha\to\infty$.
\end{remark}

\subsection{Technical lemmas}
The following two lemmas provide the rigorous bridge from $\alpha = \infty$ to finite $\alpha$. The convergence of $\e^{2 \pi i g}\cL_\alpha$ as $\alpha \to \infty$ is established in the following lemma and it represents the main technical result of this section.
\begin{lemma} \label{lemma:convergence}
    Let $\e^{2 \pi ig} \in C^1(\mathbb{T}^2)$. Then, there exist a constant $\tilde \eta$ depending only on $\|\e^{2 \pi i g} \|_{C^1}$, and $\alpha_0 \geq 1$ so that for all $\alpha \geq \alpha_0$, it holds 
$$
\|\e^{2 \pi i g}\cL_\alpha h -\e^{2 \pi i g}\cL_\infty h\| \leq C\alpha^{-\tilde \eta }\|h\| \,, \quad \forall h \in X\,.
$$
\end{lemma}
The proof of this result is delayed to Section \ref{sec: convergence}.
In order to deduce the convergence of spectra from the above, we need to ensure that the resolvents of our limiting operator are uniformly bounded in the strong norm, despite the fact that $\|\cdot \|$ intrinsically depends on $\alpha$. The following Lemma ensures this is the case.

\begin{lemma}[Resolvent Bound]
\label{lemma:technical resolvent lemma} 
Let $g \in C^1 (\TT^2)$ be a non zero function.
There exists a matrix $M \in \CC^{2\times 2}$ with eigenvalues $0$ and $\mu=(\e^{2 \pi i g})_{Q_1}+(\e^{2 \pi i g})_{Q_3}-(\e^{2 \pi i g})_{Q_2}-(\e^{2 \pi i g})_{Q_4}$, so that the following holds true.
There exists $\alpha_0 >1$ so that for any  $\alpha \geq \alpha_0$ and for any  $\lambda \in \CC \setminus \{ 0, \mu \}$ it holds 
$$
\|(\lambda -\e^{2 \pi i g}\mathcal{L}_\infty)^{-1}\|_{X \to X} \leq C|\lambda|^{-1}(1+\|(\lambda -M)^{-1}\|_{\CC^2 \to \CC^2})\,.
$$
\end{lemma}
\begin{proof}
Since the operator $\e^{2 \pi ig}\cL_\infty$ is diagonal in its components, and its second component is the zero map, it suffices to compute the inverse of its first component. Thus, for the remainder of the proof, we abuse notation and identify $\e^{2 \pi ig}\cL_\infty$ with its first component.

We express $\e^{2 \pi i g} \cL_\infty$ as the product
$$
\e^{2 \pi i g}\cL_\infty=A_g B,
$$
where $A_g : \mathbb{R}^2 \to X$ and $B: X \to \mathbb{C}^2$ are defined respectively by
\begin{align} \label{d:matrix-A-g}
A_g (v_1,v_2)=\e^{2 \pi i g} (v_1 - v_2) (\mathds{1}_{y \geq 1/2}-\mathds{1}_{y < 1/2})\,,
\end{align}
and 
\begin{align} \label{d:matrix-B}
B f=\begin{pmatrix}
    \int_{x \geq 1/2} f  
    \\
    \int_{x<1/2} f  
\end{pmatrix} \,.
\end{align}
Applying the Sherman-Morrison-Woodbury finite dimensional reduction formula to compute
$$ f= (\lambda - \e^{2 \pi ig}\cL_\infty)^{-1} h \,.$$
For some $c \in \CC^2$ we use the ansatz $f = \lambda^{-1} (h + A_g c)$
then $f= (\lambda - \e^{2 \pi ig}\cL_\infty)^{-1} h $ if and only if 
$$ A_g c - \lambda^{-1} A_g B h - \lambda^{-1} A_g B A_g c =0$$
hence we impose that $c \in \CC^2$ is such  that
$(\Id - \lambda^{-1}B A_g) c= \lambda^{-1} B h$. Hence assuming that the $2 \times 2$ matrix $(\Id - \lambda^{-1}B A_g)$ is invertible for any $\lambda \in \Gamma$ we choose 
$$ c=  \lambda^{-1} (\Id - \lambda^{-1} B A_g)^{-1} B h \,.$$
We deduce  the Sherman-Morrison-Woodbury finite dimensional reduction formula
\begin{align*}
   f= (\lambda - \cL_\infty)^{-1} h = \lambda^{-1} (h + \lambda^{-1} A_g (\Id - \lambda^{-1} B A_g)^{-1} B h  )\,.
\end{align*}
We define $M = B A_g \in \CC^{2\times 2}$. Hence, we have 
\begin{align*}
    \| (\lambda - \cL_\infty)^{-1} h\| \leq |\lambda|^{-1} (\| h \| + \| A_g \|_{\CC^2 \to X} \| (\lambda - M)^{-1} \|_{\CC^2 \to \CC^2} \| B \|_{X \to \CC^2} \| h \|) \,.
\end{align*}
Using Lemma \ref{lemma:indicators-y}, we deduce $\| B \|_{X \to \CC^2} \leq 2$ and using Lemma \ref{lemma:indicators-y} and Lemma \ref{lemma:smooth multiplier} we deduce $\| A_g \|_{\CC^2 \to X} \leq C $ uniformly in $\alpha$. Finally, we note that the matrix $M=BA_g$ can be computed to be
$$ M=BA_g = \begin{pmatrix}
(\e^{2 \pi i g})_{Q_1}-(\e^{2 \pi i g})_{Q_4} & (\e^{2 \pi i g})_{Q_4}-(\e^{2 \pi i g})_{Q_1}\\
(\e^{2 \pi i g})_{Q_2} -(\e^{2 \pi i g})_{Q_3} & (\e^{2 \pi i g})_{Q_3}-(\e^{2 \pi i g})_{Q_2} 
\end{pmatrix}  \in \CC^{2\times 2}\,,$$
which indeed possesses the claimed eigenvalues of $0$ and $\mu=(\e^{2 \pi i g})_{Q_1}+(\e^{2 \pi i g})_{Q_3}-(\e^{2 \pi i g})_{Q_2}-(\e^{2 \pi i g})_{Q_4}$. Thus, we conclude the proof.
\end{proof}
With these results under our belt, the proof of Proposition \ref{prop:eigenvalue} may now be deduced in a way analogous to the standard result on isomorphisms of ``close'' Riesz projectors.

\begin{proof}[Proof of Proposition \ref{prop:eigenvalue}]
    Let $\Gamma$ be some circle in the complex plane enclosing $(\e^{ 2 \pi ig})_{Q_1}-(\e^{ 2 \pi ig})_{Q_2}-((\e^{ 2 \pi ig})_{Q_3}-(\e^{ 2 \pi ig})_{Q_4})$, with radius $r>0$ small enough so that the circle does not enclose or pass through the origin. For any $\alpha \geq 1$ (including $\alpha = \infty$), we define the Riesz projector 
$$
P_\alpha h=\frac{1}{2 \pi i}\int_{\Gamma}(\lambda -\e^{2 \pi i g}\cL_\alpha)^{-1} h\,\dd \lambda
$$
We claim that $P_\infty$ is a non-zero operator. 
 We recall that $(\e^{2 \pi i g}\cL_\infty)^{(1)}$ may be written as the composition of two operators, namely
$$
(\e^{2 \pi i g}\cL_\infty)^{(1)}=A_g B,
$$
where $A_g : \mathbb{R}^2 \to X$, $B: X \to \mathbb{R}^2$ are defined above in \eqref{d:matrix-A-g}, \eqref{d:matrix-B}. 
 In particular, note that the non-zero eigenvalues of $\e^{2 \pi i g}\cL_\infty$ are actually equal to those of the $2 \times 2$ matrix $BA_g$. Indeed, if $BA_g v=\lambda v$ and $\lambda \neq 0$, $v \neq 0$ (which implies $A_g v \neq 0$), then $A_g B A_g v= \lambda A_g v$, and since $A_g v \neq 0$, it follows that $\lambda$ is also an eigenvalue of $A_g B$. Similarly, if $A_g Bf =\lambda f$ and $\lambda \neq 0$, $f \neq0$ (which implies $Bf \neq 0$) it holds that $B A_g Bf=\lambda Bf$ and so $\lambda$ is an eigenvalue of $BA_g$. As in the proof of Lemma \ref{lemma:technical resolvent lemma}, we see that  
$$ BA_g = \begin{pmatrix}
(\e^{2 \pi i g})_{Q_1}-(\e^{2 \pi i g})_{Q_4} & (\e^{2 \pi i g})_{Q_4}-(\e^{2 \pi i g})_{Q_1}\\
(\e^{2 \pi i g})_{Q_2} -(\e^{2 \pi i g})_{Q_3} & (\e^{2 \pi i g})_{Q_3}-(\e^{2 \pi i g})_{Q_2} 
\end{pmatrix}  \in \CC^{2\times 2}\,, $$
which has eigenvalues $0$ and $(\e^{2 \pi i g})_{Q_1}+ (\e^{2 \pi i g})_{Q_3} -(\e^{2 \pi i g})_{Q_2}-(\e^{2 \pi i g})_{Q_4}$, and so as soon as the latter quantity is non-zero, we deduce the existence of a non-zero eigenvalue for  $A_g B$.

Now, we develop the resolvent of $\e^{2 \pi ig}\cL_\alpha$ in a Neumann series 
\begin{align*}
(\lambda -\e^{2 \pi i g}\cL_\alpha)^{-1}&=(\lambda -\e^{2 \pi ig}\cL_\infty)^{-1}\left (\Id+(\e^{2 \pi i g}\cL_\infty-\e^{2 \pi i g}\cL_\alpha)(\lambda -\e^{2 \pi ig}\cL_\infty)^{-1} \right )^{-1}\\
&=(\lambda -\e^{2 \pi ig}\cL_\infty)^{-1}\sum_{n \geq 0}(-1)^n \left ((\e^{2 \pi i g}\cL_\infty-\e^{2 \pi i g}\cL_\alpha)(\lambda -\e^{2 \pi ig}\cL_\infty)^{-1} \right )^n.
\end{align*}
Indeed, by Lemma \ref{lemma:convergence} combined with Lemma \ref{lemma:technical resolvent lemma}, it holds that 
$$
\|(\e^{2 \pi i g}\cL_\infty-\e^{2 \pi i g}\cL_\alpha)(\lambda -\e^{2 \pi ig}\cL_\infty)^{-1}\|_{X \to X}\leq C\alpha^{-\tilde\eta}(1+\|(\lambda -M)^{-1}\|_{\CC^2 \to \CC^2}),
$$
which for all $\alpha$ large enough is strictly less than $1$, uniformly for $\lambda \in \Gamma$. As such, the resolvent of $\e^{2 \pi ig}\cL_\alpha$ exists on $\Gamma$, and we deduce that 
$\|P_\alpha -P_\infty \|_{X \to X} \leq C|\Gamma| \alpha^{-\tilde\eta}$, where $C$ depends only on the choice of $\Gamma$. We claim that this implies that the dimension of the range of $P_\alpha$ is bigger or equal than the one of $P_\infty$, i.e. $\text{dim} (R (P_\alpha))  \geq \text{dim}( R(P_\infty))$, for $\alpha \geq \alpha_0$. We consider  the operator $P_\alpha : R(P_\infty) \to R(P_\alpha)$. Suppose $P_\alpha x=0$. Then, it holds 
$$
\|x\|=\|P_\infty x-P_\alpha x\| \leq C\alpha^{-\tilde\eta}\|x\| \,,
$$
Thus, taking $\alpha $ large enough, we deduce that $x=0$, and hence that $\text{dim}(R(P_\alpha)) \geq \text{dim}(R(P_\infty))=1$. Exchanging the roles of $P_\alpha$ and $P_\infty$ we obtain the opposite inequality, and thus deduce that for all $\alpha$ large enough, $\e^{2 \pi ig}\cL_\alpha$ has a simple eigenvalue with the claimed property.
\end{proof}
It thus remains to prove Lemma \ref{lemma:convergence}, which will be the focus of the following section.

\subsection{Convergence as $\alpha\to\infty$} \label{sec: convergence}

In order to study the convergence of $\cL_\alpha$ as $\alpha \to \infty$, we define an ``intermediate'' operator 
$$ \bar \cL_\alpha h = 2 \begin{pmatrix}
    ( \mathbbm{1}_{\XX_\alpha (\cM_1)} - \mathbbm{1}_{\XX_\alpha (\cM_3)})  
\int_{x\geq 1/2} h^{(1)} + (\mathbbm{1}_{\XX_\alpha (\cM_4)} -  \mathbbm{1}_{\XX_\alpha (\cM_2)})
\int_{x< 1/2} h^{(1)} 
\\
0
\end{pmatrix} \,,$$
where we recall 
\begin{align}
    \cM_1 = \{ x \geq 1/2 \,, y + \alpha x \geq 1/2   \pmod{1}\} \,, \qquad \cM_2 = \{ x < 1/2 \,, y - \alpha x \geq 1/2   \pmod{1}\} \,,
    \\
     \cM_3 = \{ x \geq 1/2 \,, y + \alpha x < 1/2  \pmod{1}\} \,, \qquad \cM_4 = \{ x < 1/2 \,, y - \alpha x < 1/2  \pmod{1}\} \,.
\end{align}
By the triangle inequality and Lemma \ref{lemma:smooth multiplier} we have
\begin{align*}
    \|\e^{2 \pi i g}\cL_\alpha h -\e^{2 \pi i g}\cL_\infty h \|  \leq C \left (  \| \cL_\alpha h -  \bar{ \mathcal{L}}_\alpha h\|_s + \|  \bar{ \mathcal{L}}_\alpha h - \cL_\infty h \|_s  + \| \cL_\alpha h - \cL_\infty h \|_u \right ) \,,
\end{align*}
and so if we prove the convergence to zero of all the summands on the right-hand side, the result follows.
By Proposition \ref{prop:LY} and Lemma \ref{lemma:indicators-y} we deduce
$$ 
\| \cL_\alpha h - \cL_\infty h \|_u \leq \| \cL_\alpha h \|_u + \| \cL_\infty h \|_u \leq C \alpha^{\beta - \sigma} \| h \| + C \alpha^{-1 + \beta} \| h \|\,. 
$$
To prove Lemma \ref{lemma:convergence} it is thus sufficient to prove the following two lemmas and to choose $\tilde \eta = \min \{1-\sigma,\sigma-\beta,q,\beta \}$.
\begin{lemma}
\label{lemma:convergence 1}
There exists $\alpha_0\geq 1$ so that for all $\alpha\geq \alpha_0$, it holds 
$$
\|\mathcal{L}_\alpha h - \bar{\mathcal{L}}_\alpha h\|_s \leq C( \alpha^{-q} + \alpha^{-1 + \sigma } + \alpha^{- \beta} ) \|h\|\,.
$$
\end{lemma}

\begin{lemma}
\label{lemma:convergence 2}
There exists $\alpha_0\geq 1$ so that for all $\alpha\geq \alpha_0$, it holds 
$$
 \|\bar{\mathcal{L}}_\alpha h-\mathcal{L}_\infty h\|_s \leq C (\alpha^{-1+\sigma}+\alpha^{-q}) \| h \|.
$$
\end{lemma}

We now commence with the proof of Lemma \ref{lemma:convergence 1}.
\begin{proof}[Proof of Lemma \ref{lemma:convergence 1}]

We begin by noting that only the first component of $\mathcal{L}_\alpha h$ matters. Indeed,   from the definition of $\cL_\alpha$ in \eqref{eq:Lalpha} and the matrix decomposition \eqref{eq:matrixMi}, all components of \(\alpha^{-2}D\XX_\alpha\), other than the top-left entry, are of order at most \(\alpha^{-1}\). Therefore, we first claim that in this convergence result, we may replace the map $\mathcal{L}_\alpha h$ by the principal part of the map 
\begin{align} \label{d:principal-alpha}
h \mapsto \cL_\alpha^{(p)} h := \begin{pmatrix}
\text{sign} ([D \XX_\alpha]^{(1,1)}\circ \XX_\alpha^{-1}) h^{(1)} \circ \XX_\alpha^{-1}\\
0
\end{pmatrix}\,.
\end{align}
Indeed, the difference between $\mathcal{L}_\alpha h$ and the above map $\cL_\alpha^{(p)} h$ is equal to $(R_\alpha h) \circ \XX_{\alpha}^{-1}$, where $R_{\alpha}$ is a matrix that is constant on each $\cM_\ell$ for $\ell=1,2,3,4$, and is bounded in absolute value by $ |R_\alpha|\leq  \frac{C}{\alpha}$. In particular, since our estimates in Proposition \ref{proposition:main lasota yorke} contain a multiplicative term of a matrix with entries of order unity, it follows that 
$$
\|(R_\alpha h )\circ \XX_\alpha^{-1}\| \leq C \alpha^{-1} \|h \circ \XX_\alpha^{-1}\|\leq C \alpha^{-1} \|h\|.
$$
Hence, in the following estimates, we can replace $\cL_\alpha$ by $\cL_\alpha^{(p)}$ and denote by $c_\ell=\text{sign} ([D \XX_\alpha]^{(1,1)}(x)) \in \{ -1,1\}$, for $x \in \cM_\ell$ for all $\ell =1,2,3,4$. As such, we abuse notation for the rest of the proof and identify $h=(h^{(1)},h^{(2)})$ with its first component.
Throughout the proof, we further denote by $\mathcal{H}_\ell$ the half-plane $\{ x \geq 1/2\}$ or $\{ x\leq 1/2\}$ so that $\cM_\ell \subset \mathcal{H}_\ell$ for any $\ell$.

Pick some curve $W  \in \Sigma$ and $\varphi \in C^q (W)$ with $\| \varphi \|_{\sigma, W}\leq 1$.

If $|W|\leq \alpha^{-2}$, using property \eqref{enum:0} in Lemma \ref{lemma:technical}  we may write each $W=\bigcup_{j=1}^{M_\alpha} W_j$ with $M_\alpha  \leq 4$, where each $W_j$ is entirely contained in one of the smoothness strips $T_\alpha (\cM_\ell)$ and also $\XX_\alpha^{-1} (W_j) = \cup_{k=1}^{N_\alpha} \tilde{W}_{j,k}$ with $N_\alpha \leq 4$, $\tilde W_{j,k} \in \Sigma$ and $|\tilde W_{j,k}|\leq 2 \alpha^2 |W|$. Then,
\begin{align} \label{eq:split-comp}
    \left |\int_{W_j} (\cL_\alpha^{(p)} h-\bar{\cL}_\alpha h) \varphi \, \dd \HH^1 \right | \leq  \left | \int_{W_j} h^{(1)} \circ \XX_\alpha^{-1} \varphi \, \dd \HH^1  \right  |+ \left | \int_{W_j} \left ( \int_{\mathcal{H}_\ell} h \right  ) \varphi \, \dd \HH^1 \right | 
\end{align} 
and for the first term
\begin{align*}
     \left | \int_{W_j} h^{(1)} \circ \XX_\alpha^{-1} \varphi \, \dd \HH^1  \right  | & \leq \frac{2}{\alpha^2} \sum_k  \left | \int_{\tilde W_{j,k}} h^{(1)} \varphi \circ \XX_\alpha \, \dd \HH^1 \right | 
     \\
     & \leq 2 \alpha^{-2} \max_{k}  \| \varphi \circ \XX_\alpha \|_{\sigma, \tilde{W}_{j,k}}  \| h \|_s  
     \\
     & \leq C \alpha^{-2 + 2 \sigma} \| \varphi\|_{\sigma, W}  \| h \|_s \,,
\end{align*}
where in the last inequality we have used Lemma \ref{lemma:property-test} and $|\tilde W_{j,k}|^\sigma \leq 2^{\sigma} \alpha^{2\sigma} |W|^{\sigma}$ for any $k$. From Lemma \ref{lemma:indicators-y} the estimate of the second terms follows 
\begin{align} \label{eq:comp-lemma}
    \left | \int_{W_j} \left ( \int_{\mathcal{H}_\ell} h  \right  ) \varphi \, \dd \HH^1 \right |  \leq C \| \varphi \|_{\sigma, W} |W|^{-\sigma} |W_j| |h|_w \leq C \alpha^{-2 + 2 \sigma} |h|_w \,.
\end{align}
If $\alpha^{-2}\leq |W| \leq \alpha^{-1}$ we may write $W=\bigcup_{j=1}^{M_\alpha} W_j$ with $M_\alpha  \leq 5$, where each $W_j$ is entirely contained in one of the smoothness strips $T_\alpha (\cM_\ell)$ and also $\XX_\alpha^{-1} (W_j) = \cup_{k=1}^{N_\alpha} \tilde{W}_{j,k}$ with  $N_\alpha \leq 4 \lceil \alpha^2 |W_j| \rceil \leq 4 \lceil \alpha^2 |W| \rceil $ for any $j$. Hence, we use the bound \eqref{eq:split-comp} and  for any $j \leq 5$ we  estimate the first term 
\begin{align*}
     \left | \int_{W_j} h^{(1)} \circ \XX_\alpha^{-1} \varphi \, \dd \HH^1  \right  | & \leq \frac{C}{\alpha^2} \sum_{k=1}^{N_\alpha}  \left | \int_{\tilde W_{j,k}} h^{(1)} \varphi \circ \XX_\alpha \, \dd \HH^1 \right | 
     \\
     & \leq C |W| \max_{k}  \| \varphi \circ \XX_\alpha \|_{\sigma, \tilde{W}_{j,k}}  \| h \|_s  
     \\
     & \leq C \alpha^{-1 + \sigma} \| \varphi\|_{\sigma, W}  \| h \|_s \,.
\end{align*}
The estimate of the second term in \eqref{eq:split-comp} is identical to that in \eqref{eq:comp-lemma}, obtaining in this case a bound of order $\alpha^{-1 + \sigma} \|h\|_s$.

Thus it remains to deal with the case when $|W| \geq \frac{1}{\alpha}$.
We apply again Lemma \ref{lemma:technical} \eqref{enum:0} and we write $W = \cup_{j=1}^{M_\alpha} W_j$ with $M_\alpha \leq 4 \lceil \alpha |W| \rceil $.
We aim to estimate
\begin{align} \label{eq:estimate-1}
\sum_{j=1}^{M_\alpha}\left | \int_{W_j} (\mathcal{L}_\alpha h - \bar{\mathcal{L}}_\alpha h)\varphi \, \dd \HH^1 \right |.
\end{align}
We begin by noting that we can further select among the curves $W_j$ those that traverse the smoothness strips $\XX_\alpha (\cM_\ell)$ for some $\ell=1,2,3,4$ regularly. More precisely, we denote the set of indices $j \in \mathcal{J}$ for the lines  $W_j$  parametrized by  curve $\gamma$ so that $\gamma([t_j, t_{j+1}]) = W_j$ so that $\gamma(t_{j}) \in \{ x- \alpha y = \frac{k+1}{2}\}$ and $\gamma(t_{j+1}) \in \{ x- \alpha y = \frac{k}{2} \}$ for some $k \in \ZZ$ or $\gamma(t_{j}) \in \{ x+\alpha y = \frac{k}{2}\}$ and $\gamma(t_{j+1}) \in \{ x+\alpha y = \frac{k+1}{2} \}$ for some $k \in \ZZ$.
We note that the cardinality $\# \mathcal{J} \leq 5 \alpha |W|$ and $ \# \mathcal{J}^c  \leq 6 $.
Hence, we write
$$W=  \bigcup_{j \in \mathcal{I}} W_j \cup  \bigcup_{j \notin \mathcal{J}} W_j \,.  $$
Since all $W_j$ corresponding to $j \in \mathcal{J}^c$ have length at most $\frac{1}{\alpha}$, we estimate by the previous step 
\begin{align*}
\sum_{j \notin \mathcal{J}} \left | \int_{W_j}(\mathcal{L}_\alpha^{(p)} h - \bar{\mathcal{L}}_\alpha h)\varphi \, \dd \HH^1\right | 
&\leq 2\sum_{j \notin \mathcal{J}} \alpha^{-1+\sigma} \|\varphi\|_{\sigma, W}\|h\|_s \leq C\alpha^{-1+\sigma} \|h\|_s\,.
\end{align*}
Hence, we reduce to estimating 
$\int_{W_j}(\mathcal{L}_\alpha h- \bar{\mathcal{L}}_\alpha h) \varphi $, for any $j \in \mathcal{J}$.  Writing out the integral, we need to estimate
$$
c_j \int_{W_j}\left( h \circ \XX_\alpha^{-1}- 2\int_{\mathcal{H}_j} h \right) \varphi \, \dd \HH^1 .
$$
Without loss of generality we suppose for instance that $W_j \in \XX_\alpha (\cM_1)$, the other three cases are similar. Hence $c_1 =1$ and $\mathcal{H}_1 = \{ x \geq 1/2\}$. From the definition of the map $\XX_\alpha$ it follows that $\XX_{\alpha}^{-1} (\gamma(t_j)) \in \{ x=1\}$ and  $\XX_{\alpha}^{-1} (\gamma(t_{j+1})) \in \{ x=1/2\}$. In this case, if the tangent vector of $W$ is $\mathbf{v}= (a,b)$ with $|a|\leq 2 \alpha^{-1} b$, we  have that
$$t_{j+1} - t_j = \frac{1}{2 (\alpha b - a)} \,.$$
 Furthermore, on $\XX_\alpha (\cM_1)$, the inverse of $\XX_\alpha$ is the linear map (see \eqref{eq:matrixMi})
$$ D \XX_\alpha^{-1}|_{\XX_\alpha (\cM_1)} = \begin{pmatrix}
1 & -\alpha 
\\
-\alpha & 1+ \alpha^2 
\end{pmatrix}$$
then
$$\XX_\alpha^{-1} (\mathbf{v}) = (a- \alpha b, -\alpha a + (1+ \alpha^2)b )\,.$$
Hence viewing the curve of $\RR^2$ and denoting $(x(t), y(t)) = \XX_\alpha^{-1} (\gamma(t))$ we deduce that 
$$|x(t_{j+1})- x(t_{j})| = \frac{1}{2} \,, \qquad \left| y(t_{j+1}) - y(t_j) - \frac{\alpha}{2}\right|\leq 1\,. $$
 We  consider $x_0 = \max \{ x\in [1/2, 1] : (x,0) \in \XX_\alpha^{-1} (W_i) \}$
and define the partition 
$$ \XX_\alpha^{-1} (W_j) = \bigcup_{j=1}^{\frac{\alpha}{2} -2} \tilde W_{j,k} \cup \tilde U_{in} \cup \tilde{U}_{fin} $$
where $\tilde U_{in} $ is the curve of length $|U_{in}| \leq 1$ connecting $\XX_\alpha^{-1} (\gamma (0))$ to $(x_0, 0)$ on $\TT^2$ and $\tilde W_{j,k}$ are subsequent curves of length $1$, whereas $\tilde{U}_{fin}$ is the last curve of length $|\tilde{U}_{fin}|\leq 1$.  We define the time $S_j$ to be how long it takes   $ \XX_\alpha^{-1} (\gamma(\cdot)) $ to move the $y$-component by $1$, namely $S_j=\frac{|\XX_\alpha^{-1} \mathbf{v}|}{(1+\alpha^2)b-\alpha a}$. Now, set $x_k$ to be the $x$ component of $\XX_\alpha^{-1} (\gamma(t))$ after the $y$-component has increased by $k$, which is given by 
$$
x_k=x_0+k S_j\frac{a-\alpha b}{|\XX_\alpha^{-1} \mathbf{v}|}=x_0+k\frac{a-\alpha b}{(1+\alpha^2)b-\alpha a}=x_0-\frac{k}{\alpha}\left(1+\frac{b}{(1+\alpha^2)b-\alpha a}\right),
$$
for all $k \in \{ 1, \ldots , \frac{\alpha}{2} -2 \}$. Indeed, 
$$x_{\alpha/2}=x_0-\frac{1}{2}\left (1+\frac{b}{(1+\alpha^2)b-\alpha a} \right )<1/2,
$$
since $b>0$, but 
$$
x_{\alpha/2-2}=x_0-\left(\frac{1}{2}-\frac{2}{\alpha}\right)\left (1+\frac{b}{(1+\alpha^2)b-\alpha a} \right )>1/2,
$$
upon noting that $1-x_0\leq \frac{b}{(1+\alpha^2)b-\alpha a}$. Thus, there are at least $\alpha/2-2$ pieces, but no more than $\alpha/2-1$. Upon possibly discarding a single piece, we thus assume there are $\alpha/2-2$ pieces.

We define the vertical lines $\tilde V_{j,k } = \{ (x,y): x = x_k  \,, y \in [0,1] \}$.
Using that the second component of the base point of the curve $\tilde V_{j,k+1}$ is $0$ (which equals $k $ mod $1$), and the second component of the base point of $\tilde W_{j,k+1}$ is $k\frac{- \alpha a + (1+ \alpha^2) b}{|\XX_\alpha^{-1} (\mathbf{v})|}$, we have that    
 $$\left | k- k\frac{- \alpha a + (1+ \alpha^2) b}{|\XX_\alpha^{-1} (\mathbf{v})|} \right  | \leq   \frac{4}{\alpha^2} |k| \leq 2 \alpha^{-1} \,.$$
Similarly, the difference in $x$-components is given by 
$$
\left |k \frac{a-\alpha b}{|\XX_\alpha^{-1} \mathbf{v}|}(1-S_j) \right |\leq C \alpha^{-2}\,.
$$
Hence, it is straightforward to check that
$$\dd_\Sigma (\tilde W_{j,k} , \tilde V_{j,k} ) \leq 2\alpha^{-1} \,.
$$
From these definitions, for the particular case of $W_j$ considered, using the change of variables $\XX_\alpha^{-1}$ we have
\begin{align*}
\left | \int_{W_j}\left( h \circ \XX_\alpha^{-1}- 2\int_{x \geq 1/2} h \right) \varphi \, \dd \HH^1\right | & \leq \frac{2}{\alpha^2} \left | \sum_{k} \int_{\tilde W_{j,k}}\left(h -2\int_{x \geq 1/2} h \right) \varphi \circ \XX_\alpha \, \dd \HH^1 \right |
\\
&  \quad +  \frac{2}{\alpha^2} \left | \int_{\tilde U_{in} }\left(h -2\int_{x \geq 1/2} h \right) \varphi \circ \XX_\alpha \, \dd \HH^1 \right |
\\
& \quad +  \frac{2}{\alpha^2} \left |  \int_{\tilde U_{fin} }\left(h -2\int_{x \geq 1/2} h \right) \varphi \circ \XX_\alpha \, \dd \HH^1 \right | = I+ II + III
\end{align*}
and  the last two terms are bounded by $C \alpha^{-2 + \sigma} \|h\|_s $, so upon reintroducing the summation in $j \in \mathcal{J}$ these terms create an error of order $ \mathcal{O} (\alpha^{-1+\sigma})$. Let us thus discuss the first term $I$.
Fix $\bar x \in \XX_\alpha^{-1} (W_j)$ and sum and subtract $\varphi \circ \XX_\alpha (\bar x)$  so that 
\begin{align*}
I & \leq \left | \sum_{k} \frac{1}{\alpha^2}\int_{\tilde W_{j,k}}\left(h -2\int_{x \geq 1/2} h \right) \varphi \circ \XX_\alpha (\bar x)\, \dd \HH^1 \right | 
\\
& \quad  + \left | \sum_{k} \frac{1}{\alpha^2}\int_{\tilde W_{j,k}}\left(h -2\int_{x \geq 1/2} h \right) (\varphi \circ \XX_\alpha - \varphi \circ \XX_\alpha (\bar x))\, \dd \HH^1 \right | =I_1 + I_2 \,.
\end{align*} 
We observe that  the length of $W_j$ satisfies $|W_j| \leq 2 \alpha^{-1}$. Hence
$$ \sup_{k} \sup_{x \in \tilde W_{j,k}} |\varphi \circ \XX_\alpha (x) -  \varphi \circ \XX_\alpha (\bar x)| \leq  \sup_{x,y \in W_{j}} |\varphi (x) -  \varphi (y)| \leq C \alpha^{-q} \|\varphi \|_{C^q} \,.$$
 Using this and Lemma \ref{lemma:property-test} to estimate $[\varphi \circ \XX_\alpha]_{C^q}$ we deduce that $\| \varphi \circ \XX_\alpha - \varphi \circ \XX_\alpha (\bar x) \|_{C^q}\leq \frac{C}{\alpha^q} \|\varphi \|_{C^q}$. Thus,  denoting by $\tilde{h} = h -2\int_{x \geq 1/2} h $, observing that $\| \tilde {h} \|_s \leq C \|h\|_s$ by Lemma \ref{lemma:indicators-y} and using that the sum over $k$ contains at most $\alpha/2$ terms, we deduce that the second term is bounded by 
\begin{align*}
    I_2 & \leq  \alpha^{-1} \sup_{j,k} \left |\int_{\tilde W_{j,k}}\left(h -2\int_{x \geq 1/2} h \right) (\varphi \circ \XX_\alpha - \varphi \circ \XX_\alpha (\bar x)) \, \dd \HH^1 \right | 
    \\
    & \leq C \alpha^{-1} \|h\|_s  \sup_{j,k} \| \varphi \circ \XX_\alpha - \varphi \circ \XX_\alpha (\bar x) \|_{\sigma, \tilde{W}_{j,k}} \leq C\alpha^{-q-1} |W|^{-\sigma} \| \varphi \|_{\sigma, W}\,.
\end{align*}
Hence, reintroducing the sum over $j \in \mathcal{J}$ and noting that $\# J \leq 5 \alpha |W| $, this term becomes an error of order  $\mathcal{O} (\alpha^{-q})$. 

For the term $I_1$ we note that $|\tilde W_{j,k}|=1$, and thus $\int_{\tilde W_{j,k}} \int_{ x \geq 1/2} h =\int_{ x \geq 1/2} h . $ Furthermore, it holds that 
$$
\left | \int_{\tilde W_{j,k}} h \, \dd \HH^1-\int_{\tilde V_{j,k}} h  \, \dd \HH^1\right |\leq C \alpha^{-\beta}\|h\|_u,
$$
by definition of the strong unstable norm. As such, using that $k \leq N_\alpha \leq C \alpha$ we deduce
$$
I_1 \leq \alpha^{-\beta-1}\|h\|_u |W|^{-\sigma}\|\varphi\|_{\sigma,W}+\frac{1}{\alpha^2}\left |\sum_{k} \left (\int_{\tilde V_{j,k}} h -2 \int_{x \geq 1/2} h  \right ) \right ||W|^{-\sigma}\|\varphi\|_{\sigma,W},
$$
and reintroducing the summation in $j$, the first of these terms is bounded by $\alpha^{-\beta} \|h\|_u$. Thus, if we prove that 
$$
\frac{1}{\alpha^2}\left |\sum_{k} \left (\int_{\tilde V_{j,k}} h \, \dd \HH^1 -2 \int_{x \geq 1/2} h  \right ) \right |\leq C \alpha^{-1-\beta} \|h \|.
$$
reintroducing the summation over $j \in \mathcal{J}$ and using that $\# \mathcal{J} \leq 4\alpha |W|$ we bound the second summand of $I_1$ and we conclude the proof.
To do so, we recall that the sum runs over $k=1, \ldots, \frac{\alpha}{2} -2$ and the base points of $\tilde V_{j,k}$ are $x_k = x_{k-1} -\frac{1}{\alpha}(1+\frac{b}{(1+\alpha^2)b-\alpha a})$. From now on, we will repeatedly use the fact that
\[
\left|\frac{1}{\alpha}(1+\frac{b}{(1+\alpha^2)b-\alpha a})-\frac{1}{\alpha}\right|\le C \alpha^{-3} \,.
\]
Using the notation $\fint_{A} = \frac{1}{|A|}\int_A $ for any set $A$,  we notice that 
\begin{align*}
    &\sum_{k} 2 \int_{x\geq 1/2} h  = 2\left(\frac{\alpha}{2} -2\right) \int_{x \geq 1/2} h 
    \\
    & \qquad= (\alpha - 4) \left (  \sum_{k=1}^{\alpha/2-2} \int_{x_k}^{x_{k-1}}\int_{0}^1 h(x,y) \dd y \dd x + \int_{x_0}^{1} \int_0^1 h (x,y) \dd y \dd x  + \int_{1/2}^{x_{\frac{\alpha}{2} -2}} \int_0^1 h (x,y) \dd y \dd x   \right )
    \\
    & \qquad=  \left ( \sum_{k=1}^{\alpha/2-2}  \fint_{x_k}^{x_{k-1}}\int_{0}^1 h(x,y) \dd y \dd x  \right )+ \mathcal{O}_{\alpha \to \infty}(1)|h|_w \,.
\end{align*} 
Hence, we deduce that 
\begin{align*}
    &\frac{1}{\alpha^2}  \left |  \sum_{k=1}^{\alpha/2-2}    \left ( \int_{\tilde V_{j,k}}h \, \dd \HH^1 -2 \int_{x \geq 1/2} h  \right ) \right |  \leq  \frac{1}{\alpha^2} \sum_{k=1}^{\alpha/2-2}\left | \int_{\tilde V_{j,k}} h \, \dd \HH^1 -   \fint_{x_k}^{x_{k-1}}\int_{0}^1 h(x,y) \dd y \dd x \right | + \frac{C}{\alpha^2} |h|_w
    \\
    & \qquad\leq \frac{1}{\alpha}\sum_{k=1}^{\alpha/2-2}\int_{x_k}^{x_{k-1}}  \left |\int_{0}^1 h(x_k,y) \dd y -\int_{0}^1 h(x,y) \dd y  \right | \dd x \leq C \alpha^{-1-\beta} \| h \|_u + \frac{C}{\alpha^2} |h|_w \,,
\end{align*}
where in the last inequality we used the strong unstable norm together with the fact that 
$$\dd_\Sigma \left(\tilde{V}_{j,k}, \{ (x,y): y \in [0,1] \}\right) \leq 2 \alpha^{-1}\,, \qquad \forall x \in [x_k, x_{k-1}] \,, $$ 
concluding the proof.
\end{proof}

\begin{proof}[Proof of Lemma \ref{lemma:convergence 2}] 
This lemma is a direct consequence of the following property: there exists $\alpha_0>0$ such that for all $\alpha \geq \alpha_0$ and all vectors $\mathbf{w} \in \mathbb{C}^2$ with $|\mathbf{w}|=1$ it holds true 
$$
\|\mathbf{w}(\mathds{1}_{\XX_\alpha(\cM_1)} -1/2\mathds{1}_{y \geq 1/2})\|_s \leq C(\alpha^{-1+\sigma}+\alpha^{-q}), \quad \|\mathbf{w}(\mathds{1}_{\XX_\alpha(\cM_2)} -1/2\mathds{1}_{y \geq 1/2})\|_s \leq C(\alpha^{-1+\sigma}+\alpha^{-q}),
$$
$$
\|\mathbf{w}(\mathds{1}_{\XX_\alpha(\cM_3)} -1/2\mathds{1}_{y < 1/2})\|_s \leq C(\alpha^{-1+\sigma}+\alpha^{-q}), \quad \|\mathbf{w}(\mathds{1}_{\XX_\alpha(\cM_4)} -1/2\mathds{1}_{y < 1/2})\|_s \leq C(\alpha^{-1+\sigma}+\alpha^{-q}).
$$
We shall provide the proof for $\XX_\alpha(\cM_1)$ and  $\mathbf{w}=(1,0)$. The general proof is entirely analogous. Let $W \in \Sigma$. If $|W|\leq \frac{1}{\alpha}$, we split $W=\bigcup_{j=1}^M W_j$, with $M\leq 4$ and  using  $\|\varphi\|_{\sigma, W}\leq 1$ we obtain
\begin{align*}
\left |\int_{W} (\mathds{1}_{\XX_\alpha(\mathcal{M}_1)}-1/2\mathds{1}_{y \geq 1/2}) \varphi \, \dd \HH^1 \right | &\leq \sum_{i=j}^M \int_{W_j}\left |(\mathds{1}_{\XX_\alpha(\mathcal{M}_1)}-1/2\mathds{1}_{y \geq 1/2}) \varphi \, \dd \HH^1 \right |\\
& \leq 2M |W_j| |W|^{-\sigma}\leq 2M |W|^{1-\sigma}\leq 2M\alpha^{-1+\sigma}.
\end{align*}
Therefore, we may suppose that $|W|\geq 1/\alpha$, and furthermore we may assume that $W \subset \{ y \geq 1/2\}$, since $\XX_\alpha (\cM_1) \subset \{ y \geq 1/2\}$.  Then, $W$ may be parameterised as $\{(x_1,y_1)+t \vvv: t \in [0,T)\}$ with $T \leq 1$ and $|\vvv|=1$. By Lemma \ref{lemma:technical} we can split this curve into its connected components as $W=\bigcup_{j=0}^{M-1} W_j$, where $M \leq 4\alpha|W|$ and $W_j=\{(x_1,y_1)+t \vvv :t \in [t_j,t_{j+1}]\}$, with $t_0=0$ and $|t_{j+1} - t_j| = \frac{1}{2 (a_1 + \alpha b_1)}  \leq \alpha^{-1}$ for any $j$, where we have used $\vvv= (a_1,b_1)$, and  $W_j \subset \overline{ \XX_\alpha (\cM_\ell)}$ for some $\ell$.
The integrals over $W_0$ and $W_{M -1}$ are treated as an error since $|W_0| , |W_{M-1}| \leq \frac{2}{\alpha}$, namely for any $\varphi$ with $\|\varphi\|_{\sigma,W}\leq 1$, we have
\begin{align*}
     \left |\int_{W_0} \left (\mathds{1}_{\XX_\alpha(\cM_1)} -\frac12\mathds{1}_{y \geq 1/2}\right )\varphi \, \dd \HH^1  \right | & \leq C |W_0|^{1-\sigma} \| \varphi \|_{\sigma,W}  \leq C\alpha^{-1+\sigma} \,,
\end{align*}
and similarly for $W_{M-1}$.

Without loss of generality we now suppose that $M$ is even and  
$$\mathds{1}_{\XX_\alpha(\cM_1)} \cap W_j=\emptyset \quad \text{for } j \text { even}\,, \qquad \mathds{1}_{\XX_\alpha(\cM_1)} \cap W_j=W_j \quad  \text{for } j \text { odd} \,.$$ 
Thus, from the previous estimate, for any $\varphi \in C^q (W)$ with $\|\varphi\|_{\sigma,W}\leq 1$
we have 
$$
\int_{W} \left (\mathds{1}_{\XX_\alpha(\cM_1)} -1/2\mathds{1}_{y \geq 1/2}\right )\varphi \, \dd \HH^1 =\frac{1}{2} \sum_{j=1}^{\frac{M}{2} -2}\int_{W_{2j +1}} \varphi \, \dd \HH^1-\int_{W_{2j}} \varphi\, \dd \HH^1 +\mathcal{O}(\alpha^{\sigma-1}).
$$
Using that $|W_{2j+1}| = |W_{2j}|= \frac{1}{2 (a_1 + \alpha b_1)}  \leq \alpha^{-1}$ for any $j$, we deduce that  
$$ \left | \int_{W_{2j +1}} \varphi\, \dd \HH^1 -\int_{W_{2j}} \varphi\, \dd \HH^1 \right | \leq C \alpha^{-q-1}\|\varphi\|_{\sigma,W}|W|^{-\sigma}\,,$$
and so reintroducing the summation in $j$, and noting that $M \leq 4\alpha |W|$, we obtain
$$
\left |\int_{W} \left (\mathds{1}_{\XX_\alpha(\cM_1)} -\frac12\mathds{1}_{y \geq 1/2}\right )\varphi\, \dd \HH^1\right | \leq C(\alpha^{\sigma-1}+\alpha^{-q})\|\varphi\|_{\sigma,W},
$$
which concludes the proof.
\end{proof}

\subsection{Relationship to the flux conjecture} \label{subsec:flux}

In this subsection, we show that the velocity field $u_{\alpha, N}$ defined in \eqref{eq:u_def}, for $\alpha \in 2 \mathbb{N}$ sufficiently large, is a \emph{perfect dynamo}. This establishes that the magnetic growth in our model is not merely a result of fine-scale energy accumulation, but involves the growth of a macroscopic flux.

\begin{proposition}
\label{prop:flux conjecture}
There exists $\alpha_0 \geq 1$ so that for every $\alpha \geq \alpha_0$ the following holds true. There exists a smooth vector field $\psi \in C^\infty(\mathbb{T}^2 ; \mathbb{R}^2)$ such that
$$
\sup_{B_{\ini} \in L^2(\mathbb{T}^2)} \liminf_{n \to \infty} \frac{1}{n}\log\left |  \int_{\mathbb{T}^2} (\alpha^2 \e^{2 \pi i g}\mathcal{L}_\alpha)^n B_{\ini} \cdot \psi \,  \right | > 0\,.
$$
\end{proposition}

Note that Corollary \ref{corollary:perfect-dynamo} follows directly from Proposition \ref{prop:flux conjecture}. Indeed, in the context of Definition \ref{d:perfect-dynamo}, we may consider the smooth $3d$ test function $\Psi(x,y,z) = \e^{-2\pi i z} (\psi(x,y), 0)$. The factor $\e^{-2\pi i z}$ is chosen precisely to compensate for the phase factor in our magnetic field ansatz \eqref{d:initial} and the push-forward formula \eqref{eq:comp-intro}, effectively reducing the $3d$ flux integral to the $2d$ pairing analyzed below.

\begin{proof}
Any element of the anisotropic space $X$ naturally defines a distribution in $(C^q(\mathbb{T}^2))^*$. Here, we denote by $C^q(\mathbb{T}^2)$ the closure of Lipschitz continuous functions in the $C^q$ norm. Indeed, for any $h \in C^1(\mathbb{T}^2)$ and $\psi \in C^q(\mathbb{T}^2)$, we have the estimate
$$
\left |\int_{\mathbb{T}^2} h \cdot \psi   \right | \leq C \|\psi\|_{C^q(\mathbb{T}^2)} \|h\| \,,
$$
where $\|h\|$ is the strong norm \eqref{eq:Xstrongnorm} of our anisotropic space.  It is straightforward to prove that the inclusion $(X,\|\cdot \|_X) \hookrightarrow  (C^q(\mathbb{T}^2))^*$ is injective, see for instance  \cite{demers_liverani}. Consequently, any non-zero element $h \in X$ acts as a non-zero functional on $C^q(\mathbb{T}^2)$.

By Theorem~\ref{thm:main eigenvalue}, for sufficiently large $\alpha$, the operator $\alpha^2 \e^{2 \pi i g}\mathcal{L}_\alpha$ admits a unique, simple leading eigenvalue $\lambda$ of maximum modulus, with $|\lambda| \geq \frac{\alpha^2}{4} > 1$. Let $P_\alpha$ be the Riesz projector onto the corresponding eigenspace. Since $P_\alpha$ is a non-zero projector and $C^1(\mathbb{T}^2)$ is dense in $X$, there must exist an initial datum $B_{\ini} \in C^1(\mathbb{T}^2) \subset L^2(\mathbb{T}^2)$ such that $P_\alpha B_{\ini} \neq 0$.

By the injectivity of the inclusion into the dual space and the density of $C^\infty(\mathbb{T}^2)$ in $C^q(\mathbb{T}^2)$,\footnote{This is only true because we defined $C^q(\mathbb{T}^2)$ as the completion of Lipschitz functions in the $C^q$ norm.} we can select a test function $\psi \in C^\infty(\mathbb{T}^2)$ such that the pairing is non-vanishing; without loss of generality, we assume
$$
\left | \int_{\mathbb{T}^2} P_\alpha B_{\ini} \cdot \psi  \right | \geq 1 .
$$
We now employ the spectral decomposition  
$$
(\alpha^2 \e^{2 \pi i g} \cL_\alpha)^n B_{\ini} = \lambda^n P_\alpha B_{\ini} + (\alpha^2 \e^{2 \pi i g} \cL_\alpha)^n (\Id - P_\alpha) B_{\ini}.
$$
Since the spectral radius of $\alpha^2 \e^{2 \pi i g} \cL_\alpha (\Id-P_\alpha)$ is bounded by $\alpha^{2-\eta}$ for some $\eta > 0$, see Proposition \ref{proposition:main lasota yorke} and its proof, the spectral radius formula implies that 
$$
\| (\alpha^2 \e^{2 \pi i g} \cL_\alpha)^n (\Id - P_\alpha) B_{\ini} \|  \leq C_\alpha \alpha^{(2-\eta)n} \| (\Id - P_\alpha) B_{\ini} \| 
$$ 
for all $n$. It follows that for sufficiently large $n$, the dominant eigenmode dictates the growth:
\begin{align*}
\left|\int_{\mathbb{T}^2} (\alpha^2 \e^{2 \pi i g} \cL_\alpha)^n B_{\ini} \cdot \psi   \right |  
&\geq |\lambda|^n \left| \int_{\mathbb{T}^2} P_\alpha B_{\ini} \cdot \psi   \right| - \left| \int_{\mathbb{T}^2} (\alpha^2 \e^{2 \pi i g} \cL_\alpha)^n (\Id - P_\alpha) B_{\ini} \cdot \psi  \right| \\
&\geq |\lambda|^n - C_\alpha \alpha^{(2-\eta) n} \geq \frac{1}{2} |\lambda|^n .
\end{align*}
Since $|\lambda| \geq \frac{\alpha^2}{4} >1$, the growth rate is strictly positive, completing the proof.
\end{proof}

\section{Continuous in time diffusion} \label{sec:heat}

In this section, we prove Proposition \ref{prop: uniform Lasota Yorke}. 
For clarity, we adopt the following convention: constants denoted by 
$C_{\alpha,g}>0$ may depend on the parameter $\alpha \geq 1$ defining the velocity field in \eqref{eq:u_def} and on the shear $g$.  Constants denoted by $C>0$ are independent of $\alpha$ and $g$.

We recall the two-dimensional solution operator introduced in \eqref{d:eps-2d-operator}. For $h : \TT^2 \to \RR^2$, we study the operator
\[
h \mapsto \cL_{g, \eps} \circ \e^{- 8\eps \pi^2 } \cL_{\alpha, \eps} \circ \e^{N\eps (\Delta - 4\pi^2)} h.
\]
As in the inviscid case, the essential dynamics arise primarily from the operator 
$\cL_{\alpha,\eps}\e^{N\eps \Delta}$, while the phase operator $\cL_{g,\eps}$ serves mainly to isolate the leading resonance from the essential spectrum. Accordingly, the present section is devoted primarily to the analysis of the map
\[
h \mapsto \cL_{\alpha,\eps}\e^{N\eps \Delta}h.
\]
The guiding heuristic is that, away from the forward singular set $S^\star$ defined in \eqref{d:forward-singular}, the operator $\cL_{\alpha,\eps}\e^{N\eps \Delta}$ is well approximated by a Gaussian Green's matrix composed with $T_\alpha^{-1}$. In order to make this intuition precise, in Section \ref{sec:model problem} we introduce a so-called \emph{model operator} $\cQ_{\alpha,\eps}$, which is formally of the form of a heat pulse, followed by composition with $T_\alpha^{-1}$, and decompose
\begin{equation}
\label{eq:operator splitting}
\cL_{\alpha,\eps}\e^{N\eps \Delta}
=
\mathcal{Q}_{\alpha,\eps} \e^{N \eps \Delta}
+
\mathcal{E}_{\alpha,\eps}\e^{N\eps \Delta},
\end{equation}
where $\mathcal{E}_{\alpha,\eps}=\cL_{\alpha,\eps}- \mathcal{Q}_{\alpha,\eps}$ denotes the error between the full operator and its model approximation. We recall that $\sigma, \beta \in (0,1)$ are fixed in \eqref{eq:choice-parameters}.
We have the following result.

\begin{proposition}
\label{prop:error decomposition}
There exist constants $C>0$ and $C_\alpha>0$ such that the following holds. For every $\eps>0$, there exist operators $\mathcal{Q}_{\alpha,\eps}, \mathcal{E}_{\alpha,\eps} : X \to X$,
so that $\cL_{\alpha,\eps}=Q_{\alpha,\eps}+\mathcal{E}_{\alpha,\eps}$, and for all $\eps N<1$ there holds
\begin{align}
&\|\mathcal{Q}_{\alpha,\eps} \circ \e^{N\eps \Delta} h\|
\leq C\alpha^{-\eta}\|h\|+C|h|_w,
\quad &&
|\mathcal{Q}_{\alpha,\eps} \circ \e^{N\eps \Delta} h|_w
\leq C|h|_w,
\\
& \|\mathcal{E}_{\alpha,\eps} \circ \e^{N\eps \Delta} h\|
\leq C_\alpha N^{-\frac{1-\sigma}{2}}\|h\|,
\quad &&
|\mathcal{E}_{\alpha,\eps} \circ \e^{N\eps \Delta} h|_w
\leq C_\alpha|h|_w.
\end{align}
Furthermore, the following weak--strong convergence estimates hold:
\[
|\mathcal{Q}_{\alpha,\eps} \circ \e^{N\eps \Delta} h
-\cL_\alpha \circ \e^{N\eps \Delta} h|_w
\leq C_\alpha \eps^{\beta/2}\|h\|, \qquad |\mathcal{E}_{\alpha,\eps} \circ \e^{N\eps \Delta} h|_w
\leq C_\alpha\eps^{\beta/2}\|h\|.
\]
\end{proposition}

Once Proposition \ref{prop:error decomposition} is established, an application of the triangle inequality immediately yields the following proposition, from which Proposition \ref{prop: uniform Lasota Yorke} will be deduced. 

\begin{proposition} \label{prop:lasota-continuous-diffusion}
For any $\alpha \geq 1$, there exist constants $C>0$ and $C_\alpha>0$ such that, for all $N\eps<1$, and for all $h \in X$ the following hold true 
\[
\|\cL_{\alpha, \eps} \circ \e^{N\eps \Delta } h\| 
\leq C \bigl(\alpha^{-\eta}+C_{\alpha}N^{-\frac{1-\sigma}{2}}\bigr)\|h\|
+ C |h|_w \,,
\]
and 
$$ |\cL_{\alpha, \eps} \circ \e^{N\eps \Delta } h |_w\leq  C |h|_w  \,.$$
Furthermore, for every $\alpha \in 2\NN$ there exists a constant {$C_{\alpha, N}>0$}, depending  on $\alpha, N$, such that
\[
|(\cL_{\alpha, \eps}-\cL_\alpha)\e^{N\eps \Delta}h |_{w} 
\leq C_{\alpha, N} \eps^{\beta/2}\|h\|_u.
\]
\end{proposition}
The majority of this section will be devoted to the proof of Proposition \ref{prop:error decomposition}.

\subsection{Gaussian bounds} 
\label{sec: Gaussian bounds}
In our perturbative analysis, Gaussian kernels will appear frequently. Indeed, a central observation underlying our approach is that convolution with Gaussians is particularly well behaved on the anisotropic spaces introduced in Section \ref{sec:anisotropic}. This section is therefore devoted to collecting several useful estimates concerning the action of Gaussian kernels.

Throughout the following sections, $\bar \theta[y;Q](x)$ denotes a Gaussian on $\RR^d$ centred at $y$ with covariance matrix $Q$, i.e.
$$
\bar\theta[y; Q](x) = \frac{1}{\sqrt{(2\pi)^d \det(Q)}} \exp\left( -\frac{1}{2} (x - y)^\top Q^{-1} (x - y) \right),\qquad x,y \in \RR^d.
$$
For $y \in \TT^d$, we define the associated periodic Gaussian on the torus by
\[
\theta [y;Q](x) = \sum_{k \in \ZZ^d} \bar \theta [y;Q](x+ k),
\qquad x \in \TT^d.
\]
We begin with the following two fundamental lemmas.
\begin{lemma} \label{lemma:gaussian-bound}
Let $\theta \in L^1 (\TT^2)$ satisfy $\| \theta\|_{L^1 (\TT^2)} \leq 1$.
Then the convolution operator $\mathcal{G} : X \to X$ defined by
\[
\mathcal G h (x) = \int_{\TT^2} \theta(x-y) h(y) \,\dd y
\]
satisfies, for all $h \in X$,
\begin{equation}\label{eq:convbound}
|\mathcal{G} h |_w \leq |h|_w,
\qquad
\|\mathcal{G} h \|_s \leq \|h\|_s,
\qquad
\|\mathcal{G} h \|_u \leq \|h\|_u.
\end{equation}
Assume furthermore that $M$ is a positive definite covariance matrix, and write 
$$
\mathcal{G}_\eps h(x) = \int_{\mathbb{T}^d}\theta[z; \eps M](x) h(z)\dd z.
$$
Then
\begin{equation}\label{eq:gaussbound}
|\mathcal{G}_\eps h-h|_w
\leq C_{\alpha,M}\eps^{\beta/2}\|h\|.
\end{equation}
\end{lemma}

\begin{lemma}
\label{lemma:L-infinity regularisation}
Let $h \in X$. Then there exists a constant $C>0$ such that, for all $\tau >0$,
\begin{equation}\label{eq:heatsemiestimate}
\|\e^{\tau \Delta} h\|_{L^\infty(\TT^2)}
\leq
C\max\{1,\tau^{-\frac{1-\sigma}{2}}\} \|h\|_s,
\qquad
\|\e^{\tau \Delta} h\|_{L^\infty(\TT^2)}
\leq
C\max\{1,\tau^{-\frac{1}{2}}\} |h|_w.    
\end{equation}
\end{lemma}

\begin{remark}
In view of Proposition \ref{proposition:main lasota yorke} and the above Lemma \ref{lemma:gaussian-bound}, it follows that the following uniform-in-$\eps$ Lasota-Yorke inequality holds:
\begin{equation}
\|\e^{2 \pi i g} \e^{\eps \Delta}\mathcal{L}_\alpha h \| \leq C\alpha^{-\eta} \| h \| + C |h|_w, \qquad |\e^{2 \pi i g}  \e^{\eps \Delta}\mathcal{L}_\alpha h |_w \leq C |h|_w, \qquad \forall h \in X.    
\end{equation}  
Thanks to Theorem \ref{thm:main eigenvalue} and Proposition \ref{prop:liverani-keller}, Theorem \ref{thm:fastpulsed} can then be deduced exactly as Theorem \ref{thm:fastdynamo}, see Section \ref{sub:proof}.
\end{remark}

The estimates in Lemmas \ref{lemma:gaussian-bound}-\ref{lemma:L-infinity regularisation}  will be used repeatedly throughout the section. In particular, Lemma \ref{lemma:L-infinity regularisation} will often be applied with $\tau=N\eps$ in order to quantify the regularising effect of diffusion during the long time interval on which $u_{\alpha,N}\equiv 0$.

\begin{proof}[Proof of Lemma \ref{lemma:gaussian-bound}]
We first prove the weak norm estimate. Let $W\in \Sigma$ and let $\varphi \in C^1(W)$ satisfy
\[
|\varphi|_{C^1(W)}\leq 1.
\]
Using Fubini's theorem and translation invariance on $\TT^2$, we compute
\begin{align*}
    \int_W \mathcal G h (x)\,\varphi (x)\,\dd \HH^1(x)
    &=
    \int_W \int_{\TT^2}
    \theta(x-y) h(y)\,\dd y\,\varphi(x)\,\dd \HH^1(x)
    \\
    &=
    \int_{\TT^2} \theta(\bar y )
    \int_W
    h(x-\bar y )
    \varphi(x)
    \dd \HH^1(x)\,\dd \bar y
    \\
    &=
    \int_{\TT^2} \theta(\bar y)
    \int_{W-\bar y}
    h(\bar x)
    \varphi(\bar x+\bar y)
    \dd \HH^1(\bar x)\,\dd \bar y .
\end{align*}
Since translations preserve admissible curves and the $C^1$ norm of the test function, we obtain
\[
\left|
\int_{W-\bar y}
h(\bar x)\,
\varphi(\bar x+\bar y)\,
\dd \HH^1(\bar x)
\right|
\leq |h|_w .
\]
Using that $\|\theta\|_{L^1(\TT^2)}\leq 1$, we conclude that the first bound in \eqref{eq:convbound} holds.

The strong stable norm estimate in \eqref{eq:convbound} follows from the same argument, since the stable norm is also defined through testing against translated curves.

We next prove the unstable norm estimate, i.e. the third estimate in \eqref{eq:convbound}. Let
\[
\dd_\Sigma(W_1,W_2)<\delta,
\qquad
\dd_q(\varphi_1,\varphi_2)<\delta,
\]
with $|\varphi_i|_{C^1(W_i)}\leq 1$. Proceeding as above,
\begin{align*}
    &
    \int_{W_1} \mathcal G h(x)\varphi_1(x)\,\dd \HH^1(x)
    -
    \int_{W_2} \mathcal G h(x)\varphi_2(x)\,\dd \HH^1(x)
    \\
    &\qquad =
    \int_{\TT^2}\theta(\bar y)
    \Bigg[
    \int_{W_1-\bar y}
    h(\bar x)\varphi_1(\bar x+\bar y)\,
    \dd \HH^1(\bar x)
    -
    \int_{W_2-\bar y}
    h(\bar x)\varphi_2(\bar x+\bar y)\,
    \dd \HH^1(\bar x)
    \Bigg]
    \dd \bar y .
\end{align*}
Since translations preserve the distances defining the unstable norm,
\[
\dd_\Sigma(W_1-\bar y,W_2-\bar y)<\delta,
\qquad
\dd_q\bigl(\varphi_1(\cdot+\bar y),\varphi_2(\cdot+\bar y)\bigr)<\delta .
\]
Hence,
\[
\left|
\int_{W_1-\bar y}
h(\bar x)\varphi_1(\bar x+\bar y)\,\dd \HH^1(\bar x)
-
\int_{W_2-\bar y}
h(\bar x)\varphi_2(\bar x+\bar y)\,\dd \HH^1(\bar x)
\right|
\leq
\delta^\beta \|h\|_u .
\]
Integrating against $\theta$ yields the third bound in \eqref{eq:convbound}.

Finally, we prove the weak--strong convergence estimate for Gaussian kernels. Let $W\in\Sigma$ and let $\varphi\in C^1(W)$ satisfy $|\varphi|_{C^1(W)}\leq 1$. We compute
\begin{align}
\label{eq:weak-strong-convergence}
    \int_W (\mathcal G_\eps h-h)(x)\varphi(x)\,\dd \HH^1(x)
    &=
    \int_W
    \int_{\TT^2}
    \theta[0; \eps M](z)\bigl(h(x-z)-h(x)\bigr)\,\dd z\,
    \varphi(x)\,\dd \HH^1(x)
    \notag
    \\
    &=
    \int_{\TT^2}\theta[0; \eps M](z)
    \Bigg[
    \int_{W-z}
    h(x)\varphi(x+z)\,\dd \HH^1(x)
    \notag
    \\
    &\hspace{3cm}
    -
    \int_W
    h(x)\varphi(x)\,\dd \HH^1(x)
    \Bigg]
    \dd z .
\end{align}
We split the integral into the regions $\dd_{\TT^2}(z,0)<1/\alpha$ and $\dd_{\TT^2}(z,0)\geq 1/\alpha$. For $\dd_{\TT^2}(z,0)<1/\alpha$, the definition of the unstable norm gives
\[
\left|
\int_{W-z}
h(x)\varphi(x+z)\,\dd \HH^1(x)
-
\int_W
h(x)\varphi(x)\,\dd \HH^1(x)
\right|
\leq
C\dd^\beta_{\TT^2}(z,0) \|h\|_u .
\]
For $\dd_{\TT^2}(z,0)\geq 1/\alpha$, we use the weak norm bound,
\[
\left|
\int_{W-z}
h(x)\varphi(x+z)\,\dd \HH^1(x)
-
\int_W
h(x)\varphi(x)\,\dd \HH^1(x)
\right|
\leq
C |h|_w .
\]
Therefore,
\begin{align*}
\left|
\int_W (\mathcal G_\eps h-h)(x)\varphi(x)\,\dd \HH^1(x)
\right|
&\leq
C\|h\|_u
\int_{\dd_{\TT^2}(z,0)<1/\alpha}
\dd_{\TT^2}(z,0)^\beta \theta[0; \eps M](z)\,\dd z
\\
&\qquad
+
C|h|_w
\int_{\dd_{\TT^2}(z,0)\geq 1/\alpha}
\theta[0; \eps M](z)\,\dd z .
\end{align*}
Since $\theta_\eps$ is a Gaussian with covariance matrix $\eps M$,
\[
\int_{\RR^2}\dd_{\TT^2}^\beta(z,0) \theta[0; \eps M](z)\,\dd z
\leq C_M \eps^{\beta/2},
\]
while the Gaussian tail satisfies 
\[
\int_{\dd_{\TT^2}(z,0)\geq 1/\alpha}\theta[0; \eps M](z)\,\dd z
\leq C_{\alpha,M}\eps^{\beta/2}.
\]
Combining the estimates yields \eqref{eq:gaussbound},
which concludes the proof.
\end{proof}

\begin{proof}[Proof of Lemma \ref{lemma:L-infinity regularisation}]
The proofs of the two inequalities are identical; the weak norm bound corresponds to the case $\sigma=0$. Hence, we only prove the estimate involving the strong stable norm.

We decompose
\[
h(x,y)=h_{\neq}(x,y)+\overline h(x),
\qquad
\overline h(x)=\int_{\TT} h(x,y)\,\dd y,
\qquad
h_{\neq}(x,y)=h(x,y)-\overline h(x).
\]
First, we observe that
\[
|\overline h(x)|
=
\left|
\int_0^1 h(x,y)\,\dd y
\right|
\leq
\|h\|_s,
\]
since the vertical segment $\{x\}\times[0,1]$ is an admissible curve in $\Sigma$. Since the heat semigroup is contractive on $L^\infty$, it follows that
\[
\|\e^{\tau \Delta}\overline h\|_{L^\infty(\TT^2)}
\leq
\|\overline h\|_{L^\infty(\TT)}
\leq
\|h\|_s .
\]
It therefore remains to estimate the contribution of $h_{\neq}$. By construction, for every fixed $x$,
\[
\int_{\TT} h_{\neq}(x,y)\,\dd y=0.
\]
Hence, we may write
\[
h_{\neq}(x,y)=\partial_y H(x,y),
\qquad
H(x,y)=\int_0^y h_{\neq}(x,s)\,\dd s .
\]
Since the Laplacian separates variables,
$\e^{\tau \Delta}
=
\e^{\tau \partial_x^2}\e^{\tau \partial_y^2}$,
and therefore
\[
\e^{\tau \partial_y^2} h_{\neq}
=
\partial_y \e^{\tau \partial_y^2} H .
\]
Using standard heat semigroup estimates in Hölder spaces,
\[
\|\partial_y \e^{\tau \partial_y^2}H(x,\cdot)\|_{L^\infty(\TT)}
\leq
C\max\{1,\tau^{-\frac{1-\sigma}{2}}\}
\|H(x,\cdot)\|_{C^\sigma(\TT)} .
\]
For completeness, we briefly justify this estimate. Let $\theta[z;\tau](y)$ denote the periodic one-dimensional heat kernel. Since
\[
\int_{\TT}\partial_y\theta[0;\tau](z)\,\dd z=0,
\]
we may write
\begin{align*}
\partial_y \e^{\tau \partial_y^2}H(x,y)
&=
\int_{\TT}\partial_y\theta[0;\tau](z)
\bigl(H(x,y-z)-H(x,y)\bigr)\,\dd z .
\end{align*}
Hence,
\begin{align*}
|\partial_y \e^{\tau \partial_y^2}H(x,y)|
&\leq
C\|H(x,\cdot)\|_{C^\sigma(\TT)}
\int_{\RR}
\frac{|z|^\sigma}{\tau}
\e^{-c|z|^2/\tau}\,\dd z\leq
C\tau^{-\frac{1-\sigma}{2}}
\|H(x,\cdot)\|_{C^\sigma(\TT)} .
\end{align*}
We now estimate the Hölder norm of $H$. For any $y_1,y_2\in\TT$,
\begin{align*}
|H(x,y_1)-H(x,y_2)|
=
\left|
\int_{y_2}^{y_1} \bigl(h(x,s)-\overline h(x)\bigr)\,\dd s
\right|\leq
C|y_1-y_2|^\sigma \|h\|_s .
\end{align*}
Therefore,
\[
\|H(x,\cdot)\|_{C^\sigma(\TT)}
\leq
C\|h\|_s .
\]
Combining the previous estimates yields
\[
\|\e^{\tau \partial_y^2} h_{\neq}\|_{L^\infty(\TT^2)}
\leq
C\max\{1,\tau^{-\frac{1-\sigma}{2}}\}\|h\|_s .
\]
Finally, since $\e^{\tau \partial_x^2}$ is contractive on $L^\infty$,
\[
\|\e^{\tau \Delta} h_{\neq}\|_{L^\infty(\TT^2)}
\leq
C\max\{1,\tau^{-\frac{1-\sigma}{2}}\}\|h\|_s .
\]
Combining this with the estimate for $\overline h$ concludes the proof of \eqref{eq:heatsemiestimate}.
\end{proof}

\subsection{Construction of the model problem.}
\label{sec:model problem}

In this section, we define the model operator $\mathcal{Q}_{\alpha,\eps}$ appearing in Proposition~\ref{prop:error decomposition}. We refer to Subsection~\ref{subsec:model} for a rigorous derivation of this model problem from the kinematic dynamo equations.

The operator $\mathcal{Q}_{\alpha,\eps}$ is obtained by patching together four local operators, each corresponding to the evolution of \eqref{passive-vector} under one of the four affine branches of the planar component of the velocity field $u_{\alpha,N}^{(1,2)}$. More precisely, for $y \in \mathcal M_\ell$, the flow map associated with $u_{\alpha,N}^{(1,2)}$ is generated by alternating affine shear flows. After extending these affine shears to $\RR^2$, equation~\eqref{passive-vector} may be solved explicitly, yielding Gaussian kernels with anisotropic covariance. 

Motivated by this explicit representation, we define the $\ell$-th model operator by
\begin{align}
\label{d:model-ell-operator}
    \mathcal{Q}_{\alpha,\eps}^\ell h(x)
    =
    \mathbbm{1}_{\XX_\alpha(\mathcal M_\ell)}(x)
    \int_{\TT^2}
    Q_{\alpha,\eps}^\ell(x,y)\,h(y)\,\dd y ,\qquad \ell=1,\ldots, 4.
\end{align}
Here the Green's matrix
\[
Q_{\alpha,\eps}^\ell:\TT^2\times\TT^2\to\RR^{2\times2}
\]
is given by 
\begin{equation}
\label{eq:model kernel gaussian}
Q_{\alpha,\eps}^\ell(x,y)
=
\alpha^{-2}A_\ell\,
\theta\bigl[
y;
\eps( A_\ell^{-1}\Gamma_\alpha^\ell A_\ell^{-T}+2A_\ell^{-1} A_\ell^{-T})
\bigr]
\bigl((T_\alpha^\ell)^{-1}(x)\bigr), \qquad  A_\ell^{-T}= (A_\ell^{-1})^T,
\end{equation}
where $\Gamma_\alpha^\ell\in\RR^{2\times2}$ is a positive definite covariance matrix and $A_\ell\in\RR^{2\times2}$ is one of the matrices introduced in~\eqref{eq:matrixMi}. In particular, note that on $T_{\alpha}(\cM_\ell)$, $(T_\alpha^\ell)^{-1}$ agrees with $T_\alpha^{-1}$, so we may equivalently write \eqref{eq:model kernel gaussian} with $T_\alpha^{-1}$ instead of $(T_\alpha^\ell)^{-1}$.

In view of Section~\ref{sec: Gaussian bounds}, this representation is particularly convenient. Indeed, the operator consists of Gaussian regularisation of the initial datum $h$, followed by composition with the inverse flow map $T_\alpha^{-1}$. By Lemma~\ref{lemma:gaussian-bound}, convolution with Gaussians behaves well on the anisotropic spaces introduced in Section~\ref{sec:anisotropic}.

We then define the full model operator by
\begin{align}
\label{eq:model-problem}
\mathcal Q_{\alpha,\eps} h(x)
=
\int_{\TT^2}
Q_{\alpha,\eps}(x,y)\,h(y)\,\dd y
=
\sum_{\ell=1}^4
\mathcal Q_{\alpha,\eps}^\ell h(x).
\end{align}
Since the model operator is written explicitly as an integral kernel, it is natural to seek a corresponding representation for the true operator $\cL_{\alpha,\eps}$. The error estimates in Proposition~\ref{prop:error decomposition} will then follow from pointwise bounds comparing the true Green's matrix with its model approximation.

Accordingly, we introduce the Green's matrix
\[
\alpha^2K_{\alpha,\eps}:\TT^2\times\TT^2\to\RR^{2\times2}
\]
associated with the first two components of the solution to \eqref{passive-vector}. More precisely, if $h:\TT^2\to\RR^2$ is prescribed at time $t=N$, then the solution at time $t=N+2$ satisfies
\[
\alpha^2\cL_{\alpha,\eps}h(x)
=
\int_{\TT^2}
\alpha^2K_{\alpha,\eps}(x,y)\,h(y)\,\dd y.
\]
Finally, for convenience of notation, we define the kernel error
\[
E_{\alpha,\eps}(x,y)
=
K_{\alpha,\eps}(x,y)-Q_{\alpha,\eps}(x,y),
\]
so that
\[
\mathcal E_{\alpha,\eps}h(x)
=
\cL_{\alpha,\eps}h(x)-\mathcal Q_{\alpha,\eps}h(x)
=
\int_{\TT^2}
E_{\alpha,\eps}(x,y)\,h(y)\,\dd y .
\]
With this formulation in hand, the proof of Proposition~\ref{prop:error decomposition} reduces to suitable pointwise estimates on the difference between the true and model kernels.
\begin{lemma}\label{lemma:kernel bound}

There exist constants $C_\alpha,c_\alpha>0$ such that the following holds for all $\eps \in (0,1]$. For every multi-index $\mathbf m$ with $|\mathbf m|\leq1$, and for all $x,y\in\TT^2$, it holds
\begin{align}\label{eq:kernel-global}
|D_x^{\mathbf m}K_{\alpha,\eps}(x,y)|\leq\frac{C_\alpha}{\eps^{1+\frac{|\mathbf m|}{2}}}\exp\Bigl(-c_\alpha \eps^{-1}\dd^2_{\TT^2}(x,T_\alpha y)\Bigr).
\end{align}
Moreover, the following holds for $\mathbf m=0$, and all $x,y\in\TT^2$, and for $|\mathbf m|=1$ with $x\notin S^\star$, $y\in\TT^2$
\begin{align}\label{eq:error-kernel-bound}
|D_x^{\mathbf m}E_{\alpha,\eps}(x,y)|\leq\sum_{\ell=1}^4\frac{C_\alpha}{\eps^{1+\frac{|\mathbf m|}{2}}}\exp\Bigl(-c_\alpha \eps^{-1}\dd^2_{\TT^2}(x,S^\star)\Bigr)\exp\Bigl(-c_\alpha \eps^{-1}\dd^2_{\TT^2}(x,T^\ell_\alpha y)\Bigr),
\end{align}
where $S^\star$ denotes the forward singular set introduced in~\eqref{d:forward-singular}, and $T_\alpha^\ell$ is the extension of the map $T_\alpha|_{\cM_\ell}$ to $\TT^2$, defined in \eqref{d:T-alpha-ell}.
\end{lemma}
The proof of Lemma~\ref{lemma:kernel bound}, while lengthy, follows from fairly standard Gaussian approximation arguments and is therefore postponed to Section~\ref{sec:kernel}. The remainder of this section is devoted to proving Proposition~\ref{prop:error decomposition}, assuming Lemma~\ref{lemma:kernel bound}, and subsequently deducing Proposition~\ref{prop: uniform Lasota Yorke}.

\subsection{Uniform Lasota-Yorke and weak-strong convergence for the model problem} \label{subsec:model-lasota}

In this section, we prove the uniform Lasota-Yorke inequality and the weak-strong convergence estimate for the model operator, which constitute the first part of Proposition \ref{prop:error decomposition}. More precisely, we prove the following. 

\begin{proposition} \label{prop:lasota-model}
Let $\mathcal{Q}_{\alpha,\eps}$ be defined as in \eqref{eq:model-problem}. Then, there exists a constant $C>0$ such that for any $\alpha \geq 1$ and for any $\eps > 0$ it holds
\begin{align} \label{eq:lasota-for-model}
    \|\cQ_{\alpha, \eps}  h\| \leq C \alpha^{-\eta}\|h\|+C |h|_w \,, \qquad | \cQ_{\alpha, \eps}  h|_w \leq C |h|_w \,.
\end{align}
Furthermore, for every $\alpha \in 2\NN$, there exists a constant $C_{\alpha}>0$, depending  on $\alpha$, such that
\begin{align} \label{eq:weak-strong-model}
    |(\cQ_{\alpha, \eps}-\cL_\alpha)h |_{w} \leq  C_{\alpha} \eps^{\beta/2}\|h\|.
\end{align}
\end{proposition}

\begin{proof}
We begin by proving \eqref{eq:lasota-for-model}. The weak-to-weak bound is analogous, and so we do not prove it explicitly. Furthermore, by linearity, it suffices to prove the result for each $Q_{\alpha,\eps}^\ell$. Note that $Q_{\alpha,\eps}^\ell $ may be written as 
$$
Q_{\alpha,\eps}^\ell=\cL_\alpha^\ell \circ \mathcal{G}_\eps^\ell,
$$
where $\cL_{\alpha}^\ell h(x)=\alpha^{-2}A_\ell\mathds{1}_{T_\alpha(\cM_\ell)} h\circ T_\alpha^{-1}(x)$, and
$$
\mathcal{G}_\eps^\ell h(x)= \int_{\TT^2} \theta[y;\eps (A_\ell^{-1} \Gamma_\alpha^\ell A_\ell^{-T}+2A_\ell^{-1}A_\ell^{-T})](x)h(y) \dd y.
$$
By Corollary~\ref{lemma:improved-lasota}, we thus deduce that 
$$
\|Q_{\alpha,\eps}h\|\leq C\alpha^{-\eta}\|\mathcal{G}_{\eps}^\ell h\|+C|\mathcal{G}_{\eps}^\ell h|_w.
$$
Since the operator $\cG_\eps^\ell $ acts simply via a convolution with the Gaussian kernel $ \theta [y, \eps ( A_\ell^{-1 } \Gamma_\alpha^\ell A_\ell^{-T}+2A_\ell^{-1} A_\ell^{-T})] (x)$, applying Lemma \ref{lemma:gaussian-bound}, the claimed Lasota-Yorke estimate follows. Hence, it remains to deduce the weak-strong convergence. 
We estimate
\begin{equation} 
    |\cQ_{\alpha, \eps}(h) - \cL_\alpha (h)|_w  \leq  \sum_{\ell=1}^4 | \cQ_{\alpha,\eps}^\ell (h) - \cL_\alpha (h)\mathbbm{1}_{T_\alpha (\cM_\ell)} |_w  =\sum_{\ell=1}^4 |\cL_{\alpha}^\ell (\mathcal G_{\eps}^\ell (h) - h) |_w \leq C \sum_{\ell=1}^4| \mathcal G_{\eps}^\ell (h) - h |_w \,.
\end{equation}
where in the last line we have used Lemma \ref{lemma:improved-lasota}. From here, it once again suffices to appeal to Lemma \ref{lemma:gaussian-bound} to conclude the proof of weak-strong convergence.
\end{proof}

\subsection{The error terms}
\label{subsec:error}

In this subsection, we estimate the error operator
\[
\mathcal E_{\alpha,\eps}\circ \e^{N\eps\Delta}h
=
\cL_{\alpha,\eps}\circ \e^{N\eps\Delta}h
-
\mathcal Q_{\alpha,\eps}\circ \e^{N\eps\Delta}h .
\]
More precisely, we prove the following result.

\begin{proposition}
\label{prop:error bounds}
There exists a constant $C_\alpha>0$ such that, for all $N\eps<1$ and all $h\in X$, the following estimates hold:
\begin{align}
\label{eq:error-strong}
\|\mathcal E_{\alpha,\eps}\circ \e^{N\eps\Delta}h\|
&\leq
C_\alpha N^{-\frac{1-\sigma}{2}}\|h\|,
\\
\label{eq:error-weak}
|\mathcal E_{\alpha,\eps}\circ \e^{N\eps\Delta}h|_w
&\leq
C_\alpha N^{-1/2}|h|_w,
\\
\label{eq:error-weak-strong}
|\mathcal E_{\alpha,\eps}\circ \e^{N\eps\Delta}h|_w
&\leq
C_\alpha \eps^{\sigma/2}\|h\|.
\end{align}
\end{proposition}

Once Proposition~\ref{prop:error bounds} is established, Proposition~\ref{prop:error decomposition} follows immediately by combining these estimates with the bounds for the model operator obtained in Proposition~\ref{prop:lasota-model}.

\subsubsection{The strong stable norm}
We begin by estimating the strong stable norm.
\begin{lemma}
\label{lemma:sublemma 1}
There exists a constant $C_\alpha>0$ such that for all $N\eps<1$,
\[
\|\mathcal{E}_{\alpha,\eps}\circ \e^{N\eps \Delta } h\|_s
\leq
C_{\alpha}N^{-\frac{1-\sigma}{2}}\|h\| \,.
\]
\end{lemma}
\begin{proof}
Recall that the action of $\mathcal{E}_{\alpha,\eps}$ is given by integration against the kernel $E_{\alpha,\eps}$. Let therefore $W\in\Sigma$ and $\varphi\in C^q(W)$ satisfy
$\|\varphi\|_{C^q(W)}\le |W|^{-\sigma}$.
Then
\begin{align}
\int_W \mathcal{E}_{\alpha,\eps}f(x)\varphi(x)\,\dd\HH^1(x)
=
\int_W \int_{\TT^2}
E_{\alpha,\eps}(x,y)f(y)\,\dd y\,\varphi(x)\,\dd\HH^1(x).
\label{eq:operator-main-corrected}
\end{align}
Applying this identity with
$f=\e^{N\eps\Delta}h$,
and using the pointwise kernel bound from Lemma~\ref{lemma:kernel bound}, we obtain
\begin{align*}
&
\left|
\int_W \mathcal{E}_{\alpha,\eps}\e^{N\eps\Delta}h(x)
\varphi(x)\,\dd\HH^1(x)
\right|
\nonumber
\\
&\qquad\le
\|\e^{N\eps\Delta}h\|_{L^\infty(\TT^2)}
\|\varphi\|_{C^0(W)}
\int_W\int_{\TT^2}
|E_{\alpha,\eps}(x,y)|
\,\dd y\,\dd\HH^1(x)
\\
&\qquad\le
\sum_{\ell=1}^4C_\alpha
\|\e^{N\eps\Delta}h\|_{L^\infty(\TT^2)}
\|\varphi\|_{C^0(W)}
\int_W\int_{\TT^2}
\frac1\eps
\e^{-c_\alpha\eps^{-1}{\dd^2_{\TT^2}(x,S^\star)}}
\e^{-c_\alpha\eps^{-1}\dd^2_{\TT^2}(x, T^\ell_\alpha y)}
\,\dd y\,\dd\HH^1(x).
\label{eq:error-bound-1}
\end{align*}
Integrating first in $y$ yields
\[
\sum_{\ell=1}^4 \int_{\TT^2}
\frac1\eps
\e^{-c_\alpha\eps^{-1}\dd^2_{\TT^2}(x, T^\ell_\alpha y)}
\,\dd y
\le C_\alpha,
\]
and therefore
\[
\left|
\int_W \mathcal{E}_{\alpha,\eps}\e^{N\eps\Delta}h(x)
\varphi(x)\,\dd\HH^1(x)
\right|
\le
C_\alpha
\|\e^{N\eps\Delta}h\|_{L^\infty(\TT^2)}
\|\varphi\|_{C^0(W)}
\int_W
\e^{-c_\alpha\eps^{-1}\dd^2_{\TT^2}(x,S^\star)}
\,\dd\HH^1(x).
\]
Since every admissible curve $W\in\Sigma$ is uniformly transverse to the singular set $S^\star$, there exists $C>0$ such that 
\[
\int_W
\e^{-c_\alpha\eps^{-1}\dd_{\TT^2}^2(x,S^\star)}
\,\dd\HH^1(x)
\le
C\min\{|W|,\eps^{1/2}\}.
\]
Hence,
\begin{equation}\label{eq:estimate-model-real}
\left|
\int_W \mathcal{E}_{\alpha,\eps}\e^{N\eps\Delta}h(x)
\varphi(x)\,\dd\HH^1(x)
\right|
\le
C_\alpha
\|\e^{N\eps\Delta}h\|_{L^\infty(\TT^2)}
\min\{|W|,\eps^{1/2}\}
\|\varphi\|_{C^0(W)}.
\end{equation}
Using now Lemma~\ref{lemma:L-infinity regularisation}, we obtain
\[
\|\e^{N\eps\Delta}h\|_{L^\infty(\TT^2)}
\le
C(N\eps)^{-\frac{1-\sigma}{2}}\|h\|_s,
\]
since $N\eps<1$. Together with $\|\varphi\|_{C^0(W)}
\le |W|^{-\sigma}$,
this gives
\begin{align*}
\left|
\int_W \mathcal{E}_{\alpha,\eps}\e^{N\eps\Delta}h(x)
\varphi(x)\,\dd\HH^1(x)
\right|\le
C_\alpha
(N\eps)^{-\frac{1-\sigma}{2}}
\|h\|_s
\min\{|W|,\eps^{1/2}\}
|W|^{-\sigma}.
\end{align*}
Finally, since
\[
\min\{|W|,\eps^{1/2}\}|W|^{-\sigma}
\le
C\eps^{\frac{1-\sigma}{2}},
\]
we conclude that
\[
\left|
\int_W \mathcal{E}_{\alpha,\eps}\e^{N\eps\Delta}h(x)
\varphi(x)\,\dd\HH^1(x)
\right|
\le
C_\alpha
N^{-\frac{1-\sigma}{2}}
\|h\|_s.
\]
Taking the supremum over admissible $W$ and test functions $\varphi$ yields the result.
\end{proof}

\subsubsection{The strong unstable norm}
To complete the proof of \eqref{eq:error-strong}, we need a bound on the strong unstable norm of the error term.
\begin{lemma}
\label{lemma:sublemma 2}
There exists a constant $C_\alpha>0$ such that for all $N\eps<1$,
\[
\|\mathcal{E}_{\alpha,\eps}\circ \e^{N\eps \Delta}h \|_{u}
\leq
C_\alpha N^{-\frac{1-\sigma}{2}}\|h\| \, .
\]
\end{lemma}

\begin{proof}
Fix curves $W_1,W_2\in\Sigma$ satisfying
\[
\dd_\Sigma(W_1,W_2)\le \delta,
\]
and test functions
\[
\varphi_i\in C^1(W_i),
\qquad
\|\varphi_i\|_{C^1(W_i)}\le 1,
\qquad
\dd_q(\varphi_1,\varphi_2)\le \delta,
\]
for $i=1,2$. We distinguish two cases.

\smallskip

\noindent\textbf{Case 1: $\delta\ge \eps^{1/2}$.} 
Using the kernel bound from Lemma~\ref{lemma:kernel bound}, we estimate for each $i=1,2$
\begin{align*}
&
\left|
\int_{W_i}
\mathcal E_{\alpha,\eps}\e^{N\eps\Delta}h(x)
\varphi_i(x)\,\dd\HH^1(x)
\right|
\\
&\qquad\le
\sum_{\ell=1}^4\|\e^{N\eps\Delta}h\|_{L^\infty}
\|\varphi_i\|_{C^0}
\int_{W_i}\int_{\TT^2}
\frac{C_\alpha}{\eps}
\e^{-c_\alpha\eps^{-1}\dd_{\TT^2}^2(x,T^\ell_\alpha y)}
\e^{-c_\alpha\eps^{-1}\dd^2_{\TT^2}(x,S^\star)}
\,\dd y\,\dd\HH^1(x)\\
&\qquad
\le
C_\alpha
\|\e^{N\eps\Delta}h\|_{L^\infty}
\int_{W_i}
\e^{-c_\alpha\eps^{-1}\dd_{\TT^2}^2(x,S^\star)}
\,\dd\HH^1(x)\\
&\qquad
\le
C_\alpha
\eps^{1/2}
\|\e^{N\eps\Delta}h\|_{L^\infty}.
\end{align*}
Finally, since $\delta\ge \eps^{1/2}$, we have that
$\eps^{\sigma/2}
\le
\delta^\beta \eps^{\frac{\sigma-\beta}{2}}$,
and because $\sigma>\beta$, we conclude
\[
\left|
\int_{W_i}
\mathcal E_{\alpha,\eps}\e^{N\eps\Delta}h
\varphi_i
\right|
\le
C_\alpha
N^{-\frac{1-\sigma}{2}}
\delta^\beta
\|h\|_s.
\]

\smallskip

\noindent\textbf{Case 2: $\delta<\eps^{1/2}$.}
We now apply Lemma~\ref{lemma:technical} and decompose the curves into matched and unmatched curves as
\[
W_i
=
\bigcup_{j=1}^{M_\alpha}W_{i,j}
\cup
\bigcup_{j=1}^{8}U_{i,j}.
\]
The unmatched pieces are estimated exactly as in the strong stable norm bound:
\[
C\delta
\|\e^{N\eps\Delta}h\|_{L^\infty}
\le
C
N^{-\frac{1-\sigma}{2}}
\delta^\beta
\|h\|_s,
\]
since $\delta<\eps^{1/2}$.

Thus it suffices to estimate the matched pieces. Up to a multiplicative constant depending on $M_\alpha$, it is enough to bound 
\begin{align}
\label{eq:matched-main}
\Bigg|
\int_{\TT^2} \e^{N\eps \Delta}h(y)
\Bigg(
\int_{W_{1,j}}
E_{\alpha,\eps}(x,y)\varphi_1(x)\,\dd\HH^1(x)
-
\int_{W_{2,j}}
E_{\alpha,\eps}(x,y)\varphi_2(x)\,\dd\HH^1(x)
\Bigg)
\,\dd y
\Bigg|,
\end{align}
where $W_{1,j}, W_{2,j}$ are matched curves, and in particular lie in the same $T_\alpha(\cM_\ell)$.  We may assume
$|W_{1,j}|\le |W_{2,j}|$.
Let
\[
\gamma_i:[0,S_i]\to W_{i,j}
\]
be arc-length parametrisations. Define the translated test function
\[
\widetilde\varphi_2(\gamma_1(t))
=
\varphi_2(\gamma_2(t)),
\qquad
0\le t\le S_1.
\]
Then 
\[
\dd_q(\varphi_2,\widetilde\varphi_2)=0.
\]
Adding and subtracting the integral with $\widetilde\varphi_2$, one contribution becomes
\[
I_1:=\left|
\int_{\TT^2}
\e^{N\eps \Delta}h(y)
\int_{W_{1,j}}
E_{\alpha,\eps}(x,y)
(\varphi_1(x)-\widetilde\varphi_2(x))
\,\dd\HH^1(x)
\,\dd y
\right|.
\]
Estimating as before,
\[
I_1\le
C_\alpha
\|\e^{N\eps\Delta}h\|_{L^\infty}
\delta
\int_{W_{1,j}}
\e^{-c_\alpha\eps^{-1}\dd_{\TT^2}^2(x,S^\star)}
\,\dd\HH^1(x)
\le
C_\alpha
\delta\eps^{1/2}
\|\e^{N\eps\Delta}h\|_{L^\infty}
\le
C_\alpha
N^{-\frac{1-\sigma}{2}}
\delta^\beta
\|h\|_s,
\]
since $\delta<\eps^{1/2}$.

We now estimate the remaining term
\begin{align*}
I_2:=\int_{\TT^2} \e^{N\eps \Delta}h(y)
\left(
\int_{W_{1,j}}
E_{\alpha,\eps}(x,y)\widetilde\varphi_2(x)\,\dd\HH^1(x)-
\int_{W_{2,j}}
E_{\alpha,\eps}(x,y)\varphi_2(x)\,\dd\HH^1(x)
\right)
\,\dd y.
\end{align*}
After trimming the longer curve, we may assume
$|W_{1,j}|=|W_{2,j}|$.
The discarded piece is estimated as above.
Writing
\[
\gamma_i(t)=x_i+t v_i,
\]
the difference becomes
\begin{equation}
\label{eq:difference kernels equation corrected}
\int_0^{S_1}
\varphi_2(\gamma_2(t))
\bigl(
E_{\alpha,\eps}(\gamma_1(t),y)
-
E_{\alpha,\eps}(\gamma_2(t),y)
\bigr)
\,\dd t.
\end{equation}
Using the fundamental theorem of calculus,
\[
E_{\alpha,\eps}(\bar x,y)-E_{\alpha,\eps}(x,y)
=
\int_0^{|\bar x-x|}
\nabla_x E_{\alpha,\eps}
\Bigl(
\bar x+s\frac{x-\bar x}{|x-\bar x|},
y
\Bigr)
\cdot
\frac{x-\bar x}{|x-\bar x|}
\,\dd s.
\]
Noting that $W_{1,j},W_{2,j} $ both live in the same connected component of the projection of $T_{\alpha}(\cM_\ell)$ onto the unit square, which is in particular a convex set, we may apply Lemma~\ref{lemma:kernel bound} to bound,
\begin{align*}
|E_{\alpha,\eps}(\gamma_1(t),y)-E_{\alpha,\eps}(\gamma_2(t),y)|
\le \sum_{\ell=1}^4 C_\alpha
\int_0^{|\gamma_1(t)-\gamma_2 (t)|}
\eps^{-3/2}
\e^{-c_\alpha\eps^{-1} \dd_{\TT^2}^2( z_s(t), T^\ell_\alpha y)}
\e^{-c_\alpha\eps^{-1}\dd_{\TT^2}^2(z_s(t),S^\star)}
\,\dd s,
\end{align*}
where
\[
z_s(t)
=
\bar \gamma_1(t)+s\frac{\gamma_2(t)-\gamma_1(t)}{|\gamma_1(t)-\gamma_2(t)|}.
\]
Integrating in $y$ removes one power of $\eps^{-1}$ and yields
\[
|E_{\alpha,\eps}(\gamma_1(t),y)-E_{\alpha,\eps}(\gamma_2(t),y)|\le
C_\alpha
\int_0^{|\bar \gamma_1(t)-\gamma_2(t)|}
\eps^{-1/2}
\e^{-c_\alpha\eps^{-1}\dd_{\TT^2}^2(z_s(t),S^\star)}
\,\dd s.
\]
Since the distance between the matched curves is at most $C\delta<C\eps^{1/2}$, Young's inequality implies
\[
\dd_{\TT^2}^2(z_s(t),S^\star)
\ge
\frac12\dd_{\TT^2}^2( \gamma_1(t),S^\star)-C\delta^2
\ge
\frac12\dd^2_{\TT^2}( \gamma_1(t),S^\star)-C\eps.
\]
Substituting into \eqref{eq:difference kernels equation corrected}, we obtain
\[
|I_2|\le
C_\alpha
\delta
\|\e^{N\eps\Delta}h\|_{L^\infty}
\int_0^{S_1}
\eps^{-1/2}
\e^{-c_\alpha\eps^{-1}\dd_{\TT^2}^2(\gamma_1(t),S^\star)}
\,\dd t.
\]
The transversality of admissible curves to $S^\star$ yields
\[
\int_0^{S_1}
\eps^{-1/2}
\e^{-c_\alpha\eps^{-1}\dd_{\TT^2}^2(\gamma_1(t),S^\star)}
\,\dd t
\le
C,
\]
and therefore
\[
|I_2|\le
C_\alpha
\delta
\|\e^{N\eps\Delta}h\|_{L^\infty}\le
C_\alpha
N^{-\frac{1-\sigma}{2}}
\delta^\beta
\|h\|_s,
\]
since $\delta<\eps^{1/2}$. Taking the supremum over admissible curves and test functions concludes the proof.

\end{proof}

\subsubsection{The weak-strong convergence and the weak bound}
We now prove a slightly stronger version of \eqref{eq:error-weak}-\eqref{eq:error-weak-strong}.
\begin{lemma}
\label{lemma:weak-strong error}
There exists a constant $C_\alpha>0$, depending only on $\alpha$, such that for all $N\eps<1$ and all $h\in X$,
\begin{equation}\label{eq:weakweakstrong}
|\mathcal{E}_{\alpha,\eps} \circ \e^{N\eps \Delta}h|_w
\leq
C_\alpha N^{- \frac{1-\sigma}{2}} \eps^{\frac{\sigma}{2}}\|h\|_s,
\qquad
|\mathcal{E}_{\alpha,\eps} \circ \e^{N\eps \Delta}h|_w
\leq
C_\alpha N^{-1/2}|h|_w.
\end{equation}
\end{lemma}
\begin{proof}
Let $W \in \Sigma$ and let $\varphi \in C^1(W)$ satisfy
$\|\varphi\|_{C^1(W)} \leq 1$.
Using the same computation as in Lemma~\ref{lemma:sublemma 1}, and in particular estimate~\eqref{eq:estimate-model-real}, we obtain
\[
\left|
\int_W
\int_{\TT^2}
E_{\alpha,\eps}(x,y)
\e^{N\eps\Delta}h(y)
\,\dd y\,
\varphi(x)
\,\dd\mathcal H^1(x)
\right|
\leq
C_\alpha\eps^{1/2}
\|\e^{N\eps\Delta}h\|_{L^\infty(\TT^2)}
\|\varphi\|_{C^0(W)}.
\]
Applying Lemma~\ref{lemma:L-infinity regularisation} yields
\[
\left|
\int_W
\int_{\TT^2}
E_{\alpha,\eps}(x,y)
\e^{N\eps\Delta}h(y)
\,\dd y\,
\varphi(x)
\,\dd \mathcal H^1(x)
\right|
\leq
C_\alpha
\eps^{1/2}
(N\eps)^{-\frac{1-\sigma}{2}}
\|h\|_s 
\leq 
C_\alpha N^{-\frac{1-\sigma}{2}}
\eps^{\sigma/2}\|h\|_s  .
\]
Taking the supremum over admissible curves and test functions gives the first estimate in \eqref{eq:weakweakstrong}.

The weak-to-weak estimate in \eqref{eq:weakweakstrong} follows from the same argument, using instead the weak norm estimate from Lemma~\ref{lemma:L-infinity regularisation}, and concludes the proof.
\end{proof}

\subsection{Proof of Proposition \ref{prop: uniform Lasota Yorke}} \label{subsec:proof-prop}
Having established Proposition~\ref{prop:error decomposition}, and hence Proposition~\ref{prop:lasota-continuous-diffusion}, we now prove Proposition~\ref{prop: uniform Lasota Yorke}.
We recall that the operator appearing in Proposition~\ref{prop: uniform Lasota Yorke} is
\[
\cL_{g, \eps}
\circ
\e^{- 8\eps \pi^2 }
\cL_{\alpha, \eps}
\circ
\e^{N  \eps (\Delta_{x,y} - 4\pi^2)} .
\]
While Proposition~\ref{prop:lasota-continuous-diffusion} treats the main operator
\[
\cL_{\alpha,\eps} \circ \e^{N \eps \Delta_{x,y}},
\]
it remains to control the phase-shift operator $\cL_{g,\eps}$. Fortunately, this operator arises from a well-behaved advection--diffusion equation, whose action on the anisotropic spaces can be controlled directly.

Indeed, let $b^\eps:[0,1]\times \TT^3\to\RR^2$ solve
\begin{equation}
\label{eq:passive-scalar-equation}
\partial_t b^\eps
-
g(x,y)\partial_z b^\eps
=
\eps \Delta b^\eps,
\end{equation}
with initial datum
\[
b_{\initial}(x,y,z)
=
\e^{2\pi i z} h(x,y).
\]
Writing
\[
b^\eps(t,x,y,z)
=
\e^{2\pi i z} h_t^\eps(x,y),
\]
we recall that
\[
\cL_{g,\eps}(h)=h_1^\eps .
\]
We then have the following result.
\begin{lemma}
\label{lemma:transport-diff-in-z}
There exists a constant $C_g>0$, depending only on $g$, such that
\begin{equation}
\label{eq:norm-strong-passive}
\|
\cL_{g,\eps}(h)
\|
\leq
C_g \|h\|,
\qquad
|
\cL_{g,\eps}(h)
|_w
\leq
C_g |h|_w ,
\end{equation}
and
\begin{equation}
\label{eq:norm-weak-strong-passive}
|
\cL_{g,\eps}(h)
-
\cL_{g,0}(h)
|_w
\leq
C_g \eps^{\beta/2}\|h\|.
\end{equation}
\end{lemma}

\begin{proof}
A direct computation shows that $h_t^\eps$ solves
\[
\partial_t h_t^\eps
=
\eps(\Delta_{x,y}-4\pi^2)h_t^\eps
+
2\pi i\, g h_t^\eps,
\]
with initial datum $h_0^\eps=h$. Applying Duhamel's formula gives
\[
h_t^\eps
=
\e^{t\eps(\Delta_{x,y}-4\pi^2)}h
+
2\pi i
\int_0^t
\e^{(t-s)\eps(\Delta_{x,y}-4\pi^2)}
(g h_s^\eps)
\,\dd s.
\]
Define
\[
H_t^\eps
=
\sup_{0\leq s\leq t}\|h_s^\eps\|.
\]
Using Lemma~\ref{lemma:smooth multiplier} together with Lemma~\ref{lemma:gaussian-bound}, we obtain
\[
H_t^\eps
\leq
\|h\|
+
2\pi \|g\|_{C^1}
\int_0^t
H_s^\eps
\,\dd s.
\]
An application of Grönwall's inequality yields
$H_t^\eps\leq C_g \|h\|$,
which proves the strong norm bound in
\eqref{eq:norm-strong-passive}. The weak norm estimate follows identically.

We next prove the weak--strong convergence estimate. Using Duhamel's formula for both $h_t^\eps$ and $h_t^0$, we obtain
\begin{align*}
|h_t^\eps-h_t^0|_w
&\leq
|
\e^{t\eps(\Delta_{x,y}-4\pi^2)}h-h
|_w
+
2\pi
\int_0^t
\Big|
\e^{(t-s)\eps(\Delta_{x,y}-4\pi^2)}
(g h_s^\eps-g h_s^0)
\Big|_w
\,\dd s
\\
&\quad
+
2\pi
\int_0^t
\Big|
\e^{(t-s)\eps(\Delta_{x,y}-4\pi^2)}
(g h_s^0)
-
g h_s^0
\Big|_w
\,\dd s.
\end{align*}
Using the weak--strong convergence estimate from
Lemma~\ref{lemma:gaussian-bound},
together with Lemma~\ref{lemma:smooth multiplier},
we deduce
\begin{align*}
|h_t^\eps-h_t^0|_w
\leq
C\eps^{\beta/2}\|h\|
+
2\pi \|g\|_{C^1}
\int_0^t
|h_s^\eps-h_s^0|_w
\,\dd s
+
C\eps^{\beta/2}\|g\|_{C^1}
\int_0^t
\|h_s^0\|
\,\dd s.
\end{align*}
Since \eqref{eq:norm-strong-passive} already implies
\[
\sup_{s\in[0,1]}
\|h_s^0\|
\leq
C_g\|h\|,
\]
we obtain
\[
|h_t^\eps-h_t^0|_w
\leq
C_g \eps^{\beta/2}\|h\|
+
2\pi \|g\|_{C^1}
\int_0^t
|h_s^\eps-h_s^0|_w
\,\dd s.
\]
Defining
\[
\widetilde H_t^\eps
=
\sup_{0\leq s\leq t}
|h_s^\eps-h_s^0|_w,
\]
Grönwall's inequality yields
\[
\widetilde H_t^\eps
\leq
C_g \eps^{\beta/2}\|h\|,
\]
which proves
\eqref{eq:norm-weak-strong-passive}.
\end{proof}

\subsubsection{Uniform Lasota--Yorke inequality and the weak bound}
Combining Proposition~\ref{prop:lasota-continuous-diffusion} with Lemma~\ref{lemma:transport-diff-in-z}, we deduce that there exists a constant $C_g>0$, depending only on $g$, such that
\[
\|
\cL_{g, \eps}
\circ
\e^{- 8\eps \pi^2 }
\cL_{\alpha, \eps}
\circ
\e^{N  \eps (\Delta_{x,y} - 4\pi^2)}
h
\|
\leq
C_g
\Bigl(
\alpha^{- \eta}
+
N^{- \frac{1-\sigma}{2}}
\Bigr)
\| h \|
+
C_g |h|_w .
\]
Choosing $N$ sufficiently large, depending on $\alpha$, yields
\eqref{eq:uniform-lasota-yorke}.
Similarly, combining the weak bounds from Proposition~\ref{prop:lasota-continuous-diffusion} and Lemma~\ref{lemma:transport-diff-in-z}, we obtain
\[
|
\cL_{g, \eps}
\circ
\e^{- 8\eps \pi^2 }
\cL_{\alpha, \eps}
\circ
\e^{N  \eps (\Delta_{x,y} - 4\pi^2)}
h
|_w
\leq
C_g |h|_w ,
\]
which proves \eqref{eq:weak-full-operator}.

\subsubsection{Weak--strong convergence}
We now prove \eqref{eq:weak-strong-convergence-full-operator}.
By Proposition~\ref{prop:lasota-continuous-diffusion}, together with the inequality
$|h|_w\leq \|h\|$, we have
\[
\|
\e^{- 8\eps \pi^2 }
\cL_{\alpha, \eps}
\circ
\e^{N  \eps (\Delta_{x,y} - 4\pi^2)}
h
\|
\leq
C \| h \|.
\]
Moreover, Lemma~\ref{lemma:smooth multiplier} implies
\[
|
\cL_{g,0}h
|_w
\leq
C_g |h|_w .
\]
Using Lemma~\ref{lemma:transport-diff-in-z} together with the triangle inequality, we obtain
\begin{align*}
&
\Big|
\cL_{g, \eps}
\circ
\e^{- 8\eps \pi^2 }
\cL_{\alpha, \eps}
\circ
\e^{N  \eps (\Delta_{x,y} - 4\pi^2)}
h
-
\cL_{g, 0}
\circ
\cL_{\alpha} h
\Big|_w
\\
&\qquad\leq
\Big|
(\cL_{g, \eps} - \cL_{g,0})
\circ
\e^{- 8\eps \pi^2 }
\cL_{\alpha, \eps}
\circ
\e^{N  \eps (\Delta_{x,y} - 4\pi^2)}
h
\Big|_w
\\
&\qquad\quad
+
\Big|
\cL_{g, 0}
\Big(
\e^{- 8\eps \pi^2 }
\cL_{\alpha, \eps}
\circ
\e^{N  \eps (\Delta_{x,y} - 4\pi^2)}
h
-
\cL_{\alpha} h
\Big)
\Big|_w
\\
&\qquad\leq
C_g \eps^{\beta/2}\|h\|
+
C_g
\Big|
\e^{-8\eps \pi^2}
\cL_{\alpha,\eps}
\circ
\e^{N\eps(\Delta_{x,y}-4\pi^2)}
h
-
\cL_\alpha h
\Big|_w .
\end{align*}
For the remaining term, we first note that Lemma~\ref{lemma:gaussian-bound} gives
\[
|
\e^{N\eps\Delta}h-h
|_w
\leq
C (N\eps)^{\beta/2}\|h\|.
\]
Using Proposition~\ref{prop:lasota-continuous-diffusion}, together with the weak boundedness of $\cL_\alpha$, we therefore deduce
\[
\Big|
\e^{-8\eps \pi^2}
\cL_{\alpha,\eps}
\circ
\e^{N\eps(\Delta_{x,y}-4\pi^2)}
h
-
\cL_\alpha h
\Big|_w
\leq
C_{\alpha,N}
\eps^{\beta/2}\|h\|.
\]
Combining the previous estimates proves
\eqref{eq:weak-strong-convergence-full-operator}.

\section{Kernel bounds} \label{sec:kernel} 

In this section, we prove Lemma~\ref{lemma:kernel bound}.
We begin by studying the stochastic flows associated with the planar velocity field $u_{\alpha, N}^{(1,2)}$ defined in \eqref{eq:u_def}. Let $(\Omega,\mathcal F,\mathbb P)$ be a probability space, and let $(W_t)_{t \geq 0}$ be a standard Brownian motion defined on this space. We denote by $\Phi_t^\eps$ the stochastic flow associated with the time-shifted velocity field $u_{\alpha, N}^{(1,2)}(t+N)$, namely the solution to the SDE
\begin{equation}\label{eq:SDE-T2-u-alpha-N}
\dd\Phi_t^\eps =u_{\alpha, N}^{(1,2)}(t+N,\Phi_t^\eps)\,\dd t+ \sqrt{2\eps}\,\dd W_t .    
\end{equation}
We also denote by $\Phi_t$ the deterministic flow corresponding to the case $\eps=0$.
We first record the following elementary probabilistic estimate.
\begin{lemma}
\label{lemma:gaussian distance bound}
There exist constants $c,C>0$ depending on $\| u_{\alpha, N} \|_{L^\infty_t W^{1, \infty}_x}$ and $T>0$ such that for every $x \in \mathbb{T}^2$, every
$\varepsilon \in (0,1]$, every $r > 0$, and every $t \in [0,T]$, it holds 
\begin{equation}
\mathbb{P}(\omega : \dd_{\TT^2} ( \Phi_t^\eps (x, \omega),\Phi_t(x)) \ge r\bigr)
\le
C\exp \left(-c \frac{r^2}{\varepsilon}\right) \,.
\end{equation}
\end{lemma}

\begin{proof}
We work with a lift of the processes to $\RR^2$. From \eqref{eq:SDE-T2-u-alpha-N}, we have the pathwise identity 
\[
  \Phi_t^\eps(x,\omega)-\Phi_t(x)
=
\int_0^t
\Bigl(u_{\alpha,N}^{(1,2)}(s+N,\Phi_s^\eps(x,\omega))
-u_{\alpha,N}^{(1,2)}(s+N,\Phi_s(x))\Bigr)\,\dd s
+\sqrt{2\eps}\,W_t(\omega).
\]
Taking suprema on $[0,t]$, using the Lipschitz bound $\| u \|_{L^\infty_t W^{1, \infty}_x}  \leq C$, and applying Gr\"onwall gives
\[
\sup_{0\le s\le t}
|\Phi_s^\eps(x)-\Phi_s(x)|
\le
\sqrt{2\eps}\,\e^{Ct}
\sup_{0\le s\le t}|W_s|,
\qquad
0\le t\le T.
\]
Hence, we obtain
\[
\mathbb{P}\bigl(\dd_{\TT^2}( \Phi_t^\eps(x),\Phi_t(x)) \ge r\bigr)
\le
\mathbb{P}\left(
\sup_{0\le s\le t}|W_s|
\ge
\frac{r \e^{-C t}}{\sqrt{2\varepsilon}}
\right).
\]
By the Doob maximal inequality  for any $\lambda \geq 0$ we have
\[
\mathbb{P}\left(\sup_{0\le s\le t}|W_s|\ge \lambda\right)
\le
 C\e^{- c\lambda^2/t}\,,
\]
from where the proof follows.
\end{proof}
Furthermore, we shall require the following elementary estimates on periodised Gaussian kernels.

\begin{lemma}
\label{lemma:periodised gaussian}
Let $\Gamma$ be a positive definite matrix. Then, there exist constants $C,c>0$ depending only on $\Gamma$, such that for every $\eps\in(0,1]$ and every $x,y\in\TT^2$, there holds
\[
\frac{1}{C\eps}
\exp\left(
-\frac{c}{\eps}\dd_{\TT^2}^2(y,x)
\right)
\le
\theta[y;\eps\Gamma](x)
\le
\frac{C}{\eps}
\exp\left(
-\frac{1}{c\eps}\dd_{\TT^2}^2(y,x)
\right).
\]
Furthermore, there holds the derivative bound
\[
|D_x\theta[y;\eps\Gamma](x)|
\le
\frac{C}{\eps^{3/2}}
\exp\left(
-\frac{1}{c\eps}\dd_{\TT^2}^2(y,x)
\right).
\]
\end{lemma}

\begin{proof}
From the definition of $\theta[y;\eps\Gamma](x)$, we write
\begin{equation}
\label{eq:periodised Gaussian expression}
\theta[y;\eps\Gamma](x)
=
\sum_{k\in\ZZ^2}
\frac{1}{2\pi\eps|\Gamma|^{1/2}}
\exp\left(
-\frac{1}{2\eps}
(y+k-x)^T\Gamma^{-1}(y+k-x)
\right).
\end{equation}
The lower bound follows immediately from the contribution of the minimizing lattice point. For the upper bound, we note that
\[
(y+k-x)^T\Gamma^{-1}(y+k-x)
\ge
\frac{1}{\lambda_{\max}(\Gamma)}
|y+k-x|^2,
\]
where $\lambda_{\max}(\Gamma)$ is the maximal eigenvalue of $\Gamma$. Moreover,
\[
|y+k-x|^2
\ge
\frac12|y+k-x|^2
+
\frac12\dd_{\TT^2}^2(y,x).
\]
Substituting these bounds into \eqref{eq:periodised Gaussian expression} yields
\[
\theta[y;\eps\Gamma](x)
\le
\frac{1}{2\pi\eps|\Gamma|^{1/2}}
\e^{-\frac{1}{4\eps\lambda_{\max}(\Gamma)}
\dd_{\TT^2}^2(y,x)}
\sum_{k\in\ZZ^2}
\e^{-\frac{1}{4\eps\lambda_{\max}(\Gamma)}
|y+k-x|^2}.
\]
The summation term is uniformly bounded for $\eps\in(0,1]$ and $x,y\in\TT^2$, since the Gaussian tails are summable uniformly in the torus shifts. This proves the upper bound.

Finally, differentiating the Gaussian explicitly gives
\[
\left|
D_x
\exp\left(
-\frac{1}{2\eps}
(y+k-x)^T\Gamma^{-1}(y+k-x)
\right)
\right|
\le
C_\Gamma\eps^{-1/2}
\exp\left(
-\frac{1}{4\eps}
(y+k-x)^T\Gamma^{-1}(y+k-x)
\right),
\]
from which the derivative estimate follows immediately.
\end{proof}

\subsection{The model problem} \label{subsec:model} 
This section is devoted to a detailed derivation of the model kernels $Q_{\alpha,\eps}^\ell(x,y)$ defined in \eqref{d:model-ell-operator} and employed in the proofs of Proposition \ref{prop: uniform Lasota Yorke} and Proposition \ref{prop:error decomposition}. In keeping with the theme of the paper, the model kernels, and in particular the error bounds of Lemma \ref{lemma:kernel bound}, are most naturally derived from a pathwise perspective.

Indeed, the argument hinges on the following observation: the solution at time
$t$ of the two-dimensional passive vector equation \eqref{passive-vector},
associated with the planar velocity field
$u_{\alpha,N}^{(1,2)}(t+N)$ and divergence-free initial datum
$B_{\initial}:\TT^2\to\RR^2$, is given by
\[
B(t,x)
=\mathbb E\Big[(D\Phi_t^\eps B_{\initial})
\circ (\Phi_t^\eps)^{-1}(x)\Big],
\]
where $\Phi_t^\eps:\TT^2 \times \Omega \to\TT^2$ in \eqref{eq:SDE-T2-u-alpha-N}.

In particular, this immediately yields the existence of a Green matrix
$K_{\alpha,\eps,t}:\TT^2\times\TT^2\to\RR^{2\times 2}$ satisfying
\begin{equation}
\label{eq:greens matrix expression}
\alpha^2 K_{\alpha,\eps,t}(x,y)
=
\mathbb E\Big[
D\Phi_t^\eps(y)\,
\delta(x-\Phi_t^\eps(y))
\Big],
\footnote{
This expression is well defined since, by the result of
\cite{FlandoliGubinelliPriola}, the map
$y\mapsto \Phi_t^\eps(y)$ is almost surely
$C^{1+\xi}(\TT^2;\TT^2)$ for every $\xi\in(0,1)$.
}
\end{equation}
so that
\[
B(t,x)
=
\int_{\TT^2}
\alpha^2 K_{\alpha,\eps,t}(x,y)\,
B_{\initial}(y)\,\dd y .
\]
Throughout the section we use the notation
\[
K_{\alpha,\eps}(x,y)
=
K_{\alpha,\eps,t}|_{t=2}(x,y).
\]
Furthermore, we shall frequently extend the affine maps
$T_\alpha|_{\mathcal M_\ell}$ to $\RR^2$. In view of the explicit expressions \eqref{eq:Talpha2}\eqref{eq:Talpha1}, we therefore write  
\begin{equation}\label{d:T-alpha-ell}
T_\alpha^\ell x
=T_\alpha|_{\mathcal M_\ell}x
=A_\ell x+b_\ell , \qquad \ell=1,\ldots, 4,
\end{equation}
with $A_\ell$ as in \eqref{eq:matrixMi}, viewed as a linear transformation on all of $\RR^2$,  and a suitable translation $b_\ell$.

With these preliminaries in place, we now introduce the SDE associated with the $\ell$-th model problem on $\RR^2$. The construction is motivated by the fact
that the velocity field $u_{\alpha,N}^{(1,2)}$ is piecewise affine. For example, consider the integral curve advected by
$u_{\alpha,N}^{(1,2)}$ starting from
$(x,y)\in\mathcal M_1$. Suppose that, for all
$t\in[N,N+1/2]$, the trajectory remains in the region
$\{x\in(1/2,1)\}$, and that, for all
$t\in[N+1/2,N+1)$, it remains in the region
$\{y\in(1/2,1)\}$. Then, from the definition of
$u_{\alpha,N}^{(1,2)}$ in \eqref{eq:shears}, the trajectory experiences only
the affine velocity field
\[
u_{\alpha,N}^{(1,2)}(t,x,y)=
\begin{cases}
\begin{pmatrix}
0 \\
2\alpha(x-1/2)
\end{pmatrix},
& t\in[N,N+1/2),\\[0.4cm]
\begin{pmatrix}
2\alpha(y-1/2)\\
0
\end{pmatrix},
& t\in[N+1/2,N+1).
\end{cases}
\]
Guided by this observation, we define the $\ell$-th model SDE on $\RR^2$ by
\begin{equation} \label{eq:model-SDE}
\dd M_t^{\eps,\ell}
=
\begin{cases}
(V_\ell M_t^{\eps,\ell}+C_\ell)\,\dd t
+\sqrt{2\eps}\,\dd W_t,
& 0\le t<\frac12,\\[1mm]
(H_\ell M_t^{\eps,\ell}+D_\ell)\,\dd t
+\sqrt{2\eps}\,\dd W_t,
& \frac12\le t<1,
\end{cases}
\qquad
M_0^{\eps,\ell}=y,
\end{equation}
where we have the four matrices 
\[
V_\ell=
\begin{pmatrix}
0&0\\
\pm2\alpha&0
\end{pmatrix},
\qquad
H_\ell=
\begin{pmatrix}
0&\pm2\alpha\\
0&0
\end{pmatrix},
\]
and corresponding translations 
\[
C_\ell=
\begin{pmatrix}
0\\
\mp\alpha
\end{pmatrix},
\qquad
D_\ell=
\begin{pmatrix}
\mp\alpha\\
0
\end{pmatrix}.
\]
These correspond precisely to the four possible affine branches of the velocity
field $u_{\alpha,N}^{(1,2)}$, and therefore the index
$\ell$ ranges over $\ell=1,2,3,4$.

We note that \eqref{eq:model-SDE} is simply the composition of two Ornstein--Uhlenbeck
processes. Consequently, at time $1$ the law of the solution is Gaussian, with
mean \eqref{d:T-alpha-ell}
and covariance $\eps\Gamma_\alpha^\ell$. Here
$A_\ell\in\RR^{2\times2}$ is one of the matrices listed in
\eqref{eq:matrixMi}, $b_\ell\in\RR^2$ is the corresponding translation vector,
and $\Gamma_\alpha^\ell$ is a positive definite matrix depending only on $\alpha$ and $\ell$.

By a slight abuse of notation, we also denote by
$T_\alpha^\ell:\TT^2\to\TT^2$ the induced toral map.
The central comparison result is the following.

\begin{lemma}
\label{lemma:TV comparison} 
Let $y\in\mathcal M_\ell$, viewed as an element of $[0,1]^2 \subset \RR^2$, and let $\Phi_t^\eps$ be the solution to the SDE
\eqref{eq:SDE-T2-u-alpha-N} on $\RR^2$, where the velocity field
$u_{\alpha,N}^{(1,2)}$ is extended periodically to $\RR^2$, with initial
condition
\[
\Phi_0^\eps(y,\omega)=y.
\]
Assume that the Brownian motion driving
\eqref{eq:SDE-T2-u-alpha-N} is the same as the one driving the
$\ell$-th model SDE \eqref{eq:model-SDE}. Define
\[
\Omega_y
=
\left\{
\omega:
\sup_{t\in[0,1]}|W_t|
\le
\frac{1}{10\alpha\sqrt{2\eps}}
\dist_{\RR^2}(y,\partial\mathcal M_\ell)
\right\}.
\]
Then, for every $\omega\in\Omega_y$,
\[
\Phi_1^\eps \equiv M_1^{\eps,\ell} \pmod{1},
\qquad D\Phi_1^\eps=DM_1^{\eps,\ell}.
\]
Finally, if $\overline\mu_y^\ell(\dd z)$ denotes the law of
$M_1^{\eps,\ell}(y)$ on $\RR^2$, then
\[
\overline\mu_{y+m}^\ell =(\tau_{A_\ell m})_\#\overline\mu_y^\ell,
\]
where $\tau_k$ denotes translation by $m\in\ZZ^2$.
\end{lemma}

\begin{proof}
We present the proof in the case $\ell=1$; the remaining cases are identical. Let
$y=(y_1,y_2)\in\mathcal M_1$, namely
\begin{equation}\label{eq:modys}
y_1\in(1/2,1),
\qquad
y_2+\alpha y_1\in(k+1/2,k+1),    
\end{equation}
for some integer $k\in\ZZ$. For times $t\in[0,1/2)$, the process $\Phi_t^\eps$ satisfies
\begin{align}
(\Phi_t^\eps)^{(1)}
&=y_1+\sqrt{2\eps}\,W_t^{(1)},
\\
(\Phi_t^\eps)^{(2)} 
&=y_2+ 2\alpha\int_0^t\Big|(y_1+\sqrt{2\eps}W_s^{(1)})\pmod{1} -\frac12 \Big|\,\dd s
+ \sqrt{2\eps}\,W_t^{(2)}.
\end{align}
Since
\[
\dist_{\RR^2}(y,\partial\mathcal M_\ell)
\le
\min\{
|y_1-1/2|,
|y_1-1|
\},
\]
for every $\omega\in\Omega_y$ we have
\[
\Big|
y_1+\sqrt{2\eps}W_s^{(1)}
\pmod{1}
-\frac12
\Big|
=
y_1-\frac12
+
\sqrt{2\eps}W_s^{(1)},
\qquad
s\in[0,1/2].
\]
Therefore,
\[
(\Phi_{1/2}^\eps)^{(2)}
=
y_2
+
\alpha(y_1-1/2)
+
2\alpha
\int_0^{1/2}
\sqrt{2\eps}W_s^{(1)}\,\dd s
+
\sqrt{2\eps}W_{1/2}^{(2)}.
\]
It follows immediately that
\[
\Phi_{1/2}^\eps(y,\omega)
=
M_{1/2}^{\eps,\ell}(y,\omega)
\pmod{1},
\qquad
\omega\in\Omega_y.
\]
On the interval $[1/2,1]$, the horizontal shear gives
\begin{align}
(\Phi_t^\eps)^{(1)}
&=
y_1
+
\sqrt{2\eps}W_t^{(1)}
+
2\alpha
\int_{1/2}^t
\Big|
(\Phi_s^\eps)^{(2)}\pmod{1}
-\frac12
\Big|
\,\dd s,
\\
(\Phi_t^\eps)^{(2)}
&=
(\Phi_{1/2}^\eps)^{(2)}
+
\sqrt{2\eps}(W_t^{(2)}-W_{1/2}^{(2)}).
\end{align}
Using the previous expression,
\[
(\Phi_t^\eps)^{(2)}
=
y_2+\alpha y_1-k
+
2\alpha
\int_0^{1/2}
\sqrt{2\eps}W_s^{(1)}\,\dd s
+
\sqrt{2\eps}W_t^{(2)}
\pmod{1}.
\]
Moreover, for every $\omega\in\Omega_y$,
\begin{align}
\sqrt{2\eps}
\left|
2\alpha
\int_0^{1/2}W_s^{(1)}\,\dd s
+
W_t^{(2)}
\right|
&\le
\frac14
\dist_{\RR^2}(y,\partial\mathcal M_1)
\\
&\le
\frac14
\min\{
|y_2+\alpha y_1-k-1/2|,
|y_2+\alpha y_1-k-1|
\}.
\end{align}
From \eqref{eq:modys},
we deduce
\[
(\Phi_t^\eps)^{(2)}\pmod{1}
=
y_2+\alpha y_1-k
+
2\alpha
\int_0^{1/2}
\sqrt{2\eps}W_s^{(1)}\,\dd s
+
\sqrt{2\eps}W_t^{(2)}
\in[1/2,1).
\]
Hence, at time $t=1$,
\begin{align}
(\Phi_1^\eps)^{(1)}
&=
y_1
+
\sqrt{2\eps}W_1^{(1)}
+
2\alpha
\int_{1/2}^1
\left[
y_2+\alpha y_1-k-\frac12
+
2\alpha
\int_0^{1/2}
\sqrt{2\eps}W_r^{(1)}\,\dd r
+
\sqrt{2\eps}W_s^{(2)}
\right]
\,\dd s
\\
&=
y_1(1+\alpha^2)
+\alpha y_2
-\alpha k
-\frac{\alpha}{2}
+
\sqrt{2\eps}W_1^{(1)}
+
2\alpha^2
\int_0^{1/2}
\sqrt{2\eps}W_r^{(1)}\,\dd r
+
2\alpha
\int_{1/2}^1
\sqrt{2\eps}W_s^{(2)}\,\dd s .
\end{align}
This coincides with the explicit expression for
$M_1^{\eps,\ell}(y)$, and therefore
\[
\Phi_1^\eps(y,\omega)
=
M_1^{\eps,\ell}(y,\omega)
\pmod{1},
\qquad
\omega\in\Omega_y.
\]
The equality of derivatives follows immediately, since the same computation
holds in a neighbourhood of radius
$\frac{1}{10\alpha}\dist_{\RR^2}(y,\partial\mathcal M_\ell)$ around $y$.

Finally, the explicit formula for $M_1^{\eps,\ell}$ shows that
\[
M_1^{\eps,\ell}(y+m)
=
M_1^{\eps,\ell}(y)
+
A_\ell m,
\]
from which the translation identity for
$\overline\mu_y^\ell$ follows immediately.
\end{proof}

We next introduce the model PDE on $\RR^2$ associated with the process
$M^{\eps,\ell}$. Namely, we consider the solution
$B:[0,2]\times\RR^2\to\RR^2$ to \eqref{passive-vector}, with velocity field
given by alternating affine shears:
\begin{equation}
\label{eq:model PDE}
\begin{cases}
\partial_t B
+
(V_\ell x+C_\ell)\cdot\nabla B
-
B\cdot\nabla(V_\ell x+C_\ell)
=
\eps\Delta B,
& t\in[0,1/2),
\\[2mm]
\partial_t B
+ (H_\ell x+D_\ell)\cdot\nabla B
- B\cdot\nabla(H_\ell x+D_\ell) = \eps\Delta B, & t\in[1/2,1), \\[2mm]
\partial_t B =\eps\Delta B,
& t\in[1,2),
\end{cases}
\end{equation}
with periodic initial datum
$B(0)=B_{\initial}\in L^2(\TT^2)$ extended periodically to $\RR^2$. The corresponding model kernels are precisely the Green matrices associated
with \eqref{eq:model PDE}.

\begin{lemma}
\label{lemma:model kernel expression}
Let $\ell\in\{1,2,3,4\}$. Then the time-$2$ solution to
\eqref{eq:model PDE} satisfies
\[
B(2,x)
=
\int_{\TT^2}
\alpha^2Q_{\alpha,\eps}^\ell(x,y)
B_{\initial}(y)\,\dd y ,
\]
where
\[
\alpha^2Q_{\alpha,\eps}^\ell(x,y)
=
\int_{\TT^2\times GL(\RR^2)}
A\,\theta[z;2\eps\Id](x)\,
\mathcal P_y^\ell(\dd z,\dd A),
\]
and $\mathcal P_y^\ell$ denotes the law of
$(M_1^{\eps,\ell}(y),DM_1^{\eps,\ell}(y))$ on
$\TT^2\times GL(\RR^2)$.
Moreover,
\[
Q_{\alpha,\eps}^\ell(x,y)
=
\alpha^{-2}
A_\ell
\theta\Big[
y;
\eps\big(
A_\ell^{-1}\Gamma_\alpha^\ell A_\ell^{-T}
+
2A_\ell^{-1}A_\ell^{-T}
\big)
\Big]
\big((T_\alpha^\ell)^{-1}(x)\big).
\]
\end{lemma}

\begin{proof}
Let $\overline\mu_y^\ell$ denote the law of the lifted random variable
$M^{\ell,\eps}_1(y)$ on $\RR^2$. By the stochastic representation formula for solutions to the passive vector equation on $[0,1]$, and since in the present affine setting
$$
DM^{\ell,\eps}_1(y)=A_\ell ,
$$
we have
$$
B(1,x)
=
\int_{\RR^2}
A_\ell \,
\mathbb E\bigl[\delta(x-M^{\ell,\eps}_1(y))\bigr]
B_{\ini}(y)\,\dd y .
$$
Applying the heat flow on the time interval $[1,2]$, we obtain
$$
B(2,x)
=
A_\ell
\int_{\RR^2}
\int_{\RR^2}
\overline\theta[z;2\eps \Id](x)\,
\overline\mu_y^\ell(\dd z)\,
B_{\ini}(y)\,\dd y .
$$
Using the periodicity of \(B_{\ini}\), this can be rewritten as
$$
B(2,x)
=
A_\ell
\int_{[0,1)^2}
\sum_{m\in\ZZ^2}
\int_{\RR^2}
\overline\theta[z;2\eps \Id](x)\,
\overline\mu_{y+m}^\ell(\dd z)\,
B_{\ini}(y)\,\dd y .
$$
By Lemma \ref{lemma:TV comparison}, it holds
$
\overline\mu^\ell_{y+m} = (\tau_{A_\ell m})_\#\overline\mu^\ell_y,
$
and so
$$
\sum_{m\in\ZZ^2}
\int_{\RR^2}
\overline\theta[z;2\eps \Id](x)\,
\overline\mu_{y+m}^\ell(\dd z)
=
\sum_{m\in\ZZ^2}
\int_{\RR^2}
\overline\theta[z+A_\ell m;2\eps \Id](x)\,
\overline\mu_y^\ell(\dd z).
$$
Since $A_\ell$ and its inverse have integer entries, it defines a bijection
$A_\ell:\ZZ^2\to\ZZ^2$, and therefore
$$
\sum_{m\in\ZZ^2}
\int_{\RR^2}
\overline\theta[z+A_\ell m;2\eps \Id](x)\,
\overline\mu_y^\ell(\dd z)
=
\int_{\RR^2}
\sum_{n\in\ZZ^2}
\overline\theta[z+n;2\eps \Id](x)\,
\overline\mu_y^\ell(\dd z).
$$
By definition of the periodised heat kernel and symmetry in $x$ and $z$,
$$
\theta[z;2\eps \Id](x)
=
\sum_{n\in\ZZ^2}
\overline\theta[z+n;2\eps \Id](x),
$$
and hence
$$
B(2,x)
=
A_\ell
\int_{[0,1)^2}
\int_{\RR^2}
\theta[z;2\eps \Id](x)\,
\overline\mu_y^\ell(\dd z)\,
B_{\ini}(y)\,\dd y .
$$
Since \(\theta[\cdot;2\eps\Id](x)\) is \(\ZZ^2\)-periodic in the \(z\)-variable, we may equivalently write
$$
\int_{\RR^2}
\theta[z;2\eps \Id](x)\,
\overline\mu_y^\ell(\dd z)
=
\int_{\TT^2}
\theta[z;2\eps \Id](x)\,
\mu_y^\ell(\dd z),
$$
where \(\mu_y^\ell\) denotes the projection of \(\overline\mu_y^\ell\) onto \(\TT^2\). Thus
$$
B(2,x)
=
\int_{\TT^2}
A_\ell
\left(
\int_{\TT^2}
\theta[z;2\eps \Id](x)\,
\mu_y^\ell(\dd z)
\right)
B_{\ini}(y)\,\dd y .
$$
Equivalently, if \(\mathcal P_y^\ell\) denotes the joint law of
\(
(M^{\ell,\eps}_1(y),DM^{\ell,\eps}_1(y))
\)
on \(\TT^2\times GL(\RR^2)\), then we conclude the representation formula
$$
B(2,x)
=
\int_{\TT^2}
\left(
\int_{\TT^2\times GL(\RR^2)}
A\theta[z;2\eps\Id](x)\,
\mathcal P_y^\ell(\dd z,\dd A)
\right)
B_{\ini}(y)\,\dd y .
$$
It remains to compute this kernel explicitly. By the Ornstein--Uhlenbeck structure of the SDE \eqref{eq:model-SDE}, the law of \(M^{\ell,\eps}_1(y)\) is Gaussian on the torus with mean
\(T_\alpha^\ell(y)\) and covariance \(\eps\Gamma^\ell_\alpha\). Therefore convolution with the heat kernel gives
$$
\alpha^2Q^\ell_{\alpha,\eps}(x,y)
=
A_\ell
\theta\bigl[
T_\alpha^\ell(y);
\eps(\Gamma^\ell_\alpha+2\Id)
\bigr](x).
$$
Writing
$$
T_\alpha^\ell(y)=A_\ell y+b_\ell,
$$
we have
$$
T_\alpha^\ell(y)-x
=
A_\ell y+b_\ell-x
=
A_\ell
\left(
y-A_\ell^{-1}(x-b_\ell)
\right)
=
A_\ell
\left(
y-(T_\alpha^\ell)^{-1}(x)
\right).
$$
Since $A_\ell\in GL(2,\ZZ)$ and $|\det A_\ell|=1$, the Gaussian change of variables yields
$$
\theta
\left[
T_\alpha^\ell(y);
\eps(\Gamma^\ell_\alpha+2\Id)
\right](x)
=
\theta\left[
y;
\eps\left(
A_\ell^{-1}\Gamma^\ell_\alpha A_\ell^{-T}
+
2A_\ell^{-1}A_\ell^{-T}
\right)
\right]
((T_\alpha^\ell)^{-1}(x)),
$$
concluding the proof.
\end{proof}

We finally define the model Green's matrix on $\TT^2$ exactly as in Section \ref{sec:model problem} by gluing together the kernels associated to the four model problems $\ell=1,2,3,4$, namely\begin{align} \label{d:kernel-model}Q_{\alpha,\eps}(x,y)=\sum_{\ell=1}^4\mathbbm 1_{\XX_\alpha(\mathcal M_\ell)}(x)Q^\ell_{\alpha,\eps}(x,y).\end{align}
This corresponds to selecting the appropriate affine model according to the branch of the deterministic dynamics reaching \(x\).

\subsection{Proof of Lemma \ref{lemma:kernel bound}}

We now have all the ingredients needed to prove Lemma \ref{lemma:kernel bound}. The proof is divided into several steps.

\medskip

\noindent
\textbf{Step 1: Proof of \eqref{eq:kernel-global}.}
Similarly to Lemma \ref{lemma:model kernel expression}, the kernel $K_{\alpha,\eps}$ admits the representation
\begin{equation}
\label{eq:full-kernel-representation}
\alpha^2 K_{\alpha,\eps}(x,y)
=
\int_{\TT^2\times GL(\RR^2)}
A \,
\theta[z;2\eps \Id](x)\,
\mathcal P_y(\dd z,\dd A),
\end{equation}
where $\mathcal P_y(\dd z,\dd A)$ denotes the law of the random variable
$(\Phi_1^\eps(y),D\Phi_1^\eps(y))$ on $\TT^2\times GL(\RR^2)$. Since the Jacobian matrices satisfy the uniform bound
\[
|D\Phi_1^\eps(y)|\le C\alpha^2,
\]
it follows that the support of $\mathcal P_y$ is contained in
\[
\{A\in GL(\RR^2): |A|\le C\alpha^2\}.
\]
Consequently,
\begin{equation}
\label{eq:kernel-first-bound}
|K_{\alpha,\eps}(x,y)|
\le
C
\int_{\TT^2}
\theta[z;2\eps\Id](x)\,
\mu_y(\dd z),
\end{equation}
where $\mu_y$ denotes the law of $\Phi_1^\eps(y)$ on $\TT^2$. By Lemma \ref{lemma:gaussian distance bound}, there exist constants
$c_\alpha,C_\alpha>0$ such that
\[
\mu_y\bigl(\TT^2\setminus B_r(T_\alpha(y))\bigr)
\le
C_\alpha
\exp\left(-\frac{c_\alpha r^2}{\eps}\right).
\]
Fix now
\[
r:=\dd_{\TT^2}(T_\alpha(y),x).
\]
Splitting the integral in \eqref{eq:kernel-first-bound}, we obtain
\begin{equation}
|K_{\alpha,\eps}(x,y)|
\le
C
\int_{B_{r/2}(T_\alpha(y))}
\theta[z;2\eps\Id](x)\,
\mu_y(\dd z)+
C
\int_{\TT^2\setminus B_{r/2}(T_\alpha(y))}
\theta[z;2\eps\Id](x)\,
\mu_y(\dd z).
\label{eq:kernel-split}
\end{equation}
For $z\in B_{r/2}(T_\alpha(y))$, the triangle inequality gives
\[
\dd_{\TT^2}(x,z)
\ge
\dd_{\TT^2}(x,T_\alpha(y))
-
\dd_{\TT^2}(T_\alpha(y),z)
\ge
\frac{r}{2}.
\]
Hence, by Lemma \ref{lemma:periodised gaussian},
\[
\theta[z;2\eps\Id](x)
\le
\frac{C}{\eps}
\exp\left(
-\frac{c\dd_{\TT^2}^2(x,z)}{\eps}
\right)
\le
\frac{C}{\eps}
\exp\left(
-\frac{c r^2}{4\eps}
\right).
\]
Therefore,
\begin{equation}
\label{eq:first-piece-kernel}
\int_{B_{r/2}(T_\alpha(y))}
\theta[z;2\eps\Id](x)\,
\mu_y(\dd z)
\le
\frac{C_\alpha}{\eps}
\exp\left(
-\frac{c_\alpha r^2}{\eps}
\right).
\end{equation}
For the second term in \eqref{eq:kernel-split}, using again Lemma
\ref{lemma:periodised gaussian} together with the tail estimate on $\mu_y$ yields
\begin{equation}
\int_{\TT^2\setminus B_{r/2}(T_\alpha(y))}
\theta[z;2\eps\Id](x)\,
\mu_y(\dd z)\le
\frac{C}{\eps}
\mu_y\bigl(\TT^2\setminus B_{r/2}(T_\alpha(y))\bigr)\le
\frac{C_\alpha}{\eps}
\exp\left(
-\frac{c_\alpha r^2}{\eps}
\right).
\label{eq:second-piece-kernel}
\end{equation}
Combining \eqref{eq:kernel-split},
\eqref{eq:first-piece-kernel}, and
\eqref{eq:second-piece-kernel}, we conclude that
\[
|K_{\alpha,\eps}(x,y)|
\le
\frac{C_\alpha}{\eps}
\exp\left(
-\frac{c_\alpha}{\eps}
\dd_{\TT^2}^2(x,T_\alpha(y))
\right),
\]
which proves \eqref{eq:kernel-global}.

We now turn to the derivative bounds. Differentiating
\eqref{eq:full-kernel-representation} and using Lemma
\ref{lemma:periodised gaussian}, we obtain
\begin{align*}
|D_x^{\mathbf m}K_{\alpha,\eps}(x,y)|
=
\left|
\int_{\TT^2\times GL(\RR^2)}
A\,
D_x^{\mathbf m}\theta[z;2\eps\Id](x)\,
\mathcal P_y(\dd z,\dd A)
\right| \le
C_\alpha
\eps^{-1/2}
\int_{\TT^2}
\theta[z;c_\alpha\eps\Id](x)\,
\mu_y(\dd z).
\end{align*}
Repeating the previous argument yields
\[
|D_x^{\mathbf m}K_{\alpha,\eps}(x,y)|
\le
\frac{C_\alpha}{\eps^{3/2}}
\exp\left(
-\frac{c_\alpha}{\eps}
\dd_{\TT^2}^2(x,T_\alpha(y))
\right),
\]
which completes the proof of \eqref{eq:kernel-global}.

\smallskip

\noindent
\textbf{Step 2: Proof of \eqref{eq:error-kernel-bound} for $y\notin \cM_\ell$.}
Fix $x\in T_\alpha(\cM_\ell)$, so that
\[
Q_{\alpha,\eps}(x,y)=Q^\ell_{\alpha,\eps}(x,y).
\]
If $y\notin \cM_\ell$, then
\begin{align}
|K_{\alpha,\eps}(x,y)-Q_{\alpha,\eps}(x,y)|
&\le
|K_{\alpha,\eps}(x,y)|
+
|Q^\ell_{\alpha,\eps}(x,y)|
\nonumber
\\
&\le
\frac{C_\alpha}{\eps}
\left(
\exp\left(
-\frac{c_\alpha}{\eps}
\dd_{\TT^2}^2(T_\alpha(y),x)
\right)+
\exp\left(
-\frac{c_\alpha}{\eps}
\dd_{\TT^2}^2(T_\alpha^\ell(y),x)
\right)
\right),
\label{eq:error-bound-preliminary}
\end{align}
where the estimate on $Q^\ell_{\alpha,\eps}$ follows from Lemma
\ref{lemma:model kernel expression} together with Lemma
\ref{lemma:periodised gaussian}.
Since $T_\alpha^\ell$ agrees with $T_\alpha$ on $\cM_\ell$, and is injective on $\TT^2$, it holds
\[
T_\alpha(y)\notin T_\alpha(\cM_\ell),
\qquad
T_\alpha^\ell(y)\notin T_\alpha(\cM_\ell).
\]
Therefore, since $x\in T_\alpha(\cM_\ell)$, it holds
\[
\dd_{\TT^2}(T_\alpha(y),x)
\ge
\dd_{\TT^2}(x,S^\star),
\qquad
\dd_{\TT^2}(T_\alpha^\ell(y),x)
\ge
\dd_{\TT^2}(x,S^\star).
\]
Substituting into \eqref{eq:error-bound-preliminary} yields
\[
|K_{\alpha,\eps}(x,y)-Q_{\alpha,\eps}(x,y)|
\le
\frac{C_\alpha}{\eps}\left(
\exp\left(
-\frac{c_\alpha}{\eps}
\dd_{\TT^2}^2(T_\alpha(y),x)
\right)+
\exp\left(
-\frac{c_\alpha}{\eps}
\dd_{\TT^2}^2(T_\alpha^\ell(y),x)
\right)
\right)
\exp\left(
-\frac{c_\alpha}{\eps}
\dd_{\TT^2}^2(x,S^\star)
\right),
\]
which proves the desired estimate in the case $y\notin \cM_\ell$. Similarly, differentiating the kernels and using Lemma
\ref{lemma:periodised gaussian}, we obtain
\begin{align*}
|D_x^{\mathbf m}(K_{\alpha,\eps}(x,y)-Q_{\alpha,\eps}(x,y))|
&\le
\frac{C_\alpha}{\eps^{3/2}}
\left(
\exp\left(
-\frac{c_\alpha}{\eps}
\dd_{\TT^2}^2(T_\alpha(y),x)
\right) 
+
\exp\left(
-\frac{c_\alpha}{\eps}
\dd_{\TT^2}^2(T_\alpha^\ell(y),x)
\right)
\right),
\end{align*}
and the same argument concludes the derivative estimate.

\smallskip

\noindent
\textbf{Step 3: Proof of \eqref{eq:error-kernel-bound} for $y\in \cM_\ell$.}
Assume now that $y\in \cM_\ell$. In this case,
\begin{equation}
\label{eq:kernel-difference}
K_{\alpha,\eps}(x,y)-Q^\ell_{\alpha,\eps}(x,y)
=
\alpha^{-2}
\int_{\TT^2\times GL(\RR^2)}
A\,
\theta[z;2\eps\Id](x)\,
(\mathcal P_y-\mathcal P_y^\ell)(\dd z,\dd A).
\end{equation}
Since both $\mathcal P_y$ and $\mathcal P_y^\ell$ are supported on matrices
with norm bounded by $C\alpha^2$, Cauchy--Schwarz yields
\begin{equation}
|K_{\alpha,\eps}(x,y)-Q^\ell_{\alpha,\eps}(x,y)|
\le
C
\left(
\int_{\TT^2\times GL(\RR^2)}
\theta^2[z;2\eps\Id](x)\,
|\mathcal P_y-\mathcal P_y^\ell|(\dd z,\dd A)
\right)^{1/2}
\|\mathcal P_y-\mathcal P_y^\ell\|_{TV}^{1/2}.
\label{eq:kernel-cs}
\end{equation}
Using Lemma \ref{lemma:periodised gaussian},
\begin{align*}
\int_{\TT^2\times GL(\RR^2)}
\theta^2[z;2\eps\Id](x)\,
|\mathcal P_y-\mathcal P_y^\ell|(\dd z,\dd A)
&\le
\frac{C_\alpha}{\eps}
\left(
\int_{\TT^2}
\theta[z;2\eps\Id](x)\,
\mu_y(\dd z)
\right.
\\
&\qquad\qquad
\left.
+
\int_{\TT^2}
\theta[z;2\eps\Id](x)\,
\mu_y^\ell(\dd z)
\right).
\end{align*}
Since $T_\alpha^\ell(y)=T_\alpha(y)$ for $y\in \cM_\ell$, the estimates from
Step 1 imply
\[
\left(
\int_{\TT^2\times GL(\RR^2)}
\theta^2[z;2\eps\Id](x)\,
|\mathcal P_y-\mathcal P_y^\ell|(\dd z,\dd A)
\right)^{1/2}
\le
\frac{C_\alpha}{\eps}
\exp\left(
-\frac{c_\alpha}{\eps}
\dd_{\TT^2}^2(T_\alpha(y),x)
\right).
\]
Next, recall that $\mathcal P_y$ and $\mathcal P_y^\ell$ are respectively the
laws of
\[
(\Phi_1^\eps(y),D\Phi_1^\eps(y))
\qquad\text{and}\qquad
(M_1^{\eps,\ell}(y),DM_1^{\eps,\ell}(y)).
\]
Coupling the two processes with the same Brownian motion, we obtain
\[
\|\mathcal P_y-\mathcal P_y^\ell\|_{TV}
\le
\mathbb P\left(
(\Phi_1^\eps(y),D\Phi_1^\eps(y))
\neq
(M_1^{\eps,\ell}(y),DM_1^{\eps,\ell}(y))
\right).
\]
By Lemma \ref{lemma:TV comparison},
\begin{align*}
\|\mathcal P_y-\mathcal P_y^\ell\|_{TV}
\le
\mathbb P\left(
\sup_{t\in[0,1]}|W_t|
\ge
C_\alpha
\eps^{-1/2}
\dd_{\TT^2}(y,\partial\cM_\ell)
\right)\le
C
\exp\left(
-\frac{c_\alpha}{\eps}
\dd_{\TT^2}^2(y,\partial\cM_\ell)
\right).
\end{align*}
Since $T_\alpha$ is bi-Lipschitz on $\cM_\ell$, there exists
$C_\alpha>0$ such that
\[
\dd_{\TT^2}(y,\partial\cM_\ell)
\ge
C_\alpha
\dd_{\TT^2}(T_\alpha(y),\partial T_\alpha(\cM_\ell)).
\]
Therefore,
\[
\dd_{\TT^2}(x,S^\star)
\le
\dd_{\TT^2}(x,T_\alpha(y))
+
C_\alpha^{-1}
\dd_{\TT^2}(y,\partial\cM_\ell).
\]
Combining this inequality with \eqref{eq:kernel-cs}, we conclude that
\begin{align*}
|K_{\alpha,\eps}(x,y)-Q_{\alpha,\eps}(x,y)|
&\le
\frac{C_\alpha}{\eps}
\exp\left(
-\frac{c_\alpha}{\eps}
\dd_{\TT^2}^2(T_\alpha(y),x)
\right)
\exp\left(
-\frac{c_\alpha}{\eps}
\dd_{\TT^2}^2(y,\partial\cM_\ell)
\right)
\\
&\le
\frac{C_\alpha}{\eps}
\exp\left(
-\frac{\bar c_\alpha}{\eps}
\dd_{\TT^2}^2(T_\alpha(y),x)
\right)
\exp\left(
-\frac{\bar c_\alpha}{\eps}
\dd_{\TT^2}^2(x,S^\star)
\right),
\end{align*}
for some constant $\bar c_\alpha>0$. Finally, differentiating \eqref{eq:kernel-difference} and using Lemma
\ref{lemma:periodised gaussian}, we obtain
\[
|D_x^{\mathbf m}(K_{\alpha,\eps}(x,y)-Q_{\alpha,\eps}(x,y))|
\le
C_\alpha
\eps^{-1/2}
\int_{\TT^2\times GL(\RR^2)}
\theta[z;c_\alpha\eps\Id](x)\,
|\mathcal P_y-\mathcal P_y^\ell|(\dd z,\dd A).
\]
Repeating the previous argument yields the corresponding derivative estimate,
thereby concluding the proof of the lemma.

\appendix

\section{Proof of Lemma \ref{lemma:technical}} \label{appendix}
Here we prove  Lemma \ref{lemma:technical}, the main technical lemma of the paper.  We recall that the set of admissible curves is denoted by  $\Sigma$  \eqref{d:Sigma}, the stable cone is denoted by $C_s$ \eqref{eq:cones}, the map $\XX_\alpha$ and the sets $\mathcal{M}_\ell$ with $\ell =1, 2,3,4$ are defined in Section \ref{sec:themap}.

\begin{proof}
The proof of \eqref{enum:0} follows from the definition of the $\XX_\alpha (\cM_\ell)$ for $\ell=1,2,3,4$. Indeed, we have that $\partial \XX_\alpha (\cM_\ell)$ are curves $y = \pm \frac{1}{\alpha} x$ periodised on the torus, which are almost horizontal lines, see Figure \ref{figApp:1}. Also by definition of admissible curves $W \in \Sigma$ we have that such affine lines are almost vertical with tangent vector $\mathbf{v} =(a, b) $ with $|a| \leq \frac{2}{\alpha} b$. Thanks to Lemma \ref{lemma:uniformly hyperbolic} we deduce that the tangent vector $\frac{\XX_\alpha^{-1} (\vvv)}{|\vvv|}$ is still in the stable cone ${C}_s$ if $\alpha \geq \alpha_0$ is large enough and hence each subcurve of $\XX_\alpha^{-1} (W_j)$ of length less than $1$ is an admissible curve. Since $|D \XX_\alpha^{-1}| \leq 2 \alpha^2$ there are at most $ N_\alpha \leq 3+ \lceil \alpha^2 |W_j| \rceil$  
curves $(\tilde W_{j,k})_{k=1}^{N_\alpha}$  such that $\tilde W_{j,k} \in \Sigma$ and  $ \XX_\alpha^{-1} (W_j) = \cup_{k=1}^{N_\alpha} \tilde W_{j,k}$.
Hence the properties follow up to choosing $\alpha \geq \alpha_0$ sufficiently large.

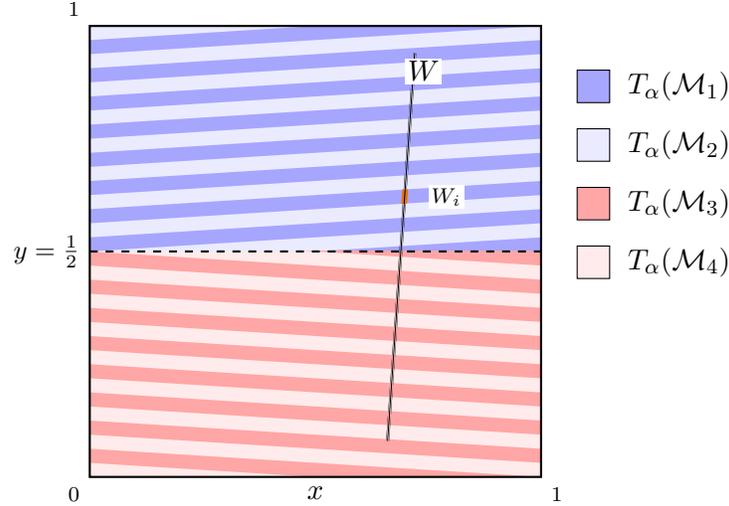
\begin{figure}[ht]
\centering
\begin{tikzpicture}[scale=6]

% Large alpha (example)
\def\A{16}

% Colors
\colorlet{Mone}{blue!35}
\colorlet{Mtwo}{blue!8}
\colorlet{Mthree}{red!35}
\colorlet{Mfour}{red!8}

% ==========================================================
% TOP HALF: y >= 1/2
% Centered at y = 1/2  => use x - A(y-1/2) mod 1
% M1: >= 1/2 mod 1   (dark blue)
% M2: <  1/2 mod 1   (light blue)
% ==========================================================
\begin{scope}
  \clip (0,0.5) rectangle (1,1);
  \fill[Mtwo] (0,0.5) rectangle (1,1); % base = M2

  \foreach \k in {-60,...,60} {
    \pgfmathsetmacro{\ca}{\k + 0.5}
    \pgfmathsetmacro{\cb}{\k + 1.0}
    % x - A(y-1/2) in [k+1/2, k+1)
    \fill[Mone]
      ({\A*(-1.5) + \ca},-1) --
      ({\A*( 1.5) + \ca}, 2) --
      ({\A*( 1.5) + \cb}, 2) --
      ({\A*(-1.5) + \cb},-1) -- cycle;
  }
\end{scope}

% ==========================================================
% BOTTOM HALF: y < 1/2
% Centered at y = 1/2  => use x + A(y-1/2) mod 1
% M3: >= 1/2 mod 1   (dark red)
% M4: <  1/2 mod 1   (light red)
% ==========================================================
\begin{scope}
  \clip (0,0) rectangle (1,0.5);
  \fill[Mfour] (0,0) rectangle (1,0.5); % base = M4

  \foreach \k in {-60,...,60} {
    \pgfmathsetmacro{\ca}{\k + 0.5}
    \pgfmathsetmacro{\cb}{\k + 1.0}
    % x + A(y-1/2) in [k+1/2, k+1)
    \fill[Mthree]
      ({-\A*(-1.5) + \ca},-1) --
      ({-\A*( 1.5) + \ca}, 2) --
      ({-\A*( 1.5) + \cb}, 2) --
      ({-\A*(-1.5) + \cb},-1) -- cycle;
  }
\end{scope}

% ==========================================================
% Curve W (almost vertical, length < 1), split into subcurves
% ==========================================================
\coordinate (Wstart) at (0.66,0.08);
\coordinate (Wend)   at (0.72,0.94); % almost vertical, length < 1

% Draw W in black first
\draw[line width=1.15pt, black] (Wstart) -- (Wend);

% --- Top half coloring of W (base color = M2) ---
\begin{scope}
  \clip (0,0.5) rectangle (1,1);
  \draw[line width=0.9pt, Mtwo] (Wstart) -- (Wend);
\end{scope}

% Overlay top-half stripes (M1) on W
\begin{scope}
  \clip (0,0.5) rectangle (1,1);
  \foreach \k in {-60,...,60} {
    \pgfmathsetmacro{\ca}{\k + 0.5}
    \pgfmathsetmacro{\cb}{\k + 1.0}
    \begin{scope}
      \clip
        ({\A*(-1.5) + \ca},-1) --
        ({\A*( 1.5) + \ca}, 2) --
        ({\A*( 1.5) + \cb}, 2) --
        ({\A*(-1.5) + \cb},-1) -- cycle;
      \draw[line width=0.9pt, Mone] (Wstart) -- (Wend);
    \end{scope}
  }
\end{scope}

% --- Bottom half coloring of W (base color = M4) ---
\begin{scope}
  \clip (0,0) rectangle (1,0.5);
  \draw[line width=0.9pt, Mfour] (Wstart) -- (Wend);
\end{scope}

% Overlay bottom-half stripes (M3) on W
\begin{scope}
  \clip (0,0) rectangle (1,0.5);
  \foreach \k in {-60,...,60} {
    \pgfmathsetmacro{\ca}{\k + 0.5}
    \pgfmathsetmacro{\cb}{\k + 1.0}
    \begin{scope}
      \clip
        ({-\A*(-1.5) + \ca},-1) --
        ({-\A*( 1.5) + \ca}, 2) --
        ({-\A*( 1.5) + \cb}, 2) --
        ({-\A*(-1.5) + \cb},-1) -- cycle;
      \draw[line width=0.9pt, Mthree] (Wstart) -- (Wend);
    \end{scope}
  }
\end{scope}

% Re-draw a thin black centerline for crispness
\draw[line width=0.22pt, black] (Wstart) -- (Wend);

% ==========================================================
% Highlight ONE subcurve W_i (chosen in the top half, one blue stripe)
% ==========================================================
\begin{scope}
  \clip (0,0.5) rectangle (1,1);
  % choose one stripe index k=-2 (for A=16 this intersects W in one short segment)
  \pgfmathsetmacro{\ca}{-2 + 0.5}
  \pgfmathsetmacro{\cb}{-2 + 1.0}
  \begin{scope}
    \clip
      ({\A*(-1.5) + \ca},-1) --
      ({\A*( 1.5) + \ca}, 2) --
      ({\A*( 1.5) + \cb}, 2) --
      ({\A*(-1.5) + \cb},-1) -- cycle;
    \draw[line width=2.2pt, orange!90!black] (Wstart) -- (Wend);
  \end{scope}
\end{scope}

% Thin black line on top of the highlighted piece
\begin{scope}
  \clip (0,0.5) rectangle (1,1);
  \pgfmathsetmacro{\ca}{-2 + 0.5}
  \pgfmathsetmacro{\cb}{-2 + 1.0}
  \begin{scope}
    \clip
      ({\A*(-1.5) + \ca},-1) --
      ({\A*( 1.5) + \ca}, 2) --
      ({\A*( 1.5) + \cb}, 2) --
      ({\A*(-1.5) + \cb},-1) -- cycle;
    \draw[line width=0.25pt, black] (Wstart) -- (Wend);
  \end{scope}
\end{scope}

% Labels for W and highlighted W_i
\node[fill=white, inner sep=1pt] at (0.74,0.90) {$W$};
\node[fill=white, inner sep=1pt, font=\scriptsize] at (0.79,0.62) {$W_i$};

% Main square and split line y = 1/2
\draw[thick] (0,0) rectangle (1,1);
\draw[thick,dashed] (0,0.5) -- (1,0.5);

% Axis labels
\node[below] at (0.5,0) {$x$};

% Mark y=1/2
\node[left] at (0,0.5) {\small $y=\frac12$};

% Corner labels
\node[below left]  at (0,0) {\scriptsize $0$};
\node[below right] at (1,0) {\scriptsize $1$};
\node[above left]  at (0,1) {\scriptsize $1$};

% Legend
\begin{scope}[shift={(1.08,0.83)}, scale=0.9]
  \fill[Mone] (0,0) rectangle (0.08,0.08);
  \draw (0,0) rectangle (0.08,0.08);
  \node[right] at (0.10,0.04) {$\XX_\alpha(\mathcal M_1)$};
\end{scope}

\begin{scope}[shift={(1.08,0.70)}, scale=0.9]
  \fill[Mtwo] (0,0) rectangle (0.08,0.08);
  \draw (0,0) rectangle (0.08,0.08);
  \node[right] at (0.10,0.04) {$\XX_\alpha(\mathcal M_2)$};
\end{scope}

\begin{scope}[shift={(1.08,0.57)}, scale=0.9]
  \fill[Mthree] (0,0) rectangle (0.08,0.08);
  \draw (0,0) rectangle (0.08,0.08);
  \node[right] at (0.10,0.04) {$\XX_\alpha(\mathcal M_3)$};
\end{scope}

\begin{scope}[shift={(1.08,0.44)}, scale=0.9]
  \fill[Mfour] (0,0) rectangle (0.08,0.08);
  \draw (0,0) rectangle (0.08,0.08);
  \node[right] at (0.10,0.04) {$\XX_\alpha(\mathcal M_4)$};
\end{scope}

\end{tikzpicture}
\caption{The sets $\XX_\alpha(\mathcal M_\ell)$ in $[0,1]^2$ with $\alpha=16$, with an admissible curve $W \in \Sigma$ and one highlighted subcurve $W_i \subset W\cap \overline{\XX_\alpha(\mathcal M_1)}$.}
\label{figApp:1}
\end{figure}

\medskip

    We now prove \eqref{enum:1}. We begin by picking a fundamental domain in $x$ so that both curves are entirely contained in this domain. This is always possible due to the length restrictions on the curves, as well as the fact that they are ``almost vertical''. Without loss of generality, we will identify this fundamental domain with $[0,1]$.
    Next, we parameterise $W_1, W_2$ as two curves $\gamma_1: [0, S_1] \to \TT^2$ and $\gamma_2: [0, S_2] \to \TT^2$ with $S_1,S_2 \leq 1$  denoted by 
    \begin{align} \label{eq:parameterisation-gamma-i}
        \gamma_1 (t) = (x_1, y_1) + t(a_1, b_1) \,, \quad \gamma_2 (t) = (x_2, y_2) + t(a_2, b_2) \,.
    \end{align}
    We first observe that, once we choose the unmatched subcurves $U$ so that their total length is bounded by
$C\,\dd_\Sigma(W_1,W_2)$, the base points of the remaining pieces $W_1\setminus U$ and $W_2$ (and vice versa)
stay close. More precisely, the distance between the corresponding base points is controlled by the triangle inequality by the distance
between the original base points plus the length of the unmatched parts, which is in turn bounded by
$(C+1) \dd_\Sigma(W_1,W_2)$. Therefore, it suffices to estimate the length of the unmatched pieces.

    \textbf{Step 1: Initial pre-processing--Unmatched curves for base points and end points.}
    As a first step, we shall ``cut'' pieces from the beginnings and ends of $W_1,W_2$ in order to obtain two curves that are more amenable for matching. In particular, we shall endeavour to ensure that both curves begin  in the same half plane $(\{y \geq 1/2\}$ or $\{y <1/2\}$) and  terminate in the same half plane $(\{y \geq 1/2\}$ or $\{y <1/2\}$), and furthermore begin and terminate on the same strip $\{x \pm \alpha y=q\}$.
    
    More formally, we claim that for any $W_1, W_2 \in \Sigma$ we can define unmatched curves $U_{1,1}, U_{1,2}$, $U_{2,1}, U_{2,2}$,
    $$ W_1 \setminus (U_{1,1} \cup U_{1,2}) \,,  W_2 \setminus (U_{2,1} \cup U_{2,2}) \in \Sigma$$ 
    and the base points of $W_1 \setminus (U_{1,1} \cup U_{1,2}) \,,  W_2 \setminus (U_{2,1} \cup U_{2,2})$
    that we call $(\bar{x}_1, \bar{y}_1), (\bar{x}_2, \bar{y}_2) $ satisfy
    $$ \bar{y}_1, \bar{y}_2 \geq 1/2 \qquad \text{or} \qquad \bar{y}_1, \bar{y}_2 < 1/2$$
    and if $\bar{y}_1 \geq 1/2$ then there exists $q \in \RR$ such that 
    $$ (\bar{x}_1, \bar{y}_1), (\bar{x}_2, \bar{y}_2) \in \{x- \alpha y = q  \} $$
    and if $\bar{y}_1 < 1/2$ then there exists $q \in \RR$ such that 
    $$ (\bar{x}_1, \bar{y}_1), (\bar{x}_2, \bar{y}_2) \in \{x+ \alpha y = q  \} \,.$$
    Furthermore, the same holds for the endpoints of $W_i \setminus \{U_{i,1} \cup U_{i,2})$. We begin by outlining the argument for the base points.

\begin{figure}[ht]
\centering
\begin{tikzpicture}[scale=6,line cap=round]

% Large alpha (example)
\def\A{16}

% Colors
\colorlet{Mone}{blue!35}
\colorlet{Mtwo}{blue!8}
\colorlet{Mthree}{red!35}
\colorlet{Mfour}{red!8}

% ==========================================================
% Background partition: \XX_\alpha(\mathcal M_\ell)
% (centered at y = 1/2)
% ==========================================================

% Top half: y >= 1/2, split by x - A(y-1/2) mod 1
\begin{scope}
  \clip (0,0.5) rectangle (1,1);
  \fill[Mtwo] (0,0.5) rectangle (1,1); % base = M2

  \foreach \k in {-60,...,60} {
    \pgfmathsetmacro{\ca}{\k + 0.5}
    \pgfmathsetmacro{\cb}{\k + 1.0}
    \fill[Mone]
      ({\A*(-1.5) + \ca},-1) --
      ({\A*( 1.5) + \ca}, 2) --
      ({\A*( 1.5) + \cb}, 2) --
      ({\A*(-1.5) + \cb},-1) -- cycle;
  }
\end{scope}

% Bottom half: y < 1/2, split by x + A(y-1/2) mod 1
\begin{scope}
  \clip (0,0) rectangle (1,0.5);
  \fill[Mfour] (0,0) rectangle (1,0.5); % base = M4

  \foreach \k in {-60,...,60} {
    \pgfmathsetmacro{\ca}{\k + 0.5}
    \pgfmathsetmacro{\cb}{\k + 1.0}
    \fill[Mthree]
      ({-\A*(-1.5) + \ca},-1) --
      ({-\A*( 1.5) + \ca}, 2) --
      ({-\A*( 1.5) + \cb}, 2) --
      ({-\A*(-1.5) + \cb},-1) -- cycle;
  }
\end{scope}

% ==========================================================
% Same-strip guide lines (dashed)
% ==========================================================

% Bottom strip: x + A(y-1/2) = q_-
\pgfmathsetmacro{\qb}{0.58 + \A*(0.17-0.5)}
\pgfmathsetmacro{\yLbot}{0.5 + (\qb-1)/\A} % x=1
\pgfmathsetmacro{\yRbot}{0.5 + \qb/\A}     % x=0
\begin{scope}
  \clip (0,0) rectangle (1,0.5);
  \draw[densely dashed, black, line width=0.4pt] (1,\yLbot) -- (0,\yRbot);
\end{scope}

% Top strip: x - A(y-1/2) = q_+
\pgfmathsetmacro{\qt}{0.44 - \A*(0.86-0.5)}
\pgfmathsetmacro{\yLtop}{0.5 - \qt/\A}       % x=0
\pgfmathsetmacro{\yRtop}{0.5 + (1-\qt)/\A}   % x=1
\begin{scope}
  \clip (0,0.5) rectangle (1,1);
  \draw[densely dashed, black, line width=0.4pt] (0,\yLtop) -- (1,\yRtop);
\end{scope}

% ==========================================================
% Two segments:
%   - Left segment  = W_2 (green top part = U_{2,2})
%   - Right segment = W_1 (red bottom part = U_{1,1})
% ==========================================================

% Line equations for dashed lines:
\pgfmathsetmacro{\aBot}{\yRbot}
\pgfmathsetmacro{\bBot}{\yLbot-\yRbot}
\pgfmathsetmacro{\aTop}{\yLtop}
\pgfmathsetmacro{\bTop}{\yRtop-\yLtop}

% --- Left segment (W_2) ---
\pgfmathsetmacro{\xLoneS}{0.22}
\pgfmathsetmacro{\xLoneE}{0.38}
\pgfmathsetmacro{\yLoneS}{\aBot + \bBot*\xLoneS}                 % ON lower dashed
\pgfmathsetmacro{\yLoneE}{\aTop + \bTop*\xLoneE + 0.06}          % ABOVE upper dashed (longer)

% Intersection with upper dashed
\pgfmathsetmacro{\dxL}{\xLoneE-\xLoneS}
\pgfmathsetmacro{\dyL}{\yLoneE-\yLoneS}
\pgfmathsetmacro{\tUpL}{(\aTop+\bTop*\xLoneS-\yLoneS)/(\dyL-\bTop*\dxL)}
\pgfmathsetmacro{\xUpL}{\xLoneS+\tUpL*\dxL}
\pgfmathsetmacro{\yUpL}{\yLoneS+\tUpL*\dyL}

% draw full left segment
\draw[line width=1.1pt, black] (\xLoneS,\yLoneS) -- (\xLoneE,\yLoneE);

% green top part = U_{2,2}
\draw[line width=2.0pt, green!60!black] (\xUpL,\yUpL) -- (\xLoneE,\yLoneE);
\draw[line width=0.25pt, black] (\xUpL,\yUpL) -- (\xLoneE,\yLoneE);
\node[fill=white, inner sep=1pt] at (0.33,0.63) {$W_2$};

% label U_{2,2} to the RIGHT of the green part
\node[fill=white, inner sep=1pt, font=\scriptsize, text=green!60!black, anchor=west]
  at (\xLoneE+0.02, {0.5*(\yUpL+\yLoneE)}) {$U_{2,2}$};

% --- Right segment (W_1) ---
\pgfmathsetmacro{\xRtwoS}{0.78}
\pgfmathsetmacro{\xRtwoE}{0.56}
\pgfmathsetmacro{\yRtwoS}{\aBot + \bBot*\xRtwoS - 0.06}          % BELOW lower dashed (longer)
\pgfmathsetmacro{\yRtwoE}{\aTop + \bTop*\xRtwoE}                 % ON upper dashed

% Intersection with lower dashed
\pgfmathsetmacro{\dxR}{\xRtwoE-\xRtwoS}
\pgfmathsetmacro{\dyR}{\yRtwoE-\yRtwoS}
\pgfmathsetmacro{\tLowR}{(\aBot+\bBot*\xRtwoS-\yRtwoS)/(\dyR-\bBot*\dxR)}
\pgfmathsetmacro{\xLowR}{\xRtwoS+\tLowR*\dxR}
\pgfmathsetmacro{\yLowR}{\yRtwoS+\tLowR*\dyR}

% draw full right segment
\draw[line width=1.1pt, black] (\xRtwoS,\yRtwoS) -- (\xRtwoE,\yRtwoE);

% red bottom part = U_{1,1}
\draw[line width=2.0pt, red!75!black] (\xRtwoS,\yRtwoS) -- (\xLowR,\yLowR);
\draw[line width=0.25pt, black] (\xRtwoS,\yRtwoS) -- (\xLowR,\yLowR);
\node[fill=white, inner sep=1pt] at (0.67,0.63) {$W_1$};

% label U_{1,1} to the RIGHT of the red part
\node[fill=white, inner sep=1pt, font=\scriptsize, text=red!75!black, anchor=west]
  at (\xRtwoS+0.02, {0.5*(\yRtwoS+\yLowR)}) {$U_{1,1}$};

% ==========================================================
% Frame and labels
% ==========================================================
\draw[thick] (0,0) rectangle (1,1);
\draw[thick,dashed] (0,0.5) -- (1,0.5);

\node[below] at (0.5,0) {$x$};
\node[left] at (0,0.5) {\small $y=\frac12$};

\node[below left]  at (0,0) {\scriptsize $0$};
\node[below right] at (1,0) {\scriptsize $1$};
\node[above left]  at (0,1) {\scriptsize $1$};

% Partition legend
\begin{scope}[shift={(1.08,0.83)}, scale=0.9]
  \fill[Mone] (0,0) rectangle (0.08,0.08);
  \draw (0,0) rectangle (0.08,0.08);
  \node[right] at (0.10,0.04) {$\XX_\alpha(\mathcal M_1)$};
\end{scope}

\begin{scope}[shift={(1.08,0.70)}, scale=0.9]
  \fill[Mtwo] (0,0) rectangle (0.08,0.08);
  \draw (0,0) rectangle (0.08,0.08);
  \node[right] at (0.10,0.04) {$\XX_\alpha(\mathcal M_2)$};
\end{scope}

\begin{scope}[shift={(1.08,0.57)}, scale=0.9]
  \fill[Mthree] (0,0) rectangle (0.08,0.08);
  \draw (0,0) rectangle (0.08,0.08);
  \node[right] at (0.10,0.04) {$\XX_\alpha(\mathcal M_3)$};
\end{scope}

\begin{scope}[shift={(1.08,0.44)}, scale=0.9]
  \fill[Mfour] (0,0) rectangle (0.08,0.08);
  \draw (0,0) rectangle (0.08,0.08);
  \node[right] at (0.10,0.04) {$\XX_\alpha(\mathcal M_4)$};
\end{scope}

\end{tikzpicture}
\caption{Pre-processing of two admissible curves $W_1,W_2$: a small initial curve $U_{1,1}$ is cut from $W_1$ and a small final curve $U_{2,2}$ is cut from $W_2$. The base points of $W_1 \setminus U_{1,1}$ and $W_2 \setminus U_{2,2}$ lie in the same half-plane and on the same $\{ x+ \alpha y = q_{in}\}$ for some $q_{in} \in \RR$ (indicated by the dashed strip lines); likewise for the endpoints of $W_1 \setminus U_{1,1}$ and $W_{2} \setminus U_{2,2}$  lie in the same half-plane and on the same $\{ x- \alpha y = q_{fin}\}$ for some $q_{fin} \in \RR$.}
\label{figApp:2}
\end{figure}
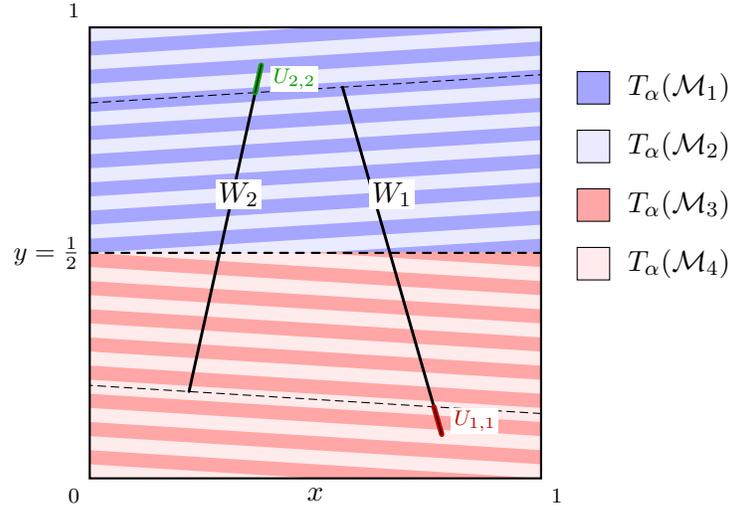

    \textbf{Base points.} Recall the parameterisations $\gamma_i$ for $i=1,2$ \eqref{eq:parameterisation-gamma-i}. There are now two options: either $\dd_{\mathbb{T}}(y_1,1/2)\leq \dd_{\mathbb{T}}(y_1,1)$, or $\dd_{\mathbb{T}}(y_1,1)\leq \dd_{\mathbb{T}}(y_1,1/2)$. Assume that $\dd_{\mathbb{T}}(y_1,1/2)\leq \dd_{\mathbb{T}}(y_1,1)$. The other case is handled in exactly the same way by symmetry. In particular, the conditions of $\dd_{\mathbb{T}}(y_1,1/2)\leq \dd_{\mathbb{T}}(y_1,1)$, together with $\dd_\Sigma(W_1,W_2)\leq 2/\alpha$ imply that the euclidean distance between the base points viewed as elements of $[0,1)^2$ is the same as the toroidal distance.  We now have two cases for the base points of the curves $\gamma_1, \gamma_2$, which are viewed as elements of $[0,1)^2$:
    \begin{itemize}
        \item[1)]  $y_1,y_2\geq 1/2$ or $y_1,y_2<1/2$.
        \item[2)]  ($y_1<1/2$ and  $y_2 \geq 1/2$) or ($y_1\geq 1/2$ and  $y_2 < 1/2)$.
    \end{itemize}
    Assume that we are in the second case, for instance $y_1<1/2$, $y_2 \geq 1/2$. Then we define the unmatched curve
    $$ U_{1,1} = \{\gamma_1 (t): t \in [0, t_1 \wedge S_1] \}$$
    where $t_1$ is chosen so that  the second component of $\gamma_1(t_1)$ is equal to $1/2$ and we  get $t_1 = \frac{1/2 - y_1}{b_1} \leq \frac{y_2 - y_1}{b_1} \leq 2 \dd_\Sigma (W_1, W_2)$. 
    
    Therefore, we deal with case 1), assuming for now that $y_1,y_2<1/2$. The other case is identical. 
    We can assume without loss of generality that $x_1 +\alpha y_1 \leq  x_2 + \alpha y_2$. 
    We define $q_2 = x_2 + \alpha y_2 \in \RR$ and we define the unmatched curve 
    $$ U_{1,1} = \{\gamma_1 (t): t \in [0, t_1 \wedge S_1] \} $$
    where $t_1$ is now defined so that $\gamma_1(t_1) \in \{x+ \alpha y = q_2\,, \, \, \, x,y \in [0,1) \}$. 
    Hence 
    $$0 \leq  t_1 = \frac{x_2 +\alpha y_2 - (x_1 + \alpha y_1) }{a_1 +\alpha b_1} \leq 2 \dd_\Sigma (W_1, W_2)\,. 
    $$  
    If $y_1+t_1 b_1<1/2$, we are done. If $y_1+t_1 b_1>1/2$, we define another unmatched piece 
    $$
    U_{2,2}=\{\gamma_2(t): t \leq t_2 \wedge S_2\},
    $$
    where $t_2=\frac{1/2-y_2}{b_2}$. Now, 
    $$
    t_2 \leq \frac{y_1+t_1 b_1-y_2}{b_2} \leq C\dd_{\Sigma}(W_1,W_2).
    $$
Hence, no matter what case we are in, we may either conclude our cutting of the initial points, or reduce to the case where $y_1,y_2 \geq 1/2$, and $y_1, y_2$ are close to $1/2$. Hence, we assume that both curves start with $y_1,y_2>1/2$, with $y_1,y_2$ a distance less than $\frac{C}{\alpha}$ to $y=1/2$. Without loss of generality, we assume that $x_1 -\alpha y_1 \geq x_2-\alpha y_2=:\tilde q_2$, and define the unmatched piece 
$$
U_{1,1}=\{\gamma_1(t):t \in [0, t_1 \wedge S_1]\},
$$
where $t_1=\frac{\tilde q_2-(x_1-\alpha y_1)}{-b_1\alpha +a_1}.$ As before, we estimate $t_1 \leq C\dd_{\Sigma}(W_1,W_2)$, and since both $y_1,y_2$ are close to $1/2$ it still holds that $y_1+t_1 b_1<1$, so the cutting of base points is complete.

Finally we remark that, although in the previous computations we had always implicitly assumed that $t_1\wedge S_1=t_1$ and $t_2 \wedge S_2 = t_2$, in the case where for instance $S_1 \leq t_1$, we may actually view both $W_1, W_2$ entirely as unmatched. Indeed, we saw that $t_1 \leq C\dd_{\Sigma}(W_1,W_2)$. Thus, $S_1 \leq C \dd_\Sigma(W_1,W_2)$. Furthermore, $S_2 \leq |S_1-S_2|+S_2 \leq (C+1)\dd_\Sigma(W_1,W_2)$, and thus they are indeed unmatched.
\subsubsection*{End points.} We now undertake precisely the same cutting procedure that we did for the base points in reverse: we consider the curves $\tilde \gamma_i(t)=\{(x_i+S_i a_i,y_i+S_i b_i)+t(-a_i,-b_i): t \leq S_i-t_i\}$, where the $t_i$ are the new starting points of $W_i\setminus U_{i,1}$. Note that the distance between these ``flipped'' curves is bounded by the distance between the non-flipped ones. Indeed, the only thing to check is the distance between the base points, which is bounded by 
$$
|x_1-x_2|+|y_1-y_2|+|S_1a_1-S_2 a_2|+|S_1 b_1-S_2 b_2|\leq C\dd_\Sigma(W_1,W_2).
$$
As such, we now obtain curves $U_{1,2}$, $U_{2,2}$ with $|U_{i,2}|\leq C\dd_{\Sigma}(W_1,W_2)$ so that the end points $\tilde x_i, \tilde y_i$ of $W_{i} \setminus U_{i,2}\cup U_{i,1}$ satisfy either 
$$ \tilde{y}_1, \tilde{y}_2 \geq 1/2 \qquad \text{or} \qquad \tilde{y}_1, \tilde{y}_2 < 1/2$$
    and if $\tilde{y}_1 \geq 1/2$ then there exists $q \in \RR$ such that 
    $$ (\tilde{x}_1, \tilde{y}_1), (\tilde{x}_2, \tilde{y}_2) \in \{x- \alpha y = q  \} $$
    and if $\tilde{y}_1 < 1/2$ then there exists $q \in \RR$ such that 
    $$ (\tilde{x}_1, \tilde{y}_1), (\tilde{x}_2, \tilde{y}_2) \in \{x+\alpha y = q  \} \,.$$

    \textbf{Step 2: Matched curves.} 
    We first define the matched curves. Without loss of generality and up to discarding some unmatched curves defined above,  we suppose that the base points $(x_1, y_1), (x_2,y_2)$ are such that $y_1, y_2 \in [0,1/2)$ and $(x_1, y_1), (x_2,y_2) \in \{ x+ \alpha y = q_{in} \}$ for some $q_{in} \in [\frac{j_{in}}{2}, \frac{j_{in}+1}{2})$, where $j_{in} \in \ZZ$. Without loss of generality we can suppose that $\gamma_1(S_1),\gamma_2 (S_2)$ are such that the second component is in $[1/2, 1)$ and  $\gamma_1(S_1),\gamma_2 (S_2) \in \{ x-\alpha y = q_{fin} \}$ for some $q_{fin} \in (\frac{j_{fin}}{2}, \frac{j_{fin}+1}{2}]$, where $j_{fin} \in \ZZ$. We define accordingly $t_{1 , j_{in}} = t_{2, j_{in}}=0 $ and $t_{1, j_{fin}} = S_1$ and $t_{2, j_{fin}}= S_2$.
    For  $ j_{in} \leq j$ we also define  $t_{1,j} \geq t_{1, j_{in}} =0$ as the minimal time so that $\gamma_1(t_{1,j}) \in \{ x+ \alpha y = \frac{j}{2} \}$ and similarly $t_{2, j} >0$ as the minimal time so that $\gamma_2(t_{2,j}) \in \{ x+\alpha y = \frac{j}{2} \}$.  Hence
    $$ t_{1,j}  = \frac{\frac{j}{2} - (x_1 + \alpha y_1)}{a_1 + \alpha b_1}\,,  \qquad t_{2,j}  = \frac{\frac{j}{2} - (x_2+ \alpha y_2)}{a_2 + \alpha b_2}$$
    $$ |t_{1,j} - t_{2,j}| = \left | \frac{(\frac{j}{2} - (x_1 + \alpha y_1)) (a_2 +\alpha b_2 ) - (\frac{j}{2} - (x_2 + \alpha y_2)) (a_1 +\alpha b_1 ) }{(a_2 +\alpha b_2 )(a_1 + \alpha b_1)} \right | \leq 2 \dd_\Sigma (W_1, W_2)\,,$$
   where in the last we have used that $x_2 +\alpha y_2 = x_1 +\alpha y_1$ and $|j| \leq C \alpha$. Hence, for all $j \in \ZZ$ so that $\gamma_{1} (t_{1,j})$ and $\gamma_{2} (t_{2,j})$ has the second component less then $1/2$, we  define the matched curves  as 
    \begin{align}\label{eq:matched-curves}
        W_{1,j} = \{ \gamma_1 (t) : t \in [t_{1,j+1}, t_{1,j} ] \} \,, \qquad  W_{2,j} = \{ \gamma_2 (t) : t \in [t_{2,j+1}, t_{2,j}] \}  \,.
    \end{align}
    Defining $ (x_{1,j}, y_{1,j})$ the intersection of $W_1$ with $\{ x+\alpha y = \frac{j}{2}\}$ and defining similarly $(x_{2,j}, y_{2,j})$ we deduce that 
    $$ |y_{2,j} - y_{1,j}| = |y_{2} +  t_{2,j} b_2 - (y_{1} + t_{1,j} b_1)| \leq   |y_{2} - y_1 | + |t_{1,j} b_1 - t_{2,j} b_2| \leq 3 \dd_\Sigma (W_1, W_2) \,.$$
     The same argument holds true for the $x$-variable. Hence, these matched curves satisfy all the properties.
    The above argument may be continued until the second component of either curve reaches $1/2$. We outline here the case of crossing $y=1/2$.   We denote by $\bar{j} \in \ZZ$ the largest number so that the second component of $\gamma_{1} (t_{1, \bar j})$ and $\gamma_{2} (t_{2, \bar j})$ are less then $1/2$. Hence, at least one of the two curves cross the set $\{ y =1/2\}$ before crossing $\{x + \alpha y = \frac{\bar j +1}{2} \}$. Since both curves terminate in $\{y \geq 1/2\}$ by Step 1, there exists  $\bar t_2 \in (t_{2, \bar j}, t_{2, \bar j +1}]$ such that the second component of  $\gamma_{2} (\bar t_{2})$ is $1/2$ and $\bar t_1 \in (t_{1, \bar j}, t_{1, \bar j +1}]$ such that the second component of  $\gamma_{1} (\bar t_{1})$ is $1/2$. Suppose that $x_1+\bar t_1a_1+\alpha/2 \geq x_2+\bar t_2a_2+\alpha/2$, see Figure \ref{figApp:3_zoom}. 
    We now define  $q_{2,up}, q_{2, down} \in \RR$ such that 
    $$\gamma_{2} (\bar t_2) \in \{ x+ \alpha y = q_{2, down}\} \cap \{ x-\alpha y = q_{2, up}\} \,.$$

\begin{figure}[ht]
\centering
\begin{tikzpicture}[x=8cm,y=35cm] % zoomed scales (x reduced to avoid TeX dimension overflow)

% Large alpha (example)
\def\A{16}

% Zoom window around y = 1/2
\def\ymin{0.455}
\def\ymax{0.545}

% Colors
\colorlet{Mone}{blue!35}
\colorlet{Mtwo}{blue!8}
\colorlet{Mthree}{red!35}
\colorlet{Mfour}{red!8}

% ==========================================================
% Geometry
%   Left broken curve is labeled W_2
%   Right straight curve is labeled W_1
%   We shift W_2 so that P = W_2 \cap {y=1/2} sits at the center
%   of the T_\alpha(M_2) region on y=1/2, i.e. x = 1/4.
% ==========================================================

% Parameter for the point on segment B1--Mmid where y=1/2
\pgfmathsetmacro{\tPm}{(0.5-0.17)/(0.52-0.17)}

% Base x-coordinate of P (before shifting W_2)
\pgfmathsetmacro{\xPmBase}{0.58 + (0.52-0.58)*\tPm}

% Target x-coordinate: middle of the light-blue region T_\alpha(M_2) at y=1/2
% (on [0,1], the T_\alpha(M_2) portion is [0,1/2], so the center is x=1/4)
\pgfmathsetmacro{\xTargetPm}{0.25}

% Shift needed for W_2
\pgfmathsetmacro{\shiftWtwo}{\xTargetPm - \xPmBase}

% Left broken curve (W_2), shifted left
\pgfmathsetmacro{\xBOne}{0.58+\shiftWtwo}
\pgfmathsetmacro{\xMmid}{0.52+\shiftWtwo}
\pgfmathsetmacro{\xEOne}{0.44+\shiftWtwo}

\coordinate (B1)   at (\xBOne,0.17);
\coordinate (Mmid) at (\xMmid,0.52);
\coordinate (E1)   at (\xEOne,0.86);

% Right straight curve (W_1)
\coordinate (B2) at (0.72,0.16125);
\coordinate (E2) at (0.8991,0.8887);

% ----------------------------------------------------------
% P = W_2 \cap {y=1/2}   (W_2 is the LEFT broken curve)
% ----------------------------------------------------------
\pgfmathsetmacro{\xPm}{\xBOne + (\xMmid-\xBOne)*\tPm}
\coordinate (Pm) at (\xPm,0.5);

% Dashed lines through Pm:
%   x + A y = q_{2,down}
%   x - A y = q_{2,up}
\pgfmathsetmacro{\qtwodown}{\xPm + 0.5*\A}
\pgfmathsetmacro{\qtwoup}{\xPm - 0.5*\A}

% Useful equivalent constant in centered coordinates:
%   x + A(y-1/2) = xPm,  x - A(y-1/2) = xPm
\pgfmathsetmacro{\qPm}{\xPm}

% Endpoints on x=1 (draw only portions to the right of Pm)
% x - A y = q_{2,up}  (slope +1/A)
\pgfmathsetmacro{\yRpmPlus}{(1-\qtwoup)/\A}
% x + A y = q_{2,down} (slope -1/A)
\pgfmathsetmacro{\yRpmMinus}{(\qtwodown-1)/\A}

% ----------------------------------------------------------
% U_{1,1} = segment on W_1 (RIGHT curve) between intersections
% with the two dashed lines
% ----------------------------------------------------------
\pgfmathsetmacro{\dxWone}{0.8991-0.72}
\pgfmathsetmacro{\dyWone}{0.8887-0.16125}

% Intersection with x + A(y-1/2) = qPm  (equiv. x + A y = q_{2,down})
\pgfmathsetmacro{\tUoneoneA}{
  (\qPm - (0.72 + \A*(0.16125-0.5))) / (\dxWone + \A*\dyWone)
}
\coordinate (U11A) at ({0.72 + \tUoneoneA*\dxWone},{0.16125 + \tUoneoneA*\dyWone});

% Intersection with x - A(y-1/2) = qPm  (equiv. x - A y = q_{2,up})
\pgfmathsetmacro{\tUoneoneB}{
  (\qPm - (0.72 - \A*(0.16125-0.5))) / (\dxWone - \A*\dyWone)
}
\coordinate (U11B) at ({0.72 + \tUoneoneB*\dxWone},{0.16125 + \tUoneoneB*\dyWone});

% ==========================================================
% Draw everything clipped to the zoom window
% ==========================================================
\begin{scope}
  \clip (0,\ymin) rectangle (1,\ymax);

  % Background partition (same as original, but clipped)

  % Top half: y >= 1/2
  \begin{scope}
    \clip (0,0.5) rectangle (1,\ymax);
    \fill[Mtwo] (0,0.5) rectangle (1,\ymax);

    \foreach \k in {-6,...,6} {
      \pgfmathsetmacro{\ca}{\k + 0.5}
      \pgfmathsetmacro{\cb}{\k + 1.0}
      \fill[Mone]
        ({\A*(-1.5) + \ca},-1) --
        ({\A*( 1.5) + \ca}, 2) --
        ({\A*( 1.5) + \cb}, 2) --
        ({\A*(-1.5) + \cb},-1) -- cycle;
    }
  \end{scope}

  % Bottom half: y < 1/2
  \begin{scope}
    \clip (0,\ymin) rectangle (1,0.5);
    \fill[Mfour] (0,\ymin) rectangle (1,0.5);

    \foreach \k in {-6,...,6} {
      \pgfmathsetmacro{\ca}{\k + 0.5}
      \pgfmathsetmacro{\cb}{\k + 1.0}
      \fill[Mthree]
        ({-\A*(-1.5) + \ca},-1) --
        ({-\A*( 1.5) + \ca}, 2) --
        ({-\A*( 1.5) + \cb}, 2) --
        ({-\A*(-1.5) + \cb},-1) -- cycle;
    }
  \end{scope}

  % Dashed construction lines starting at Pm and drawn only to the right
  \draw[dashed, black, line width=0.6pt] (Pm) -- (1,\yRpmPlus);  % x - A y = q_{2,up}
  \draw[dashed, black, line width=0.6pt] (Pm) -- (1,\yRpmMinus); % x + A y = q_{2,down}

  % Curves
  \draw[line width=1.1pt, black] (B1) -- (Mmid) -- (E1); % left curve (W_2)
  \draw[line width=1.1pt, black] (B2) -- (E2);           % right curve (W_1)

  % U_{1,1} highlighted in RED
  \draw[line width=2.3pt, red!85!black] (U11A) -- (U11B);

  % Markers
  \fill (Pm)   circle (0.45pt);
  \fill (U11A) circle (0.50pt);
  \fill (U11B) circle (0.50pt);
\end{scope}

% ==========================================================
% Frame and labels for zoomed view
% ==========================================================
\draw[thick] (0,\ymin) rectangle (1,\ymax);
\draw[thick,dashed] (0,0.5) -- (1,0.5);

\node[left] at (0,0.5) {\small $y=\frac12$};
\node[below left]  at (0,\ymin) {\scriptsize $0$};
\node[below right] at (1,\ymin) {\scriptsize $1$};
\node[below] at (0.5,\ymin) {$x$};

% Labels
\node[fill=white, inner sep=1pt] at (0.18,0.515) {$W_2$};
\node[fill=white, inner sep=1pt] at (0.86,0.528) {$W_1$}; % moved upward
\node[fill=white, inner sep=1pt, font=\scriptsize, text=red!85!black] at (0.875,0.5015) {$U_{1,1}$}; % closer to red line

% Optional label for P
% \node[fill=white, inner sep=1pt, font=\scriptsize] at ({\xPm+0.015},0.503) {$P$};

% ==========================================================
% Larger, more spaced legend for T_\alpha(\mathcal M_\ell) (no box)
% ==========================================================
\begin{scope}[shift={({1.12},{0.543})}, x=1cm, y=1cm]

  % entry 1
  \fill[Mone]  (0, 0.00) rectangle (0.18,-0.18);
  \draw        (0, 0.00) rectangle (0.18,-0.18);
  \node[anchor=west, font=\footnotesize] at (0.26,-0.09) {$T_\alpha(\mathcal M_1)$};

  % entry 2
  \fill[Mtwo]  (0,-0.36) rectangle (0.18,-0.54);
  \draw        (0,-0.36) rectangle (0.18,-0.54);
  \node[anchor=west, font=\footnotesize] at (0.26,-0.45) {$T_\alpha(\mathcal M_2)$};

  % entry 3
  \fill[Mthree](0,-0.72) rectangle (0.18,-0.90);
  \draw        (0,-0.72) rectangle (0.18,-0.90);
  \node[anchor=west, font=\footnotesize] at (0.26,-0.81) {$T_\alpha(\mathcal M_3)$};

  % entry 4
  \fill[Mfour] (0,-1.08) rectangle (0.18,-1.26);
  \draw        (0,-1.08) rectangle (0.18,-1.26);
  \node[anchor=west, font=\footnotesize] at (0.26,-1.17) {$T_\alpha(\mathcal M_4)$};

\end{scope}

\end{tikzpicture}
\caption{We represent $W_1, W_2$ as two black lines. We zoom near $y=\frac12$.  The two dashed lines are $\{ x+ \alpha y = q_{2, down}\}$ and $\{ x-\alpha y = q_{2, up}\}$ passing through $W_2 \cap \{ y = 1/2 \}$. The red segment is the unmatched curve $U_{1,1}\subset W_1$.}
\label{figApp:3_zoom}
\end{figure}
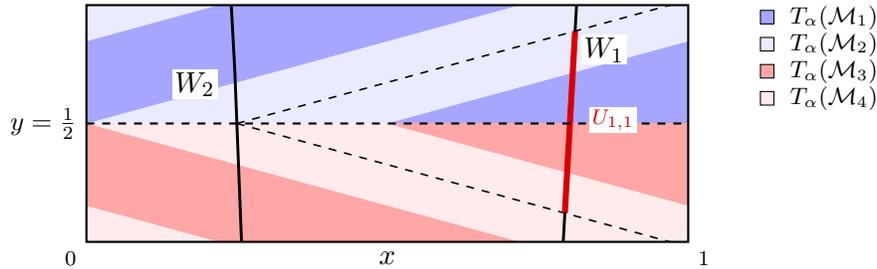

    Hence we define the unmatched curve
    $$ U_{1,1} = \{ \gamma_1(t): t \in [t_{1, down} \, , t_{1, up}] \}  \,,$$
    so that 
    $$ \gamma_1 (t_{1, down}) \in \{ x+\alpha y = q_{2, down}\} \,, \qquad  \gamma_1 (t_{1, up}) \in \{ x-\alpha y = q_{2, up}\} \,.$$
    By the assumption that $x_1+\bar t_1 a_1+\alpha/2 \geq x_2+\bar t_2 a_2 +\alpha/2$, it follows that $y_1+t_{1,up}b_1 \leq 1/2$, and so the curves $\{\gamma_2(t): t \in [t_{2,\bar j},\bar t_2]\}$, and $\{\gamma_1(t):t \in [t_{1,\bar j}, t_{1,down}\}$ are indeed matched, with distance bounded as before by $c\dd_{\Sigma}(W_1,W_2)$. Next, we have 
    \begin{align*}
        x_1 + t_{1, down} a_1 + \alpha (y_1 + t_{1, down} b_1) = q_{2, down} = x_2 + \bar{t}_2 a_2 + \alpha (y_2 + \bar t_2 b_2)\,,
        \\
        x_1 + t_{1, up} a_1 - \alpha (y_1 + t_{1, up} b_1) = q_{2, up} = x_2 + \bar{t}_2 a_2 - \alpha (y_2 + \bar t_2 b_2)\,,
    \end{align*}
    and summing the two we deduce 
    \begin{align*}
        |t_{1, up} - t_{1, down}| & = \left | \frac{2 (x_2 - x_1) + (\bar t_2 a_2 - t_{1, down} a_1) + (\bar t_2 a_2 - t_{1, up} a_1)}{\alpha b_1} \right | 
        \\
        & \leq C \alpha^{-1} ( \dd_\Sigma (W_1, W_2) +  |\bar t_2 a_2 - t_{1, up} a_1|)
    \end{align*}
    where in the last we used $|\bar t_2 a_2 - t_{1, down} a_1|=| \bar t_2 (a_2 - a_1) + a_1 (\bar t_2 - t_{1, down}) | \leq C \dd_\Sigma (W_1, W_2)$ since we have already an upper bound on $|\bar t_2 - t_{1, down} |$ by similar computations given above. Hence, using also that $|\bar t_2 a_2 - t_{1, up} a_1| \leq |(\bar t_2 - t_{1, down}) a_2 | + |(a_2 - a_1)t_{1, down}| + |(t_{1, down}- t_{1, up}) a_1| \leq C \dd_{\Sigma} (W_1, W_2) + C \alpha^{-1} |(t_{1, down}- t_{1, up})|$ we deduce that up to relabelling the constant $C>0$
    $$ |t_{1, up} - t_{1, down}|  \leq C \alpha^{-1} \dd_{\Sigma} (W_1, W_2)  + C \alpha^{-2} |(t_{1, down}- t_{1, up})|$$
    and since $\alpha \geq \alpha_0$ is sufficiently large we deduce (up to again relabelling the constant) $|t_{1, up} - t_{1, down}| \leq C \alpha^{-1} \dd_{\Sigma} (W_1,W_2)$.

    From the new base points given by $\gamma_1(t_{1,up})$ and $\gamma_2(\bar t_2)$ we create matched pieces iteratively similarly as in \eqref{eq:matched-curves} with iterative times so that the curves cross $\{x - \alpha y = \frac{j}{2} \}$ for $j \in \ZZ$ for $y \in [1/2, 1)$. We then may define also unmatched curves close to  $y =1$ similarly to the case $y =1/2$ discussed above.

\medskip

     We now prove \eqref{enum:2}. We suppose that we have two matched curves $W_{1,1}, W_{2,1}$  parametrized by $\gamma_1, \gamma_2$, more precisely
     $$\gamma_1 (t) = (x_1, y_1) + t (a_1, b_1)\,, \qquad \gamma_2 (t) = (x_2, y_2) + t (a_2, b_2)\,.$$
     By \eqref{enum:1} we have that $\gamma_1 ([t_{1, in}, t_{1,fin}]), \gamma_2 ([t_{2, in}, t_{2,fin}]) \subset \XX_\alpha (\cM_\ell)$ for some $\ell=1,2,3,4$. It is straightforward to see that $|t_{i,in} - t_{i,fin}|\leq \frac{2}{\alpha}$ for $i=1,2$. To simplify the notation and by the matching of the curves given in \eqref{enum:1}, we can always translate the parameterisation and suppose that $\gamma_1 (0), \gamma_2 (0) \in \{ x+ \alpha y = q\}$ and $\gamma_1 (t_1),\gamma_2 (t_2) \in \{ x+ \alpha y = \bar{q} \}$ with $\bar{q} \geq q$ and $|q- \bar{q}|\leq 1/2$ and $|t_1 - t_2|\leq C \alpha^{-1}\dd_\Sigma (W_{1,1}, W_{2,1})$. Indeed, $t_i=\frac{\bar q-q}{a_i+\alpha b_i}$, so $t_1-t_2=(\bar q-q)\frac{a_2-a_1+\alpha (b_2-b_1)}{(a_1+\alpha b_1)(a_2+\alpha b_2)} \leq C\alpha^{-1}\dd_\Sigma(W_{1,1},W_{2,1})$. We suppose that we are working on $\XX_\alpha (\cM_4)$, more precisely 
     $$ D \XX_\alpha^{-1} = \begin{pmatrix}
          1 & \alpha
         \\
         \alpha & 1+ \alpha^2 
     \end{pmatrix} \,.$$
    We compute the first component of $\XX_\alpha^{-1} (\gamma_1 (t))$
    $$ (\XX_\alpha^{-1} (\gamma_1 (t)))^{(1)} =x_1 + \alpha y_1 +  t( a_1 + \alpha b_1)    $$
    and the second component 
    $$ (\XX_\alpha^{-1} (\gamma_1 (t)))^{(2)} = \alpha x_2+  (1+ \alpha^2) y_2 + t  \left (\alpha a_2 +(\alpha^2 +1) b_2 \right ) \,. $$
    Hence,  we have 
    that the base point of $\XX_\alpha^{-1} (\gamma_1 (t))$ is $\XX_\alpha^{-1} (\gamma(0)) =  (x_1 + \alpha y_1, \alpha x_1+  (1+ \alpha^2) y_1)$. From which we deduce  
    $$X_1 = \XX_\alpha^{-1} (\gamma_1 (0)) = (q, y_1 + \alpha q) \,.$$
    Similarly we find the parameterisation of $\XX_\alpha^{-1} (\gamma_2 (t))$ and  its base point is 
    $$ X_2 = (q, y_2 + \alpha q) \,.$$
    We deduce that the second component of these base points satisfy 
    $$ |X_1^{(2)} -  X_2^{(2)}| \leq |y_1 - y_2| \leq \alpha^{-1} |x_1 - x_2| \leq C\alpha^{-1} \dd_\Sigma (W_{1,1}, W_{2,1})\,,$$
    where we have used that $x_i + \alpha y_i = q$ for $i=1,2$. We now suppose that $y_1 \leq y_2$ and we define the time $t^\star$ such that the second component of  
    $\XX_\alpha^{-1} (\gamma_1 (t))$ is the same as the second component of $\XX_\alpha^{-1} (\gamma_2 (0))$. Hence, using again that $|y_2 - y_1| \leq \alpha^{-1} |x_2 - x_1|\leq C\alpha^{-1} \dd_\Sigma (W_{1,1},W_{2,1})$ we deduce 
    $$ t^\star = \frac{y_2 - y_1}{(1+ \alpha^2) b_1 + \alpha a_1} \leq C \alpha^{-3} \dd(W_{1,1}, W_{2,1})\,.$$
    We define the unmatched curve 
    $$ \tilde U_1 = \{\XX_\alpha^{-1}(\gamma_1(t)) : t \in [0, t^*] \} \,. $$

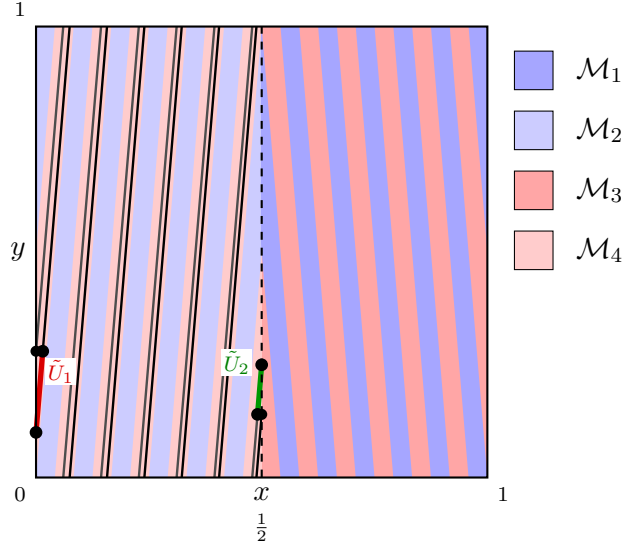
\begin{figure}[ht]
\centering
\begin{tikzpicture}[scale=6]

% Parameter alpha
\def\A{12}

% Colors (same palette style as before, but left side more visible)
\colorlet{Mone}{blue!35}   % M1
\colorlet{Mtwo}{blue!20}   % M2
\colorlet{Mthree}{red!35}  % M3
\colorlet{Mfour}{red!20}   % M4

% ==========================================================
% Partition of [0,1]^2 into M1,...,M4
% ==========================================================

% Left half: x < 1/2  (M4 base + M2 stripes)
\begin{scope}
  \clip (0,0) rectangle (0.5,1);

  % base region M4
  \fill[Mfour] (0,0) rectangle (0.5,1);

  % stripes for M2: y - A x in [k+1/2, k+1)
  \foreach \k in {-20,...,20} {
    \pgfmathsetmacro{\ca}{\k + 0.5}
    \pgfmathsetmacro{\cb}{\k + 1.0}
    \fill[Mtwo]
      (0,\ca) --
      (1,{\A + \ca}) --
      (1,{\A + \cb}) --
      (0,\cb) -- cycle;
  }
\end{scope}

% Right half: x >= 1/2  (M3 base + M1 stripes)
\begin{scope}
  \clip (0.5,0) rectangle (1,1);

  % base region M3
  \fill[Mthree] (0.5,0) rectangle (1,1);

  % stripes for M1: y + A x in [k+1/2, k+1)
  \foreach \k in {-20,...,40} {
    \pgfmathsetmacro{\ca}{\k + 0.5}
    \pgfmathsetmacro{\cb}{\k + 1.0}
    \fill[Mone]
      (0,\ca) --
      (1,{\ca - \A}) --
      (1,{\cb - \A}) --
      (0,\cb) -- cycle;
  }
\end{scope}

% ==========================================================
% Two wrapped curves in M4 (left half), almost parallel but not parallel
% They satisfy y = frac(s_i x + c_i), x in [0,1/2], and are chosen to
% stay in M4 (equivalently y - A x mod 1 stays < 1/2).
% ==========================================================

% Curve 1 parameters (lower starting point)
\pgfmathsetmacro{\sOne}{\A + 0.08}
\pgfmathsetmacro{\cOne}{0.10}

% Curve 2 parameters (higher starting point)
\pgfmathsetmacro{\sTwo}{\A - 0.06}
\pgfmathsetmacro{\cTwo}{0.28}

% ----------------------------------------------------------
% Red initial segment \tilde U_1 on the lower curve:
% from x=0 until it reaches the same y-level as the other
% curve's starting point y = cTwo.
% ----------------------------------------------------------
\pgfmathsetmacro{\xUredEnd}{(\cTwo-\cOne)/\sOne}
\coordinate (UredA) at (0,\cOne);
\coordinate (UredB) at (\xUredEnd,\cTwo);

% ----------------------------------------------------------
% End y-values at x=1/2 (torus y-coordinate = fractional part)
% ----------------------------------------------------------
\pgfmathsetmacro{\yOneEndRaw}{0.5*\sOne + \cOne}
\pgfmathsetmacro{\yOneEnd}{\yOneEndRaw - floor(\yOneEndRaw)}

\pgfmathsetmacro{\yTwoEndRaw}{0.5*\sTwo + \cTwo}
\pgfmathsetmacro{\yTwoEnd}{\yTwoEndRaw - floor(\yTwoEndRaw)}
\pgfmathsetmacro{\nTwoEnd}{floor(\yTwoEndRaw)} % branch index near x=1/2

% ----------------------------------------------------------
% Green final segment \tilde U_2 on the curve with higher endpoint
% (with these parameters, curve 2 ends higher than curve 1):
% take the terminal piece of curve 2 from the point where it reaches
% y = yOneEnd up to x = 1/2.
% ----------------------------------------------------------
\pgfmathsetmacro{\xUgreenStart}{(\nTwoEnd + \yOneEnd - \cTwo)/\sTwo}
\coordinate (UgreenA) at (\xUgreenStart,\yOneEnd);
\coordinate (UgreenB) at (0.5,\yTwoEnd);

% ==========================================================
% Draw the two wrapped curves (black/gray)
% ==========================================================

% Curve 1: y = frac(sOne*x + cOne), x in [0,1/2]
\begin{scope}
  \clip (0,0) rectangle (0.5,1);
  \foreach \n in {0,...,7} {
    \pgfmathsetmacro{\xa}{max(0,(\n-\cOne)/\sOne)}
    \pgfmathsetmacro{\xb}{min(0.5,(\n+1-\cOne)/\sOne)}
    \ifdim \xb pt>\xa pt
      \draw[line width=0.9pt, black]
        plot[domain=\xa:\xb, samples=2] (\x,{\sOne*\x + \cOne - \n});
    \fi
  }
\end{scope}

% Curve 2: y = frac(sTwo*x + cTwo), x in [0,1/2]
\begin{scope}
  \clip (0,0) rectangle (0.5,1);
  \foreach \n in {0,...,7} {
    \pgfmathsetmacro{\xa}{max(0,(\n-\cTwo)/\sTwo)}
    \pgfmathsetmacro{\xb}{min(0.5,(\n+1-\cTwo)/\sTwo)}
    \ifdim \xb pt>\xa pt
      \draw[line width=0.9pt, black!70]
        plot[domain=\xa:\xb, samples=2] (\x,{\sTwo*\x + \cTwo - \n});
    \fi
  }
\end{scope}

% ==========================================================
% Highlighted pieces
% ==========================================================

% Red initial piece \tilde U_1 on curve 1
\draw[line width=2.0pt, red!85!black]
  plot[domain=0:\xUredEnd, samples=2] (\x,{\sOne*\x + \cOne});

% Green final piece \tilde U_2 on curve 2
\draw[line width=2.0pt, green!60!black]
  plot[domain=\xUgreenStart:0.5, samples=2] (\x,{\sTwo*\x + \cTwo - \nTwoEnd});

% Endpoints / markers
\fill (0,\cOne) circle (0.35pt);
\fill (0,\cTwo) circle (0.35pt);
\fill (0.5,\yOneEnd) circle (0.35pt);
\fill (0.5,\yTwoEnd) circle (0.35pt);

\fill (UredA) circle (0.40pt);
\fill (UredB) circle (0.40pt);
\fill (UgreenA) circle (0.40pt);
\fill (UgreenB) circle (0.40pt);

% Labels for highlighted pieces
\node[fill=white, inner sep=0.8pt, font=\scriptsize, text=red!85!black]
  at (0.055,0.235) {$\tilde U_1$};

\node[fill=white, inner sep=0.8pt, font=\scriptsize, text=green!50!black]
  at (0.445,0.255) {$\tilde U_2$};

% ==========================================================
% Frame and guides
% ==========================================================
\draw[thick] (0,0) rectangle (1,1);
\draw[thick,dashed] (0.5,0) -- (0.5,1); % x = 1/2

% Axes labels
\node[below] at (0.5,0) {$x$};
\node[left]  at (0,0.5) {$y$};

% Tick labels
\node[below left]  at (0,0) {\scriptsize $0$};
\node[below right] at (1,0) {\scriptsize $1$};
\node[above left]  at (0,1) {\scriptsize $1$};
\node[below]       at (0.5,0) [yshift=-10pt] {\scriptsize $\frac12$};

% ==========================================================
% Legend
% ==========================================================
\begin{scope}[shift={(1.06,0.86)}, scale=0.95]
  \fill[Mone] (0,0) rectangle (0.09,0.09);
  \draw (0,0) rectangle (0.09,0.09);
  \node[right] at (0.12,0.045) {$\mathcal M_1$};

  \fill[Mtwo] (0,-0.14) rectangle (0.09,-0.05);
  \draw (0,-0.14) rectangle (0.09,-0.05);
  \node[right] at (0.12,-0.095) {$\mathcal M_2$};

  \fill[Mthree] (0,-0.28) rectangle (0.09,-0.19);
  \draw (0,-0.28) rectangle (0.09,-0.19);
  \node[right] at (0.12,-0.235) {$\mathcal M_3$};

  \fill[Mfour] (0,-0.42) rectangle (0.09,-0.33);
  \draw (0,-0.42) rectangle (0.09,-0.33);
  \node[right] at (0.12,-0.375) {$\mathcal M_4$};
\end{scope}

\end{tikzpicture}
\caption{We represent $\XX_\alpha^{-1} (W_1)$ and $\XX_\alpha^{-1} (W_2)$ as the two black lines in the region $\mathcal{M}_4$. In the proof these lines are cut in $\mathcal{O}(\alpha)$ curves of length $1$.  The red segment $\tilde U_1$ is the initial unmatched curve of $\XX_\alpha ^{-1} (W_1)$, starting from $x=0$ until it reaches the same $y$-level as the  curve $\XX_\alpha ^{-1} (W_2)$. The green segment $\tilde U_2$ is the final unmatched curve of $\XX_\alpha ^{-1} (W_2)$.}
\label{fig:App4}
\end{figure}

By reparametrizing the time $\mathcal{O} (\alpha^2) t$ so that this curve has tangent vector of norm $1$ we get that $\tilde U_1$ is an admissible curve with length $|\tilde U_1|\leq C \alpha^{-1} \dd_{\Sigma} (W_{1,1}, W_{2,1})$. We now define the new base points $ \tilde{X}_1= \XX_\alpha^{-1} (\gamma_1(t^*))$. Since the second component of $\tilde{X}_1$ is the same of the second component of  $ X_2$, their distance is just given by the distance of the first component distance, i.e.
$$ |\tilde X_1 - X_2| = |t^* (a_1 + \alpha b_1)| \leq C \alpha^{-2} \dd_{\Sigma} (W_{1,1}, W_{2,1})\,.$$
We denote  by $\mathbf{v}_1 = (a_1, b_1)$ and $\mathbf{v}_2 = (a_2, b_2)$.
The tangent vectors of $\XX_\alpha^{-1} (\gamma_1(t))$ and $\XX_\alpha^{-1}(\gamma_2(t))$ will be denoted by 
$$ 
\www_1= {\XX_\alpha^{-1} (\mathbf{v}_1)}, \qquad \www_2= {\XX_\alpha^{-1} (\mathbf{v}_2)},
$$
and the matched curves are defined as
$$ 
\tilde W_{1,1,,k} = \left\{\tilde{X}_1 + t \frac{\mathbf{w}_1}{ |\mathbf{w}_1|} : t \in [k, k+1]\right\}, \qquad \tilde W_{2,1,k} = \left\{{X}_2 + t \frac{\mathbf{w}_2}{|\mathbf{w}_2|} : t \in [k, k+1]\right\},
$$ 
for any $k$ until they reach the end of the curve $\XX_\alpha^{-1} (\gamma_1(t))$ and $\XX_\alpha^{-1} (\gamma_2(t))$. Note that, despite these curves being of length $1$, they are well-defined objects in $\Sigma$. Indeed, a computation shows that if $W \in \Sigma$, then $\XX_\alpha^{-1}W$ is never a perfectly vertical curve. Note also that there is an unmatched curve $\tilde U_2$ for the final part of the curve that satisfies the same properties as $\tilde U_1$. We address this curve later.

We now compute the distance of the tangent vectors. We claim that 
\begin{align} \label{eq:distance-vector}
   \left|\frac{\mathbf{w}_1}{|\mathbf{w}_1|}- \frac{\mathbf{w}_2}{|\mathbf{w}_2|}\right|=  \left |  \frac{\XX_\alpha^{-1} (\mathbf{v}_1)}{|\XX_\alpha^{-1} (\mathbf{v}_1)|} - \frac{\XX_\alpha^{-1} (\mathbf{v}_2)}{|\XX_\alpha^{-1} (\mathbf{v}_2)|} \right | \leq C \alpha^{-4} |\mathbf{v}_1 - \mathbf{v}_2| \,.
\end{align}
We simply rewrite
\[
A_\alpha \coloneqq 
\begin{pmatrix}
1 & \alpha\\
\alpha & 1+\alpha^2
\end{pmatrix},
\qquad 
\XX_\alpha^{-1}(z)=A_\alpha z,
\qquad 
\hat z \coloneqq \frac{z}{|z|} \qquad (z\neq 0) \,.
\]
Recall that $\alpha\ge 1$ and that $\mathbf{v}_1=(a_1, b_1)$, $\mathbf{v}_2=(a_2, b_2)$ satisfy
\[
|\mathbf{v}_1|=|\mathbf{v}_2|=1,
\qquad 
|a_1|\le 2\alpha^{-1}|b_1|,
\qquad 
|a_2|\le 2\alpha^{-1}|b_2|.
\]
We observe that since the angle $\theta$ between $\mathbf{w}_1, \mathbf{w}_2$ satisfies $\theta \in [- \pi/2 , \pi/2]$ we deduce
\[
\left|\hat{\mathbf{w}}_1-\hat{\mathbf{w}}_2 \right|^2
=2(1-\cos\theta)\le 2\sin^2\theta,
\]
hence
\[
\left|\hat{\mathbf{w}}_1-\hat{\mathbf{w}}_2 \right|
\le \sqrt2\,|\sin\theta|
=\sqrt2\,\frac{|\det(\mathbf{w}_1,\mathbf{w}_2)|}{|\mathbf{w}_1|\,|\mathbf{w}_2|}.
\]
Here $(\mathbf{w}_1,\mathbf{w}_2)$ denotes the matrix with colums given by $\mathbf{w}_1,\mathbf{w}_2$ respectively.
Since $\det(A_\alpha)=1$,
\[
\det(\mathbf{w}_1,\mathbf{w}_2)=\det(A_\alpha\mathbf{v}_1,A_\alpha\mathbf{v}_2)=\det(\mathbf{v}_1,\mathbf{v}_2).
\]
Moreover,
\[
|\det(\mathbf{v}_1,\mathbf{v}_2)|
=|\det(\mathbf{v}_1-\mathbf{v}_2,\mathbf{v}_2)|
\le |\mathbf{v}_1-\mathbf{v}_2|\,|\mathbf{v}_2|
=|\mathbf{v}_1-\mathbf{v}_2|.
\]
Therefore
\[
\left|\hat{\mathbf{w}}_1-\hat{\mathbf{w}}_2\right|
\le \sqrt2\,\frac{|\mathbf{v}_1-\mathbf{v}_2|}{|A_\alpha\mathbf{v}_1|\,|A_\alpha\mathbf{v}_2|}.
\] 
It is straightforward to check that for $\alpha $ sufficiently large 
\[
|A_\alpha\mathbf{v}_1|\,|A_\alpha\mathbf{v}_2|\ge \frac{\alpha^4}{2}\,,
\]
hence the claim follows.

We consider the distances of base points. We have already shown that the distance of the base points is bounded by $|\tilde{X}_1 - X_2|\leq C \alpha^{-2} \dd_\Sigma (W_{1,1}, W_{2,1})$. Furthermore, since $|D\XX_\alpha^{-1} |\leq C\alpha^2$ and the original curves are of length bounded by $ \frac{2}{\alpha}$ we deduce that $\XX_{\alpha}^{-1} (W_{1,1}), \XX_{\alpha}^{-1} (W_{2,1})$ are made of at most $\mathcal{O}(\alpha)$ admissible curves. By \eqref{eq:distance-vector} we deduce that base points on each matched admissible curves can separate at most by $C {\alpha^{-4}}\dd_\Sigma (W_{1,1}, W_{2,1})$. Therefore, all matched base points are always distant at most $\leq  C \alpha^{-2} \dd_\Sigma (W_{1,1}, W_{2,1})$.

We now consider the last unmatched curve. Recalling the notation for the curves $\tilde W_{1,1,k}$ and $\tilde W_{2,1,k}$, their lengths are always $1$, until we arrive at the final matched curve of length $1$, which terminates at $t=N \in \NN$ with $N \leq C\alpha$. Following this, we construct another set of matched curves by $\tilde W_{1,1,fin}=\{\tilde X_1+N\hat \www_1 +t \hat \www_1: t \leq |\mathbf{w}_1 | t_1 \wedge |\mathbf{w}_2| t_2-N\}$, $\tilde W_{2,1,fin}=\{\tilde X_2+N\hat \www_2 +t \hat \www_2: t \leq |\mathbf{w}_1 | t_1 \wedge |\mathbf{w}_2| t_2-N\}$. Recalling that 
$t_i=\frac{\bar q-q}{a_i+\alpha b_i}$ for $i=1,2$ we deduce that for $i=1,2$ we have 
\begin{align*}
     t_i^2 |DT_\alpha^{-1} \mathbf{v}_i|^2 & = t_i^2 \left ( (a_i + \alpha b_i)^2 (1 + \alpha^2) + 2 \alpha (a_i + \alpha b_i) b_i + b_i^2 \right )
     \\
     & = (\bar q - q)^2 \left ( (1+ \alpha^2) + \frac{2  \alpha  b_i}{(a_i + \alpha b_i)} +  \frac{b_i^2}{(a_2 + \alpha b_2)^2} \right )
\end{align*}
thus, defining $R_i = \frac{2 \alpha b_i}{(1+ \alpha^2)(a_i + \alpha b_i)} + \frac{b_i^2}{(1+ \alpha^2)(a_i + \alpha b_i)^2}$ for $i=1,2$ we deduce
\begin{align*}
     |t_1 |DT_\alpha^{-1} \mathbf{v}_1| - t_2 |DT_\alpha^{-1} \mathbf{v}_2| |
     \leq |\bar q - q| \sqrt{1+ \alpha^2} \left| \sqrt{1 + R_1} - \sqrt{1 + R_2}   \right |  \leq C \alpha |R_1 - R_2 |\,.
\end{align*}
Using that 
$$ |R_1 - R_2| \leq C \alpha^{-2} \dd_\Sigma (W_1, W_2)$$
we conclude that the final unmatched curve defined by (assuming for instance that $|\mathbf{w}_2|t_2\geq |\mathbf{w}_1|t_1$)
$$ 
\tilde U_2 = \left\{ \XX_\alpha^{-1} (\gamma_2 (t)): t\in [|\mathbf{w}_1|t_1, |\mathbf{w}_2|t_2] \right\}
$$
satisfies $|\tilde U_2| \leq C \alpha^{-1} \dd_\Sigma (W_1, W_2)$. Hence, we conclude the proof.
\end{proof}

\section*{Acknowledgments}
MS and DV thank Sam Punshon--Smith, Kyle Liss and Carlangelo Liverani for fruitful discussions on anisotropic Banach spaces. The research of MCZ was partially supported by the Royal Society URF\textbackslash R1\textbackslash 191492 and the ERC/EPSRC Horizon Europe Guarantee EP/X020886/1. MS acknowledges support from the Chapman Fellowship at Imperial College London. The research of DV was funded by the Imperial College President’s PhD Scholarships.

 \bibliographystyle{abbrv}
 \bibliography{DynamoBiblio.bib}

\end{document}